\documentclass{amsart}
\usepackage{lmodern}
\usepackage[english]{babel}
\usepackage{amsmath}
\usepackage{amssymb}
\usepackage{mathptmx} 
\usepackage{color}
\usepackage{graphicx}
\usepackage{caption}
\usepackage{subcaption}
\usepackage{float}
\usepackage[all,cmtip]{xy}
\usepackage{bm}
\usepackage{enumerate}
\usepackage{pinlabel}
\usepackage{enumitem}
\newtheorem{theorem}{\bf Theorem}[section]
\newtheorem{lemma}[theorem]{\bf Lemma}
\newtheorem{definition}[theorem]{\bf Definition}
\newtheorem{corollary}[theorem]{\bf Corollary}
\newtheorem{proposition}[theorem]{\bf Proposition}
\newtheorem{remark}[theorem]{\bf Remark}
\newtheorem{conjecture}[theorem]{\bf Conjecture}

\newcommand{\rme}{\mathrm{e}}
\newcommand{\rmi}{\mathrm{i}}
\newcommand{\rmd}{\mathrm{d}}

\newcommand{\defeq}{\mathrel{\mathop:}=}

\begin{document}

\title{Braided open book decompositions in $S^3$}

\author{
Benjamin Bode}
\date{}

\address{Instituto de Ciencias Matemáticas, Consejo Superior de Investigaciones Científicas, Campus Cantoblanco UAM, C/ Nicolás Cabrera, 13-15, 28049 Madrid, Spain}
\email{benjamin.bode@icmat.es}




\maketitle
\begin{abstract}
We study four (a priori) different ways in which an open book decomposition of the 3-sphere can be defined to be braided. These include generalised exchangeability defined by Morton and Rampichini and mutual braiding defined by Rudolph, which were shown to be equivalent by Rampichini, as well as P-fiberedness and a property related to simple branched covers of $S^3$ inspired by work of Montesinos and Morton. We prove that these four notions of a braided open book are actually all equivalent to each other. We show that all open books in the 3-sphere whose binding has a braid index of at most 3 can be braided in this sense. We relate our findings to a conjecture on real algebraic links by Benedetti and Shiota and to a stronger version of Harer's conjecture due to Montesinos and Morton.
\end{abstract}
%



\section{Introduction}\label{sec:intro}
It is one of the most fundamental results in knot theory, proved by Alexander in 1923, that every link $L$ in the 3-sphere $S^3$ is the closure of some braid \cite{alexander}. In particular, the binding of any open book decomposition of $S^3$ can be braided. This article is concerned with different ways in which an open book decomposition in $S^3$ (and not only its binding) can be braided. For all of these, the binding of the open book is a braid in the usual sense, but there are some additional requirements regarding the relation between the pages of the open book and the braid axis.

We denote an open book decomposition of $S^3$ by $(L,\Psi)$, where $L$ is a link in $S^3$ and $\Psi:S^3\backslash L\to S^1$ is a fibration map over the circle with a specified behaviour on a tubular neighbourhood of $L$. The fibers of $\Psi$, the \textit{pages} of the open book, are Seifert surfaces of $L$, the \textit{binding} of the open book. One particular open book is the \textit{unbook}, whose pages are disks $D$ and whose binding is the unknot $O$. Basic information on open books is reviewed in Section \ref{sec:defs}.


The following four notions all capture some amount of what it could mean for an open book $(L,\Psi)$ to be braided. Detailed definitions are given in Section \ref{sec:defs}.
\begin{enumerate}[label=B\arabic*)]
\item The binding $L$ is the closure of a generalised exchangeable braid (cf. Definition \ref{def:exch}).
\item The open book $(L,\Psi)$ and the unbook $(O,\Phi)$ are mutually braided (cf. Definition \ref{def:mut}).
\item The binding $L$ and a braid axis $O$ of $L$ arise as the preimages of the two components of a \textit{Hopf $n$-braid axis} of the branch link of a simple branched cover $\pi:S^3\to S^3$ (cf. Definition \ref{def:hopf}). 
\item The binding $L$ is the closure of a P-fibered braid (cf. Definition \ref{def:pfib}).
\end{enumerate}

The first type of braiding B1) is due to Morton and Rampichini \cite{mortramp, rampi}, generalising exchangeable braids as proposed by Goldsmith \cite{goldsmith} (cf. also \cite{mortonex} by Morton). It means that $L$ is positively transverse to the pages of the unbook and the binding of the unbook is positively transverse to the pages of $(L,\Psi)$. Mutually braided open book decompositions as in B2) were introduced and studied by Rudolph \cite{rudolphmut}. For this type of braiding it is required that every page of $(L,\Psi)$ is a \textit{braided surface} (with respect to the unbook), in the sense of Rudolph, and every page of the unbook is a generalised braided surface with respect to $(L,\Psi)$. B3) is inspired by work of Montesinos and Morton on simple branched covers of $S^3$ \cite{morton}, but has not appeared in this form so far. Here the binding $L$ and a braid axis $O$ should be preimages $\pi^{-1}(\alpha)$ and $\pi^{-1}(\beta)$ of the two components $\alpha$ and $\beta$ of a Hopf link, each of which forms a braid axis for the branch link of a simple branched cover $\pi:S^3\to S^3$. The term \textit{P-fibered braid} in B4) was introduced in \cite{bodesat}, but the concept of B4) has already appeared in earlier work in the context of links of isolated critical points of real polynomial maps \cite{bode:real} and loops in the space of monic complex polynomials \cite{bode:adicact}. We associate to every braid a loop in the space of monic polynomials. A braid is defined to be P-fibered if the argument of the polynomial defines a fibration over the circle leading to an explicit open book.

The main result of this paper is that these different characterisations are actually all equivalent in the following sense.
\begin{theorem}\label{thm:main}
If an open book $(L,\Psi)$ in $S^3$ satisfies one of the above properties B1) through  B4), say B$i)$, then for any of the above properties B1) through B4), say B$j)$, $(L,\Psi)$ can be isotoped so that it satisfies property B$j)$. We call an open book in $S^3$ \textbf{braided} if it satisfies any (and hence all) of the properties B1) through B4).
\end{theorem}

The implication B$2)\implies$ B1) follows directly from their definitions, while the converse B$1)\implies$ B2) has been proved by Rampichini.


The central aspect of the definition of a braided open book $(L,\Psi)$ is how a page $F_\varphi\defeq\Psi^{-1}(\varphi)$, $\varphi\in S^1$, lies in $S^3$ relative to the unbook and, conversely, how a page of the unbook lies in $S^3$ relative to the pages $F_\varphi$, $\varphi\in S^1$. This is strongly related to braid foliations, going back to Bennequin \cite{benn} and Birman-Menasco \cite{birmena, birmena2}, and the more general open book foliations, which were introduced by Ito and Kawamuro \cite{obf}. The relevant definitions are reviewed in Section \ref{sec:defs}.

In this article we only deal with open book decompositions of the 3-sphere. Properties B1) through B3) have natural generalisations to more general 3-manifolds, with the (generalised) braid axis $O$ no longer an unknot but the binding of another open book $(L',\Psi')$. However, the results in this direction are extremely scarce.

It remains an open problem if any open book in $S^3$ can be braided. A positive answer to this question would not only address the corresponding question for each of the four braid properties, but also prove a conjecture by Montesinos and Morton \cite{morton}, which would imply a new proof of Harer's conjecture. It could also play a big role in proving a conjecture by Benedetti and Shiota \cite{benedetti} on real algebraic links. The connections between braided open books and these conjectures are explained in Section \ref{sec:defs}. 

While we are not able to provide a proof for all fibered links, we prove that if the braid index of the binding is at most 3, the corresponding open book can be braided, thereby proving these conjectures for this simple family of links.

\begin{theorem}
\label{thm:3}
Let $(L,\Psi)$ be an open book in $S^3$ whose binding $L$ has a braid index of at most 3. Then $L$ is the closure of a generalised exchangeable braid.
\end{theorem}

Together with Theorem \ref{thm:main} this immediately implies the following corollary.

\begin{corollary}
Let $(L,\Psi)$ be an open book in $S^3$ whose binding $L$ has a braid index of at most 3. Then $(L,\Psi)$ can be braided.
\end{corollary}

A \textit{braided Seifert surface} of degree $n$ consists of $n$ parallel disks and some half-twisted bands connecting them. A braided Seifert surface of a link $L$ that has minimal genus among all Seifert surfaces of $L$ is called a \textit{Bennequin surface}. Every braided Seifert surface can be represented by a band word in the band generators $a_{i,j}$, $1\leq i<j\leq n$, which represent a positively half-twisted band between the $i$th and the $j$th disk, and their inverses $a_{i,j}^{-1}$, which represent bands with negative half-twists. More detailed definitions of braided surfaces and the band generators are provided in Section \ref{sec:mutual}.

\begin{theorem}
\label{thm:benn}
Let $F$ be a Bennequin surface of degree $n$, whose boundary is a fibered link $L$. If $F$ is represented by a band word that contains the band generators $a_{1,2}$, $a_{2,3}$, $a_{3,4},\ldots, a_{n-1,n}$ and $a_{1,n}$ (anywhere in the band word, not necessarily consecutively, in any order, with any sign), then the open book with binding $L$ and fiber $F$ can be braided. 
\end{theorem}




The remainder of this paper is structured as follows. Section \ref{sec:defs} gives all the definitions and necessary background on the four properties B1) through B4) and how they relate to each other and to open book foliations. We then prove the implications B4)$\implies$B3) in Section \ref{sec:43}, B3)$\implies$B1) in Section \ref{sec:42} and the implication B2)$\implies$B4) in Section \ref{sec:31}, which together with the implications that are already known prove the equivalence between the properties B1) through B4). Theorem \ref{thm:3} and Theorem \ref{thm:benn} are proved in Section \ref{sec:23}.

\ \\

\textbf{Acknowledgements:} This work was partially supported by the Severo Ochoa Postdoctoral Programme at ICMAT, Grant CEX2019-000904-S funded by MCIN/AEI/ 10.13039/501100011033 and by the European Union's Horizon 2020 research and innovation programme through the Marie Sklodowska-Curie grant agreement 101023017. The author would like to thank James Dix for pointing out errors in an earlier version of this article.

\section{Definitions and background}\label{sec:defs}
In this section we provide the reader with the necessary background and, most importantly, give the definitions of the four braiding properties B1) through B4) from the introduction. Some basic knowledge on knots, links and open books is assumed, but can be found in several standard texts such as \cite{etnyre, rolfsen} if necessary.

\subsection{Open books}

An abstract open book is described by a compact oriented surface $F$ with non-empty boundary and a homeomorphism $h:F\to F$, which is the identity on the boundary. From this description we obtain a 3-manifold 
\begin{equation}
M\defeq F\times[0,2\pi]/ \sim,
\end{equation}
where the equivalence relation $\sim$ identifies $(x,0)$ with $(h(x),2\pi)$ for all $x\in F$ and $(x,\varphi)$ with $(x,0)$ for all $x\in\partial F,\varphi\in[0,2\pi]$.
The 3-manifold $M$ is thus filled by homeomorphic disjoint surfaces $F_\varphi\defeq (F,\varphi)$, $\varphi\in [0,2\pi]$, the \textit{pages}, that are identified at their boundary, the \textit{binding}. Since the pages $F_0$ and $F_{2\pi}$ are glued together using the homeomorphism $h$, we might as well consider the variable $\varphi$ on the second factor $[0,2\pi]$ as a circular variable, i.e., valued in $S^1$. 

Note that the common boundary of all surfaces $F_\varphi$ is a link $L=\partial F_\varphi$ in $M$ with an orientation induced by that of $F_\varphi$. The projection map on the second factor $\Psi:(x,\varphi)\mapsto \varphi$ defines a fibration of $M\backslash L$ over the circle $S^1$. We say that $(L,\Psi)$ is an \textit{open book decomposition} of $M$, often abbreviated to OBD or \textit{open book}. Occasionally, we will be somewhat inexact regarding the distinction between the abstract surface $F$ and its embeddings $F_\varphi$, $\varphi\in S^1$, and refer to both of these as the page of the open book.

The terms \textit{open book}, \textit{pages} and \textit{binding} originate from the way in which the surfaces $F_\varphi=\Psi^{-1}(\varphi)$ meet along their common boundary. In particular, for every component $K$ of $L$ the fibration map $\Psi$ on the tubular neighbourhood $N(K)\cong S^1\times D$ of $K$ can be given by $\Psi(\ell,z)=\arg(z)$, where $\arg$ maps any non-zero complex number $z\in D\subset\mathbb{C}$ to its argument, i.e., $\arg(r\rme^{\rmi \chi})=\chi$, for all positive, real $r$ and $\chi\in\mathbb{R}/2\pi\cong S^1$.

We say that a link $L$ in $S^3$ is \textit{fibered} if it is the binding of an open book decomposition of the 3-sphere $S^3$. In this case the page of the open book is a Seifert surface of $L$ of minimal genus among all Seifert surfaces of $L$.

The unbook $(O,\Phi)$ is an open book decomposition of $S^3$, whose binding is the unknot, whose pages are disks and the surface homeomorphism $h$ is the identity map. Using the description $S^3=\{(u,v)\in\mathbb{C}^2:|u|^2+|v|^2=1\}$ as a subset of $\mathbb{C}^2$, the fibration map can be given explicitly by $\Phi:S^3\backslash O\to S^1$, $\Phi(u,v)=\arg v$, where the unknot $O$ is parametrised as $(\rme^{\rmi \chi},0)$ with $\chi$ going from 0 to $2\pi$.

Not all links are fibered, the knot $5_2$ being the simplest non-fibered knot (in terms of minimal crossing number). There are several characterisations of fibered links that can be used to check if a given link $L$ is fibered. These include properties of the commutator subgroup $[\pi_1(S^3\backslash L),\pi_1(S^3\backslash L)]$ of the link group $\pi_1(S^3\backslash L)$ (if $L$ is a knot) \cite{stallings1}, knot Floer homology (the categorification of the Alexander polynomial, which itself also offers an obstruction to being fibered) \cite{yi} and Gabai's test using sutured manifolds \cite{gabai}.

Another characterisation of fibered links is more constructive in nature. It was conjectured by Harer in 1982 \cite{harer} and eventually proved by Giroux and Goodman \cite{giroux}. Given a fibered link $L$ and a simple path $\gamma:[0,1]\to F$ on its fiber surface $F$ (i.e., the embedded page in $S^3$) with $\gamma\cap L=\gamma(0)\cup\gamma(1)$, we can glue a twisted annulus $H$, whose boundary is a positive or negative Hopf link, depending on the direction of the twisting, to $F$ along $\gamma$ as indicated in Figure \ref{fig:hopf}. We call $H$ a positive or negative Hopf band. This gluing operation is a special type of Murasugi sum, called positive or negative \textit{Hopf plumbing} depending on the sign of the added Hopf band. Note that this operation changes both the binding and the page of the open book. However, the resulting link and surface $F'$ are still a fibered link and its corresponding fiber surface. We call the reverse of this operation, i.e., the removal of a Hopf band from a fiber surface, (positive/negative) deplumbing.

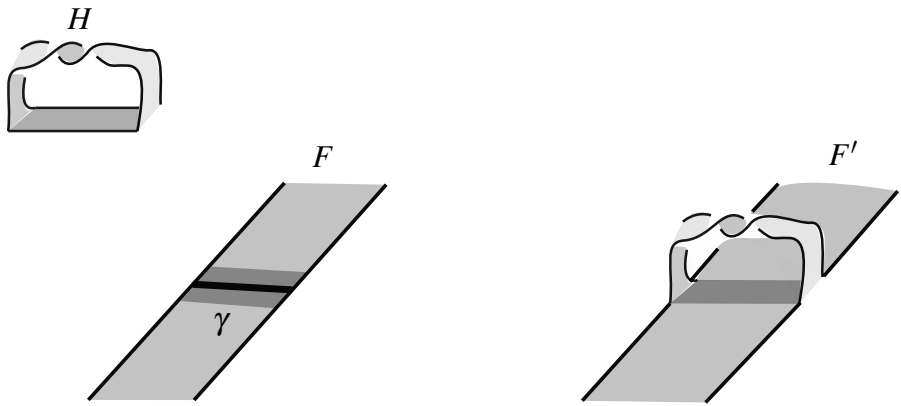
\begin{figure}[h]
\centering
\labellist
\Large
\pinlabel $H$ at 300 1150
\pinlabel $F$ at 950 780
\pinlabel $\gamma$ at 675 330
\pinlabel $F'$ at 2340 790
\endlabellist
\includegraphics[height=6cm]{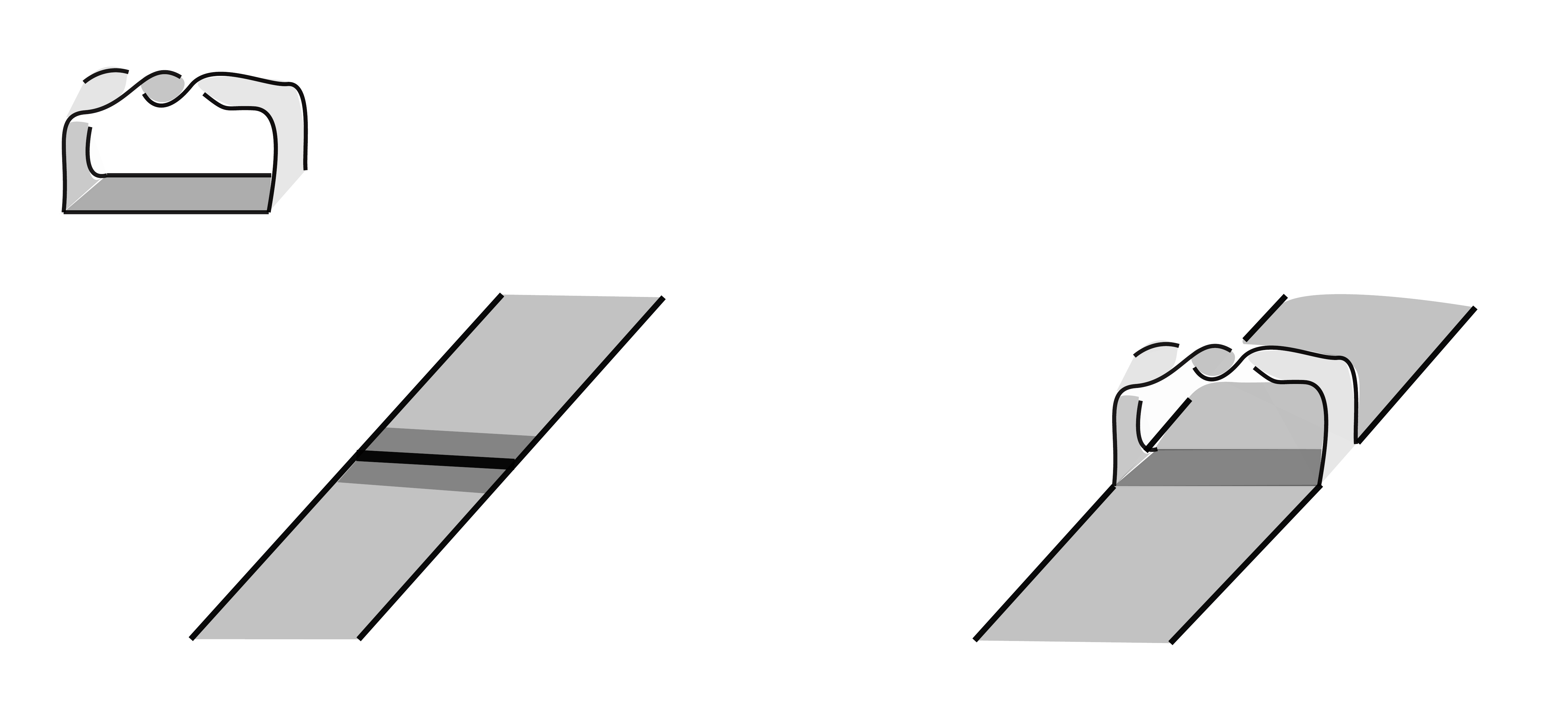}
\caption{A Hopf band $H$ is plumbed to a surface $F$ along a path $\gamma$. \label{fig:hopf}}
\end{figure}

\begin{conjecture}[Harer's conjecture \cite{harer}]
\label{con:harer}
A link $L$ in $S^3$ is fibered if and only if it and its fiber surface can be obtained from the unknot and its fiber disk via a finite sequence of successive Hopf plumbings and deplumbings.
\end{conjecture} 

\begin{theorem}[Giroux-Goodman \cite{giroux}]
Harer's conjecture is true.
\end{theorem}

In fact, Giroux and Goodman's proof shows that for any fibered link the sequence of Hopf plumbings and deplumbings can be taken to be a sequence of Hopf plumbings followed by a (possibly empty) sequence of deplumbings. There are examples of fibered links, where deplumbing is required \cite{melmor}.

The proof relies on the Giroux correspondence between contact structures (up to isotopy) on a 3-manifold $M$ and open book decompositions (up to stabilisation) of $M$ \cite{giroux1}. One of the motivations for this work is to eventually provide a new proof of Harer's conjecture that does not rely on Giroux's correspondence.

\subsection{Open book foliations}

All four braiding properties B1) through B4) from the introduction require the pages of an open book $(L,\Psi)$ to be positioned relative to the unbook in a certain way. Braid foliations were introduced by Bennequin \cite{benn} and Birman and Menasco \cite{birmena, birmena2} and are an important tool to study how a surface is positioned relative to the unbook. 

For the type of mutual braiding that is inherent in the braiding properties B1) through B4) from the introduction we will not only need to study how an open book $(L,\Psi)$ is positioned relative to the unbook, but also how the unbook is positioned relative to the open book $(L,\Psi)$. The concept of open book foliations is a generalisation of braid foliations introduced by Ito and Kawamuro \cite{obf}, which allows us to study the position of a surface relative to the pages of a given open book. For an extensive overview of the subject and its applications we point the reader to \cite{folbook}.



Let $L'$ be a fibered link. We say that an oriented link $L\subset S^3\backslash L'$ is \textbf{braided relative to} $L'$ if it is positively transverse to the pages of the open book $(L',\Psi')$ with binding $L'$. Consequently, each page of $(L',\Psi')$ intersects $L$ in the same number of points, say $n$. We also say that $L$ is a braid on $n$ strands relative to $L'$ or that $L'$ is a \textbf{generalised braid axis} of $L$. Note that if $L'=O$ is the unknot, then $L$ is braided relative to $L'$ if and only if it is in braid form (in the usual sense).

Let $(L',\Psi')$ be an open book in $S^3$ with fibers $S_t=\Psi'^{-1}(t)$, $t\in S^1$, and let $F$ be an oriented, connected, compact surface smoothly embedded in $S^3$ such that its boundary is braided relative to $L'$. Then $\Psi'|_{F}$ induces a singular foliation $\mathcal{F}$ on $F$, which is given by the integral curves of the singular vector field $\{T_pS_t\cap T_p F\}$. A connected component of the integral curves, which corresponds to $S_t\cap F$ for some $t\in S^1$, is called a \textit{leaf}. 

\begin{definition}\label{def:fol}
Let $(L',\Psi')$ be an open book in $S^3$ with fibers $S_t=\Psi'^{-1}(t)$, $t\in S^1$, and let $F$ be an oriented, connected, compact surface smoothly embedded in $S^3$ such that its boundary is braided relative to $L'$. The singular foliation $\mathcal{F}$ on $F$ induced by $\Psi'$ is called an \textbf{open book foliation} if the following conditions are satisfied:
\begin{itemize}
\item The surface $F$ intersects $L'$ transversely in finitely many points, around each of which the foliation is radial.
\item All leaves of $\mathcal{F}$ along $\partial F$ meet $\partial F$ transversely.
\item All but finitely many fibers $S_t$ intersect $F$ transversely. For each fiber $S_t$ that has a tangential intersection point with $F$ this tangential intersection point is unique.
\item All the tangential intersection points of $F$ and fibers $S_t$ are saddles.
\end{itemize}
\end{definition}

Definition \ref{def:fol} is stated for open books in $S^3$, which is the only 3-manifold of relevance for this article. The importance of open book foliations (defined for open books in more general 3-manifolds) in a wider context is that it allows to generalise many tools regarding braids, braided surfaces and foliations to other 3-manifolds. 

For the special case where $(L',\Psi')$ is the unbook the above definition is essentially the definition of a braid foliation. The original definition of braid foliations also allows tangential intersection points that are local extrema, which we have excluded, following Ito and Kawamuro \cite{obf}.

An open book foliation has two types of singular points. The intersections of $F$ with $L'$ are \textit{elliptic} (cf. Figure \ref{fig:singulars}a)). We associate to each elliptic singular point a sign, positive or negative, depending on whether $L'$ is positively or negatively transverse to $F$ at that intersection point. The second type of singular points of the foliation are the tangential intersection points of $F$ with fibers $S_t$. They are \textit{hyperbolic} (cf. Figure \ref{fig:singulars}b)). The sign of a hyperbolic point $p\in F\cap S_t$ is positive if the orientation of the tangent plane $T_pF$ coincides with the orientation of $T_p S_t$ and is negative if the two orientations do not coincide.

\begin{figure}[h]
\labellist
\Large
\pinlabel a) at 100 1150
\pinlabel b) at 1600 1150
\endlabellist
\centering
\includegraphics[height=4cm]{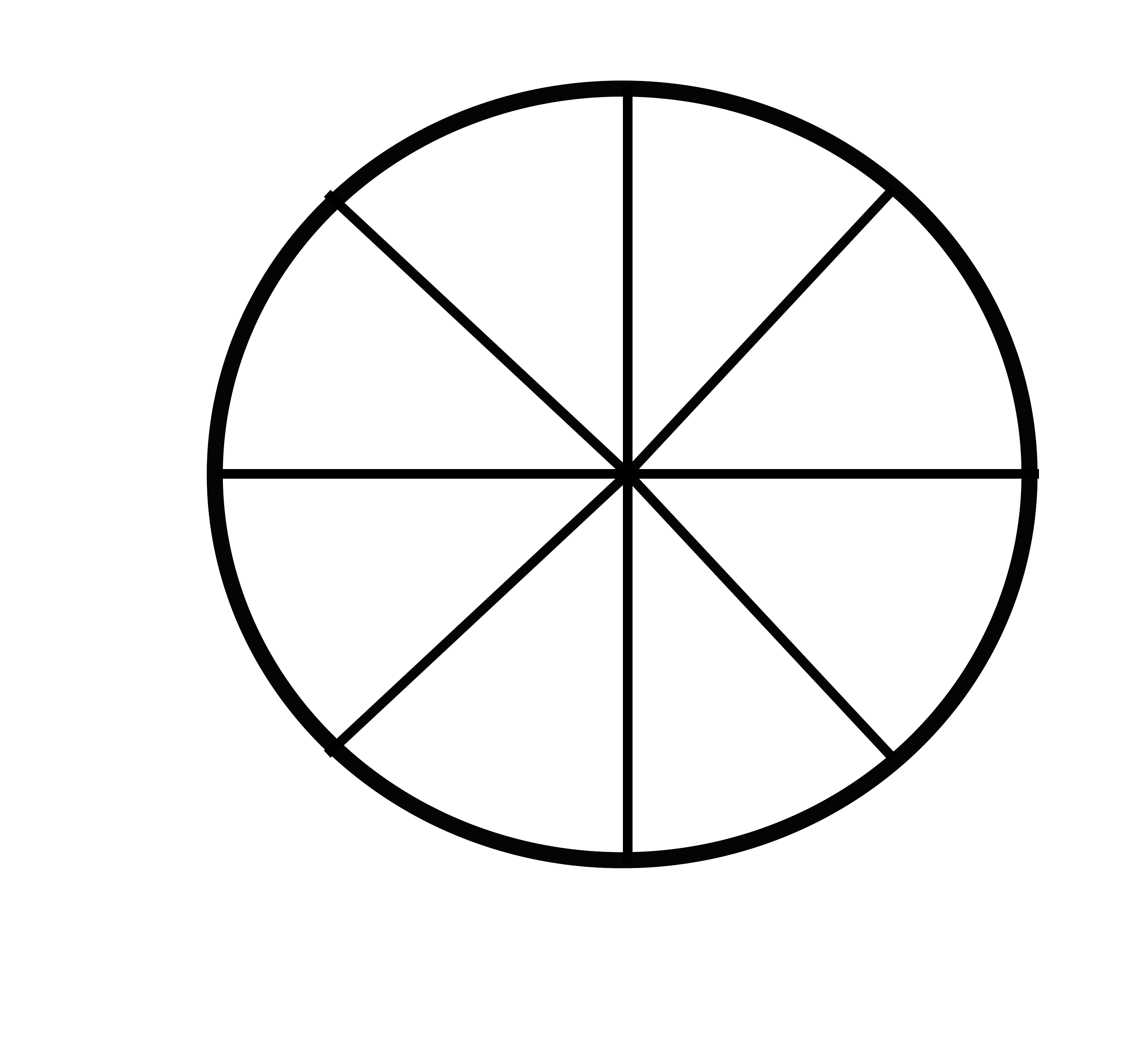}
\includegraphics[height=4cm]{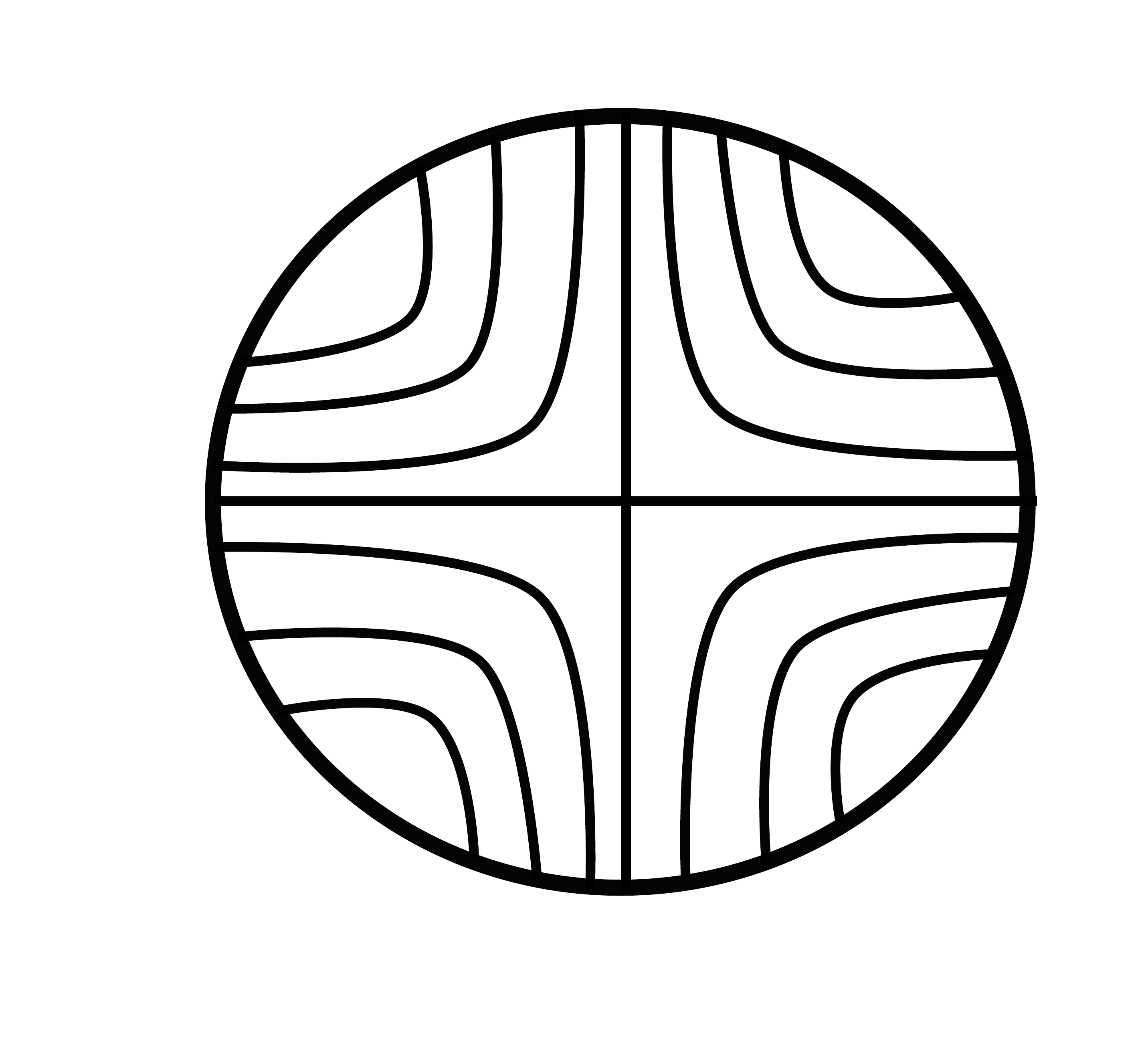}
\caption{Singular points of an open book foliation. a) The neighbourhood of an elliptic singular point. b) The neighbourhood of a hyperbolic singular point. \label{fig:singulars}}
\end{figure}

We say that a leaf is \textit{regular} if it does not contain a singular point.
\begin{proposition}[Ito-Kawamuro \cite{obf}]
Let $\mathcal{F}$ be an open book foliation. Then the regular leaves of $\mathcal{F}$ can be classified into three distinct types:
\begin{enumerate}[label=(\roman*)]
\item a-arc: An arc where one of its endpoints lies on $L'$ and the other lies on $\partial F$.
\item b-arc: An arc where both endpoints lie on $\partial F$.
\item c-arc: A simple closed curve.
\end{enumerate}
\end{proposition}

Consider an open book $(L,\Psi)$ with pages $F_\varphi\defeq\Psi^{-1}(\varphi)$, $\varphi\in S^1$, and the unbook $(O,\Phi)$ with pages $D_t\defeq \Phi^{-1}(t)$, $t\in S^1$. Then the unbook induces a singular foliation on each fiber $F_{\varphi}$ and, conversely, the open book $(L,\Psi)$ induces a singular foliation on each fiber disk $D_t$.

Given a surface $F$ in an open book, we can always isotope $F$ so that the foliation $\mathcal{F}$ induced by the open book satisfies the conditions above and is therefore an open book foliation. For example, the so-called `\textit{finger move}' can be used repeatedly to remove all tangential intersection points between $F$ and the fibers that are not saddles \cite{obf}. In this article however, we are not dealing with an individual Seifert surface $F$, but with all the pages $F_{\varphi}$ of an open book and all pages $D_t$ of the unbook at the same time. A finger move may be used to remove maxima and minima on one page $F_{\varphi}$, but in doing so we might introduce new maxima and minima to a different page $F_{\varphi'}$. In particular, we cannot expect all pages $F_{\varphi}$ and $D_t$ to be equipped with an open book foliation at the same time. However, we will see that we can arrange that all but finitely many pages $F_{\varphi}$ and $D_t$ are equipped with an open book foliation.

\subsection{Generalised exchangeable braids}\label{sec:exch}

Recall that a link $L$ is braided relative to a fibered link $L'$ if it is positively transverse to pages of the open book with binding $L'$.

\begin{definition}\label{def:exch}
Let $L_{ex}=L\cup O$ be a link in $S^3$, where $L$ is fibered and $O$ is the unknot. We say that $L_{ex}$ is \textbf{generalised exchangeable} if $L$ is braided relative to $O$ and vice versa.\\
A braid $B$ is called \textbf{generalised exchangeable} if the union of its closure $L$ and its braid axis $O$ forms a generalised exchangeable link.
\end{definition}

This definition is due to Morton and Rampichini \cite{mortramp, rampi} and generalises the concept of \textit{exchangeable braids} by Goldsmith \cite{goldsmith}, where $L$ is also required to be an unknot. The term `exchangeable braid' is also used by Stoimenow (e.g. \cite{stoi}) to describe a braid on which an `exchange move' (defined by Birman and Menasco \cite{birmena4}) can be performed. This is not related to exchangeability in the sense of Goldsmith and Definition \ref{def:exch}. Another established term that could cause confusion is that of an \textit{interchangeable} link $L=L_1\cup L_2$, for which an ambient isotopy exchanges $L_1$ and $L_2$, i.e., maps $L_1$ to $L_2$ and vice versa (cf. for example \cite{kadok}). Again, this concept is not related to Definition \ref{def:exch}.

By Alexander every link $L$ is ambient isotopic in $S^3$ to a closed braid \cite{alexander}. However, this isotopy often requires to pass $L$ through the eventual braid axis $L'$. Hence, it is not at all clear if every fibered link can be arranged to be the closure of a generalised exchangeable braid. One might start with an unknot $L'=O$ that is braided relative to $L$, for example a meridian of a tubular neighbourhood of $L$, and then (for example via Vogel's algorithm \cite{vogel} or the procedure outlined in Alexander's proof \cite{alexander}) braid $L$ relative to $O$. However, we cannot expect that at the end of this isotopy $O$ is still braided relative to $L$.

Recall that the braid group $\mathbb{B}_n$ on $n$ strands is the group of homotopy classes of motions of $n$ distinct points in the plane. In this article it will have advantages to work with the band generators, also called BKL-generators, $a_{i,j}$, $1\leq i<j\leq n$, shown in Figure \ref{fig:BKL}, instead of the usual Artin generators \cite{bkl}. The BKL-generators are named for Birman, Ko and Lee. Expressed in Artin generators they are
\begin{equation}
a_{i,j}^{\pm 1}=\sigma_i\sigma_{i+1}\ldots\sigma_{j-2}\sigma_{j-1}^{\pm 1}\sigma_{j-2}^{-1}\ldots\sigma_{i+1}^{-1}\sigma_i^{-1}.
\end{equation}

\begin{figure}[h]
\labellist
\Large
\pinlabel a) at 100 950
\pinlabel b) at 1250 950
\endlabellist
\centering
\includegraphics[height=6cm]{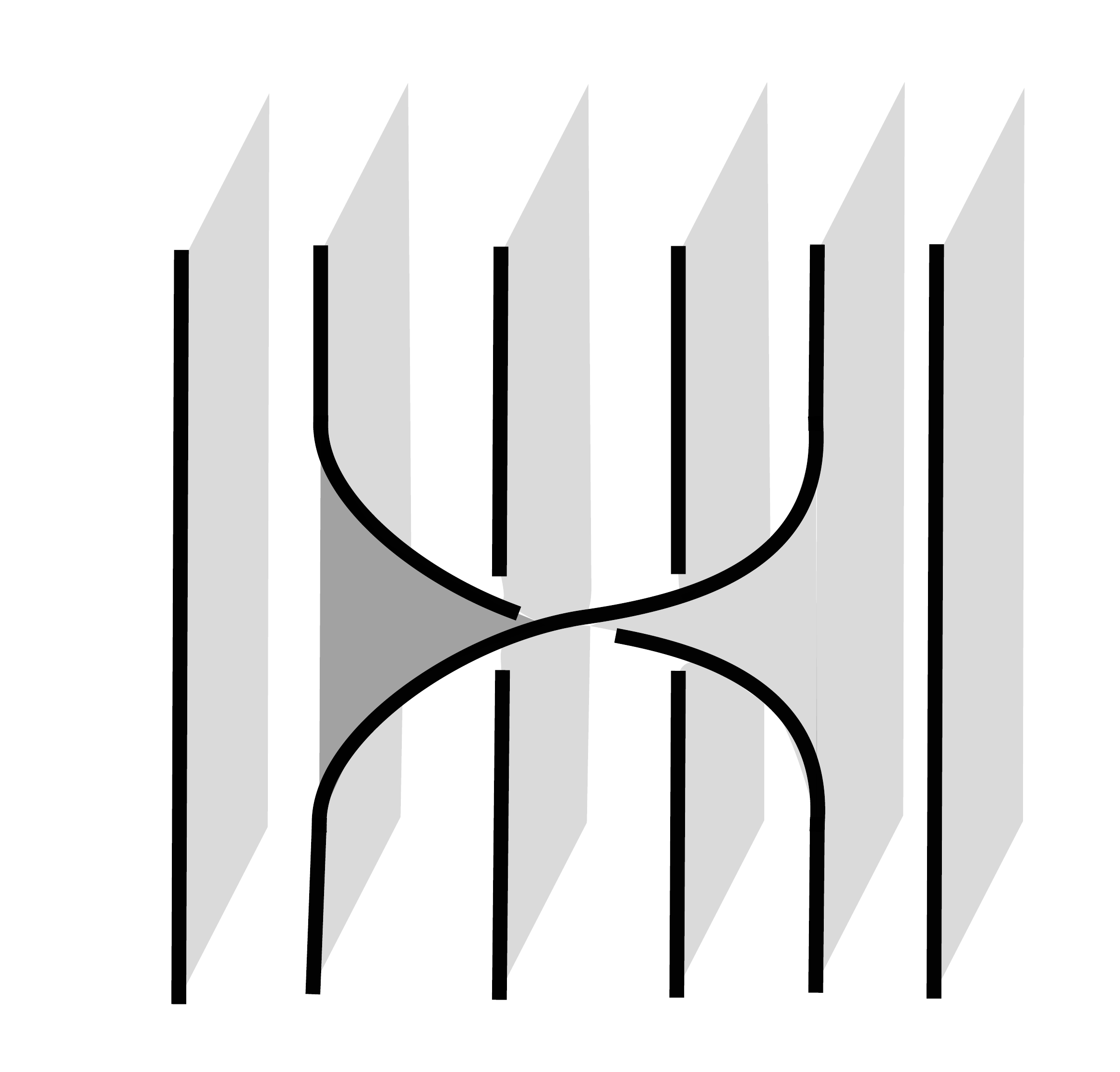}
\includegraphics[height=6cm]{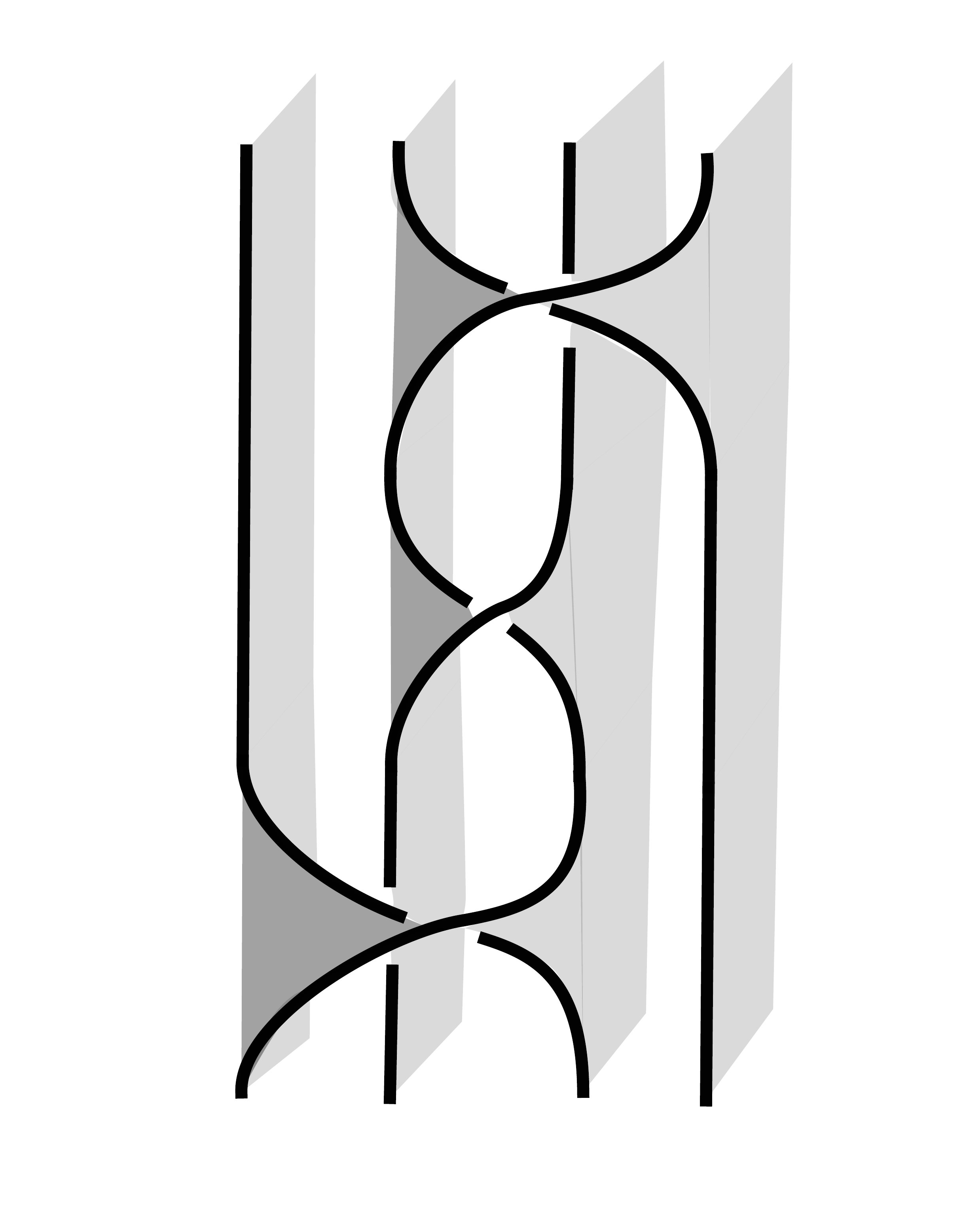}
\caption{a) The band generator $a_{2,5}$. b) A non-exchangeable braid, whose closure is the unknot, given by $a_{1,3}a_{2,3}a_{2,4}$.  \label{fig:BKL}}
\end{figure}

Geometrically, we can think of the band generators as follows. For a braid on $n$ strands we start with $n$ parallel half-planes in $\{z\in\mathbb{C}:\text{Im}(z)\geq 0\}\times[0,2\pi]\subset\mathbb{C}\times[0,2\pi]$, whose boundaries form a trivial braid in $\{z\in\mathbb{C}:\text{Im}(z)= 0\}\times[0,2\pi]$. Going from the bottom of the braid to the top as we go through the braid word from left to right we insert for every band generator $a_{i,j}$ a positively half-twisted band between the $i$th and the $j$th strand in front of all strands between $i$ and $j$. Inverses correspond to negatively half-twisted bands. Which direction of the half-twist is regarded as positive and the choice of reading the braid word from the bottom of the braid to the top are choices of convention, but these choices should be consistent with those made for the Artin generators, so that $a_{i,i+1}a_{i+1,i+2}=\sigma_i\sigma_{i+1}$. After closing the braid as usual the constructed surface consists of $n$ disks connected by a number of half-twisted bands, one for each letter in the band word. Occasionally it will simplify our notation if we allow ourselves to write $a_{j,i}$ with $1\leq i<j\leq n$. The geometric description of the band generators makes it obvious that $a_{j,i}=a_{i,j}$.

In terms of the band generators $a_{i,j}$ the defining relations of the braid group read:
\begin{align}
a_{i,j}a_{k,m}&=a_{k,m}a_{i,j}&\text{ if }(i-k)(i-m)(j-k)(j-m)>0,\label{eq:1rel}\\
a_{i,j}a_{j,k}&=a_{i,k}a_{i,j}=a_{j,k}a_{i,k}&\text{ for all }i,j,k\text{ with }1\leq k<j<i\leq n.\label{eq:2rel}
\end{align}

Note that for $j=i+1$ and $m=k+1$ the relations above become (unsurprisingly) the usual braid relations. More precisely, it follows from Eq.~\eqref{eq:2rel} that
\begin{align}
\sigma_i\sigma_{i+1}\sigma_i&=a_{i,i+1}a_{i+1,i+2}a_{i,i+1}\nonumber\\
&=a_{i+1,i+2}a_{i,i+2}a_{i,i+1}\nonumber\\
&=a_{i+1,i+2}a_{i,i+1}a_{i+1,i+2}=\sigma_{i+1}\sigma_i\sigma_{i+1}
\end{align}
for all $i=1,2,\ldots,n-2$. The condition in the first type of relation (Eq. (\ref{eq:1rel})) is often described as ``the strands/indices do not $\textit{interlace}$''.

Note that changes of band words via a relation of the form $a_{i,j}a_{i,j}^{-1}=a_{i,j}^{-1}a_{i,j}=id$, while not changing the braid, do change the associated banded surface. 

Morton found an example of a braid that closes to the unknot and that is not exchangeable \cite{mortonex}. Neither is any of its conjugates. In band generators it is given by $a_{1,3}a_{2,3}a_{2,4}$. Of course, there are other braids that close to the unknot and that are exchangeable, but Morton's example highlights an important aspect of this topic, namely that not every braid axis of the binding is also a braid axis for the corresponding open book. 

At the moment it is still unknown if every fibered link is the closure of a generalised exchangeable braid.

\subsection{Mutual braiding}
\label{sec:mutual}

We adopt several definitions from Rudolph \cite{rudolphbraid, rudolph2, rudolphmut}.

\begin{definition}\label{def:braidsurf}
Let $(L',\Psi')$ be an open book decomposition of $S^3$. A compact connected oriented surface $F$ smoothly embedded in $S^3$ is called a \textbf{generalised braided surface} of degree $n$ relative to the open book $(L',\Psi')$ if 
\begin{itemize}
\item its boundary $\partial F=L$ is a generalised closed braid on $n$ strands (relative to $L'$),
\item $F$ has $n$ intersection points with the generalised braid axis $L'$, all of which are positive transverse,
\item when restricted to $F$, the fibration map $\Psi':S^3\backslash L'\to S^1$ is Morse without any local extrema.
\end{itemize}
If $(L',\Psi')$ is the unbook, we simply say that $F$ is a \textbf{braided surface}.
\end{definition}

\begin{definition}\label{def:mut}
We say that two open book decompositions $(L,\Psi)$ and $(L',\Psi')$ of $S^3$ are \textbf{mutually braided} if every page $F_t$, $t\in S^1$, is a generalised braided surface relative to $(L',\Psi')$ and every page $F'_t$, $t\in S^1$, is a generalised braided surface relative to $(L,\Psi)$.
\end{definition}

In this article we only study open book decompositions that are braided relative to a braid axis, i.e. relative to an unknot/unbook. In this case, the notion of mutual braiding is equivalent to that of a \textbf{totally braided open book}, also defined by Rudolph \cite{rudolph2}.  




Rudolph showed that Alexander's Theorem can be generalised to surfaces, so that every Seifert surface can be braided (Proposition 1 of \cite{rudolphquasi}). However, it is not clear if all fibers can be braided simultaneously, i.e., whether every open book in $S^3$ can be totally braided.

Let $(L,\Psi)$ and the unbook $(O,\Phi)$ be mutually braided open books. Let $pos(L)$ (respectively, $neg(L)$) denote the set of points where fibers of $\Psi$ and fibers of $\Phi$ intersect tangentially with the same (respectively, opposite) orientation.


\begin{definition}
\label{def:cdb}
Let $(L,\Psi)$ and $(O,\Phi)$ be mutually braided open books. Then $pos(L)\cup neg(L)$ is a braid called the \textbf{derived closed braid} or the \textbf{derived bibraid}.
\end{definition}

Our notation differs slightly from that in \cite{rudolph2}, where $pos(L)$ and $neg(L)$ are defined in terms of tangencies of fibers of $\Phi$ with complex lines in $\mathbb{C}^2$. Overall, the two definitions are identical except that in \cite{rudolph2} $pos(L)$ obtains one extra component $pos_0(L)$ in a tubular neighbourhood of $O$. The \textit{derived closed braid} in \cite{rudolph2} is then $(pos(L)\backslash pos_0(L))\cup neg(L)$, which is the same object as in Definition \ref{def:cdb}.

Note that Definition \ref{def:braidsurf} and Definition \ref{def:mut} guarantee that every point of tangential intersection between fibers of mutually braid open books is a saddle point. This justifies the term `bibraid', since it implies that $pos(L)\cup neg(L)$ is braided relative to both $L$ and $O$. To be precise, we can choose an orientation for the bibraid so that the bibraid is braided relative to $O$. Reversing the orientation on the components of $neg(L)$ then results in a braid relative to $L$.


Let $(L,\Psi)$ be an open book that is mutually braided with an unbook $(O,\Phi)$. Then $\Psi$ induces a singular foliation on each page of the unbook and $\Phi$ induces a singular foliation on each page of $(L,\Psi)$. Using the concepts from Section \ref{sec:exch} Definition \ref{def:mut} can be rephrased as follows. Two open books are mutually braided if and only if their bindings are generalised exchangeable and there are no c-arcs in the corresponding singular foliations. This way it becomes obvious that the bindings of any pair of mutually braided open books are generalised exchangeable and hence property B2) implies property B1).

The converse B1)$\implies$B2) was proved by Rampichini \cite{rampi}.

\begin{theorem}[Rampichini, Theorem 3 in \cite{rampi}]
Let $(L,\Psi)$ be an open book in $S^3$ and let $O$ be an unknot such that $L_{ex}=L\cup O$ is generalised exchangeable. Then (potentially after an isotopy of their pages) $(L,\Psi)$ and the unbook with binding $O$ are mutually braided.
\end{theorem}


For a generalised exchangeable link $L\cup O$ we can isotope the pages of the two open books to remove all local tangential intersection points that are not saddle points and so property B1) is equivalent to B2). One important insight of Rampichini's and Morton's work \cite{mortramp, rampi} is that the interesting information about the position of one open book relative to another is given by the combinatorial structure of the singular leaves of the singular foliation on each fiber. Note that the hyperbolic points of this collection of singular foliations form precisely the derived bibraid defined above. The second important insight is that this information can be encoded in certain diagrams, which can be used to check whether a given open book is totally braided. Several of their concepts and notations will be useful to us later and are explained in the following paragraphs.

\begin{figure}[h]
\labellist
\small
\pinlabel 0 at 150 160
\pinlabel 0 at 210 100
\pinlabel $2\pi$ at 850 100
\pinlabel $2\pi$ at 150 800
\pinlabel $(3\ 4)$ at 270 260
\pinlabel $(1\ 2)$ at 270 460
\pinlabel $(2\ 3)$ at 270 630
\pinlabel $(1\ 2)$ at 410 840
\pinlabel $(2\ 4)$ at 530 840
\pinlabel $(2\ 3)$ at 690 840
\pinlabel $(2\ 4)$ at 415 660
\pinlabel $(3\ 4)$ at 440 465
\pinlabel $(1\ 2)$ at 440 320
\pinlabel  $(4\ 1)$ at 550 670
\pinlabel $(2\ 3)$ at 495 570
\pinlabel $(2\ 4)$ at 630 360
\pinlabel $(2\ 3)$ at 720 230
\pinlabel $(1\ 2)$ at 900 590
\pinlabel $(4\ 1)$ at 900 440
\pinlabel $(2\ 3)$ at 900 260
\pinlabel $(4\ 1)$ at 610 575
\Large
\pinlabel $t$ at 100 500
\pinlabel $\varphi$ at 500 50
\endlabellist
\centering
\includegraphics[height=7cm]{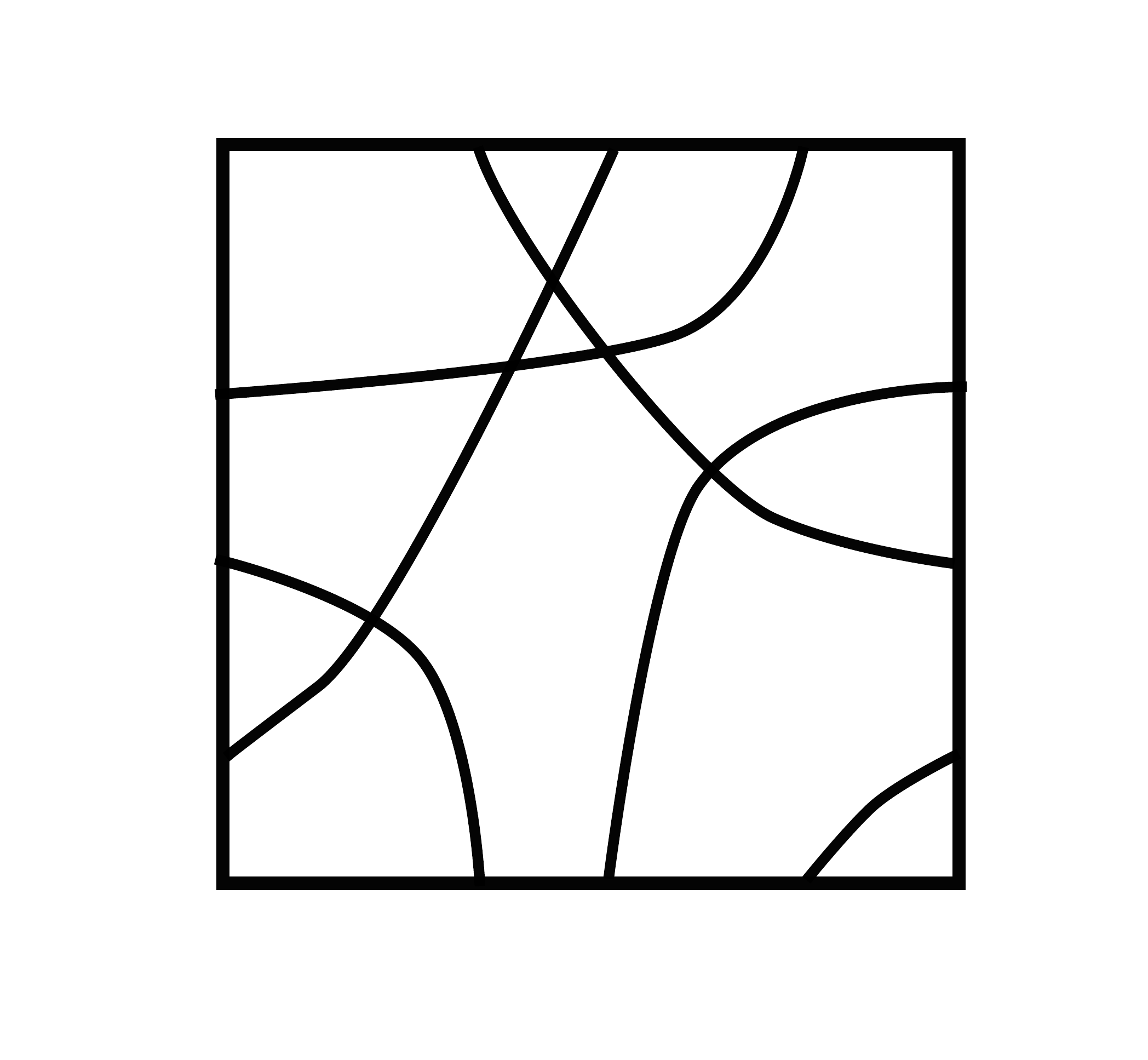}
\caption{A Rampichini diagram. The transpositions at $t=0$ are $\tau_1=(1\ 2)$, $\tau_2=(2\ 4)$ and $\tau_3=(2\ 3)$. A band word for the fiber $F_{\varphi=0}$ is given by $a_{3,4}a_{1,2}^{-1}a_{2,3}$. \label{fig:rampichini}}
\end{figure}
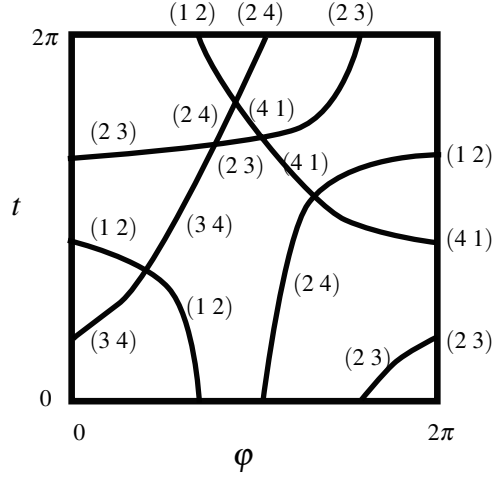

\begin{definition}
\label{def:rampichini}
A \textbf{Rampichini diagram} of degree $n$ is a square, whose horizontal and vertical edges are coordinate axes representing values of variables $\varphi$ and $t$ between 0 and $2\pi$, together with a set of curves $\ell_i$, $i=1,2,\ldots,k$, in the square and a set of transpositions in $S_n$ labelling the points on the curves such that
\begin{enumerate}[label=\arabic*)]
\item The curves are projections of smooth closed curves on a torus with cyclic variables $\varphi$ and $t$.
\item Every curve $\ell_i=\cup_{s\in [0,2\pi]}(\varphi_i(s),t_i(s))$ is either monotone increasing or monotone decreasing, that is for every $\ell_i$ either $\text{sign}\tfrac{\rmd \varphi_i}{\rmd s}=\text{sign}\tfrac{\rmd t_i}{\rmd s}$ everywhere or $\text{sign}\tfrac{\rmd \varphi_i}{\rmd s}=-\text{sign}\tfrac{\rmd t_i}{\rmd s}$ everywhere.
\item There are exactly $n-1$ intersection points between the curves and the horizontal $t=0$-edge.
\item There are finitely many intersections between curves and all of them are simple (i.e., of multiplicity two) and transverse.
\item Every point on a curve that is not an intersection point is labelled by a transposition.
\item Labels only change at intersection points and at the $\varphi=2\pi$-line. 
\item At intersection points the labels change via BKL relations explained below in Eq. (\ref{eq:tau3}) and Eq. (\ref{eq:tau4}).
\item The labels at the $\varphi=0$-line are the labels at the $\varphi=2\pi$-line plus 1 modulo $n$.
\item Along the $t=0$-edge the transpositions $\tau_i$ (indexed with increasing $\varphi$) satisfy the following three properties:
\begin{itemize}
\item No transposition is repeated, $\tau_i\neq \tau_j$ if $i\neq j$.
\item No pair of transpositions is interlaced, i.e., there are no $i<k<j<m$ such that $(i\ j)$ and $(k\ m)$ appear in the list of permutations.
\item There are no $i<k<j$ such that $(i\ j)$ and $(i\ k)$ appear in the list of permutations in this order.
\end{itemize}
\end{enumerate}
\end{definition}

In \cite{mortramp} and \cite{rampi} Rampichini diagrams are called \textit{labelled graphics}. In condition 9) the orderings of at least 3 elements, such as $i<k<j$, should be understood as a cycling ordering in $\mathbb{Z}/n\mathbb{Z}$. 

Since labels only change at intersection points and at the right edge of a Rampichini diagram, we label each arc of the curves between the intersection points (with other curves or the edges of the diagram). It is understood that this label is the transposition associated to all points on that arc, see Figure \ref{fig:rampichini}.

We now describe the \textit{BKL relations} that determine how the labels in a Rampichini diagrams change at an intersection point between curves. Let $\tau_i(t)$, $i=1,2,\ldots,n-1$, denote the transpositions at height $t$, labelled with increasing $\varphi$, that is, going from left to right in the diagram, and let $\varepsilon_i(t)\in\{\pm 1\}$ denote the sign of the slope of the corresponding curve in the diagram. Then the BKL relation at an intersection point $(\varphi_0,t_0)$ between the $i$th and $(i+1)$th line (again counted from left to right) can be expressed as
\begin{align}
\varepsilon_i(t_0-\delta)&\mapsto\varepsilon_i(t_0+\delta)=\varepsilon_{i+1}(t_0-\delta)\nonumber\\
\varepsilon_{i+1}(t_0-\delta)&\mapsto\varepsilon_{i+1}(t_0+\delta)=\varepsilon_{i}(t_0-\delta)
\end{align}
and
\begin{align}
\label{eq:tau3}
\tau_i(t_0-\delta)&\mapsto\tau_{i}(t_0+\delta)=\tau_{i+1}(t_0-\delta)\nonumber\\
\tau_{i+1}(t_0-\delta)&\mapsto\tau_{i+1}(t_0+\delta)=\tau_{i+1}(t_0-\delta)\tau_{i}(t_0-\delta)\tau_{i+1}(t_0-\delta)^{-1}
\end{align}
or
\begin{align}
\label{eq:tau4}
\tau_i(t_0-\delta)&\mapsto\tau_{i}(t_0+\delta)=\tau_i(t_0-\delta)^{-1}\tau_{i+1}(t_0-\delta)\tau_i(t_0-\delta)\nonumber\\
\tau_{i+1}(t_0-\delta)&\mapsto\tau_{i+1}(t_0+\delta)=\tau_{i}(t_0-\delta).
\end{align}
In all of the equations above $\delta$ is a small positive number. Since all $\tau_i(t)$ are transpositions, we have $\tau_i(t)^{-1}=\tau_i(t)$. The reason for explicitly stating the inverse will become apparent in Section \ref{sec:31}, where Rampichini diagrams are compared to similar combinatorial structures that in principle allow more general permutations. Note that, since $\varepsilon_i(t)$ determines the sign of the slope of the $i$th line, $\varepsilon_i(t_0-\delta)=-1$ implies $\varepsilon_{i+1}(t_0-\delta)=-1$. Otherwise, no intersection between these two lines would be possible at $t=t_0$.


We now describe how a pair of mutually braided open books $(L,\Psi)$ and $(O,\Phi)$ gives rise to a Rampichini diagram. Their bindings, a fibered link $L$ and an unknot $O$, are generalised exchangeable. Consider the page $F_0=\Psi^{-1}(0)$ of the open book $(L,\Psi)$. By definition it intersects $O$ in $n$ distinct points, which we may label from $1$ to $n$ such that going along $O$ (with the orientation of $O$) the point with label $i$ follows the point with label $i-1$ modulo $n$. This splits $O$ into $n$ half-open intervals $[i,i+1)$, whose start and end points are the $n$ intersection points. We label each of these arcs $A_i$ with the same index as its start point.

Now consider a fixed fiber disk $D_{t}=\Phi^{-1}(t)$, $t\in S^1$, and its singular foliation induced by $(L,\Psi)$. We assume that this singular foliation is an open book foliation and in particular, each singular leaf only contains one hyperbolic point. Then a singular leaf $D_{t}\cap F_{\varphi}$ for some $\varphi\in S^1$ of this foliation is cross-shaped, consisting of an arc that connects two elliptic points in $L\cap D_t$ and an arc that connects two points on $\partial D_t$. They intersect in a hyperbolic point. Let $A_i$ and $A_j$ be the two arcs of $O$ that contain the two points of the singular leaf on $\partial D_t$. Then the hyperbolic point is labelled by the transposition $(i\ j)$. Note that $i\neq j$, since every arc on $\partial D_t$ contains exactly one point on $F_{\varphi}$. 


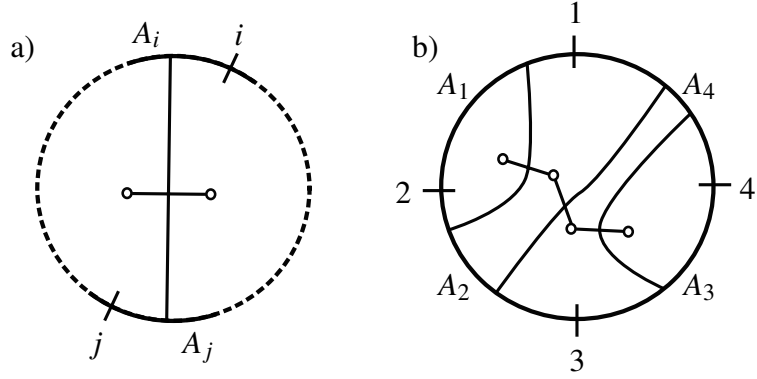
\begin{figure}[h]
\labellist
\Large
\pinlabel a) at 100 900
\pinlabel b) at 1200 900
\pinlabel $A_i$ at 440 940
\pinlabel $A_j$ at 580 85
\pinlabel $i$ at 690 940
\pinlabel $j$ at 300 90
\pinlabel $A_1$ at 1280 800
\pinlabel $A_2$ at 1280 250
\pinlabel $A_3$ at 1950 250
\pinlabel $A_4$ at 1950 800
\pinlabel 1 at 1610 1000
\pinlabel 2 at 1140 510
\pinlabel 3 at 1615 50
\pinlabel 4 at 2080 520
\endlabellist
\centering
\includegraphics[height=5cm]{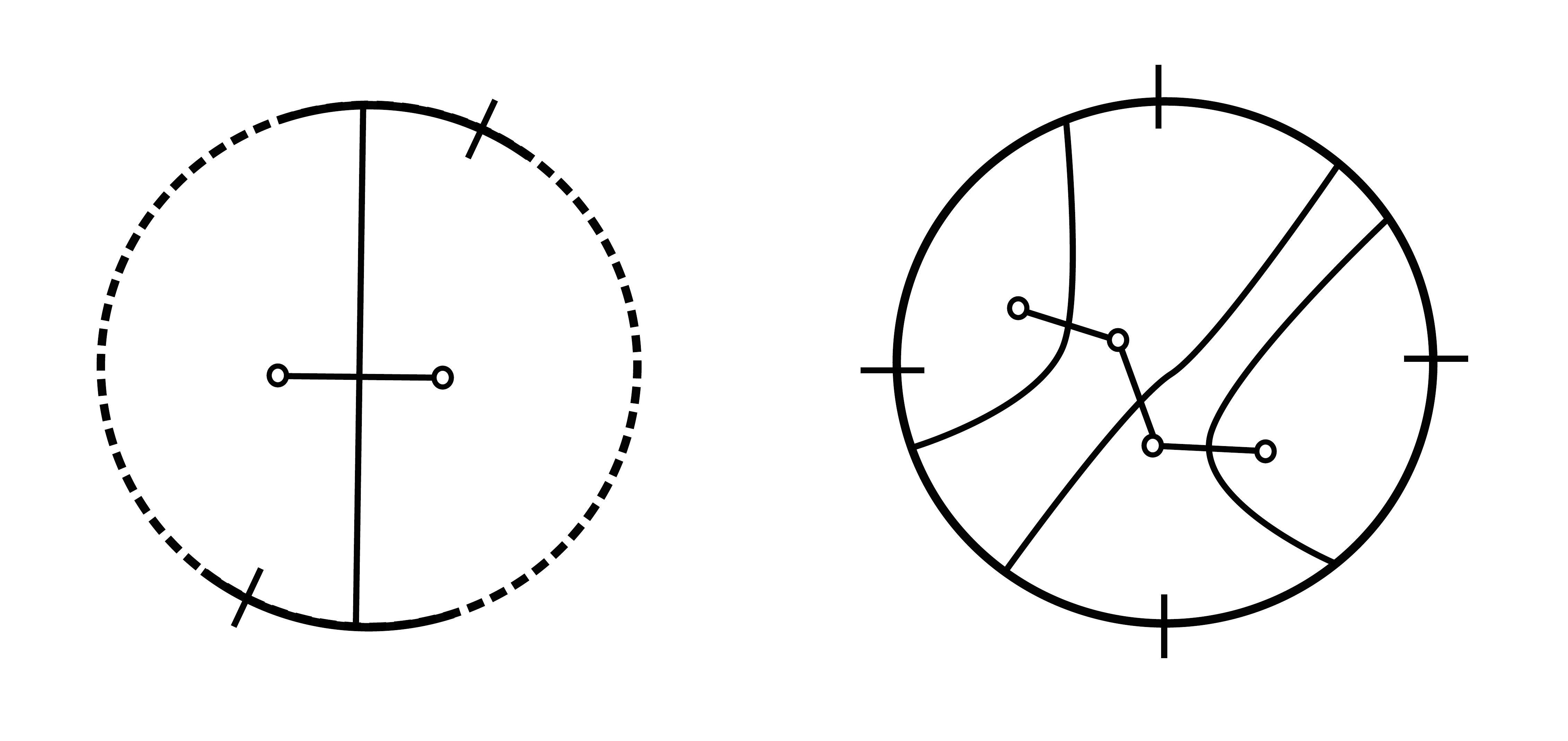}
\caption{a) A singular leaf is cross-shaped and consists of an arc between two elliptic points and an arc between two points on $\partial D_t$ meeting in a hyperbolic point. Here, the points on the boundary $\partial D_t$ belong to the arcs $A_i$ and $A_j$, so that the associated transposition is $\tau=(i\ j)$. The associated band generator is $a_{i,j}$ or $a_{i,j}^{-1}$ depending on the sign of the hyperbolic point. b) A photogram showing the singular leaves of the open book foliation on a disk $D_t$ induced by a mutually braided open book. The associated transpositions are $\tau_1=(1\ 2)$, $\tau_2=(3\ 4)$ and $\tau_3=(2\ 4)$. \label{fig:movie}}
\end{figure}

For each $t\in S^1$, for which all the singular leaves only contain one hyperbolic point each, we obtain in this way $n-1$ hyperbolic points labelled with transpositions. As we vary $t$, the hyperbolic points trace out lines, the derived bibraid of the mutually braided open books, with some points missing, corresponding to the values of $t$ for which there are singular leaves with more than one hyperbolic point, i.e., the values of $t$ for which the singular foliation fails to be an open book foliation. 

We start with a square whose bottom edge is an axis for $\varphi$ and its left edge is an axis for $t$, both going from 0 to $2\pi$. We obtain a Rampichini diagram by drawing in the diagram the lines made up from points with coordinates $(\varphi,t)$ for which there is a tangential intersection between the pages $F_{\varphi}$ and $D_{t}$. The labels of the lines in the diagram correspond to the transpositions labelling the corresponding hyperbolic points in the open book foliation on $D_{t_0}$. The lines $\ell_i$ in the Rampichini diagram are completed by including the data from the singular foliations on $D_t$ that are not open book foliations, which is done in the straightforward way, i.e., by filling in the intersection points between curves in the diagram.

Rampichini \cite{rampi} showed that this produces a Rampichini diagram. Condition 9 in Definition \ref{def:rampichini} is equivalent to the condition that the combinatorial structure of hyperbolic points with labelling transpositions at $t=0$ can indeed be realised by some singular foliation on the disk $D_{t=0}$. The way that the labels change at intersection points and at the right edge of the square guarantees that this remains true for all values of $t$. The fact that every totally braided open book gives rise to a Rampichini diagram also shows that we can assume that there are only finitely many values of $t$ and $\varphi$ for which the singular foliations on $D_t$ and $F_{\varphi}$ fail to be open book foliations and braid foliations, respectively, since there are only finitely many intersection points in a Rampichini diagram. 

We should justify why we called Eq. (\ref{eq:tau3}) and Eq. (\ref{eq:tau4}) BKL relations. Instead of labelling the lines in a Rampichini diagram by transpositions $(i\ j)$ we could label them by band generators $a_{i,j}$ if the corresponding hyperbolic point is positive or, equivalently, if the line in the Rampichini diagram is strictly monotone increasing when considered locally as the graph of a function in the variable $\varphi$, and by $a_{i,j}^{-1}$ if the corresponding hyperbolic point is negative or, equivalently, the line is strictly monotone decreasing. This is the actual labelling convention chosen in \cite{mortramp, rampi}. Since the sign of each labelling band generator is reflected in the sign of the slope of the labelled line in a Rampichini diagram, we do not lose any information by choosing transpositions instead of band generators as labels.

If we use band generators as labels, the change of labels at intersection points in Rampichini diagrams are given precisely by the BKL relations in Eq. (\ref{eq:1rel}) and Eq. (\ref{eq:2rel}). Say that an intersection point between two lines occurs at $(\varphi_0,t_0)$. We can read the labels of the two lines at the height $t=t_0-\delta$ for some small positive $\delta$ from left to right (i.e., with increasing $\varphi$), say $a_{i,j}^{\varepsilon_1}a_{k,m}^{\varepsilon_2}$, $\varepsilon_1,\varepsilon_2\in\{\pm 1\}$. Reading the labels from left to right at height $t_0+\delta$ then gives a sequence of two band generators that is equivalent to $a_{i,j}^{\varepsilon_1}a_{k,m}^{\varepsilon_2}$ via a BKL relation Eq. (\ref{eq:1rel}) or Eq. (\ref{eq:2rel}). Similarly, reading from bottom to the top at vertical lines $\varphi=\varphi_0-\delta$ and $\varphi=\varphi_0+\delta$ results in two sequences of two band generators that are equivalent via a BKL relation.




\begin{figure}[h]
\labellist
\pinlabel $i$ at 550 960
\pinlabel $i$ at 1595 960
\pinlabel $i$ at 2720 970
\pinlabel $j$ at 170 230
\pinlabel $j$ at 1230 220
\pinlabel $j$ at 2350 230
\pinlabel $k$ at 835 155
\pinlabel $k$ at 1890 150
\pinlabel $k$ at 3020 150
\Large
\pinlabel a) at 100 1000
\pinlabel b) at 1200 1000
\pinlabel c) at 2300 1000
\endlabellist
\centering
\includegraphics[height=4cm]{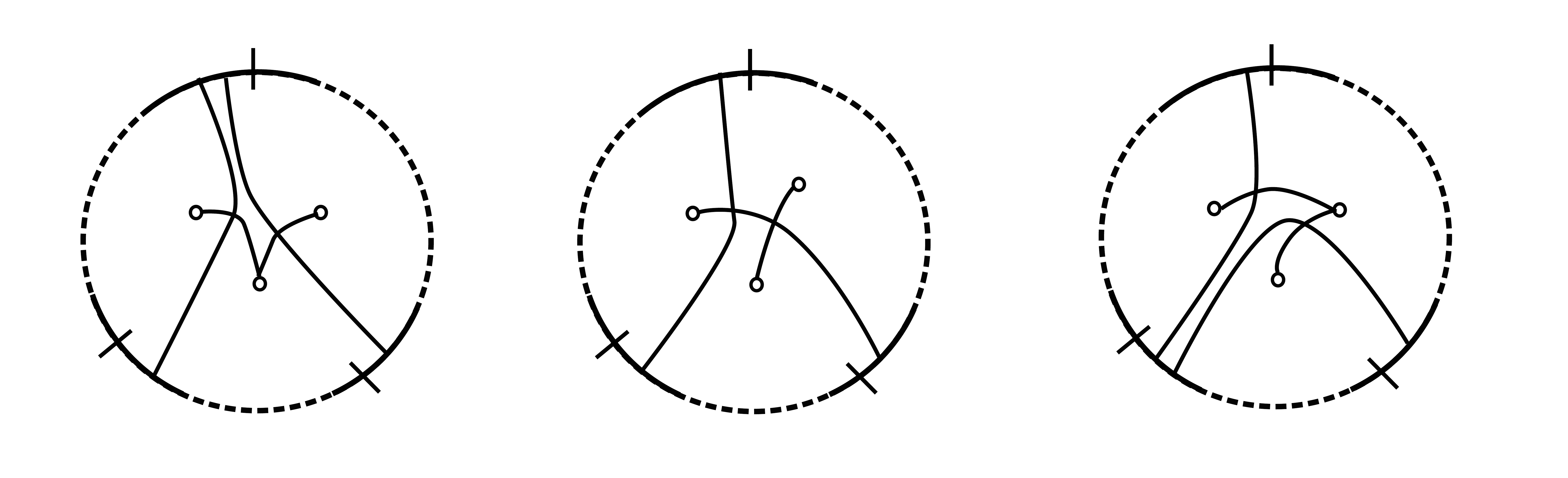}
\caption{The change in the singular foliation on $D_t$ as $t$ varies between $t_0-\delta$ to $t_0+\delta$, $\delta>0$, where $(\varphi_0,t_0)$ is an intersection point in the corresponding Rampichini diagram. a) The singular leaves of the singular foliation on $D_{t_0-\delta}$. b) The singular leaves of the singular foliation on $D_{t_0}$. c) The singular leaves of the singular foliation on $D_{t_0+\delta}$. \label{fig:rampimovie}}
\end{figure}
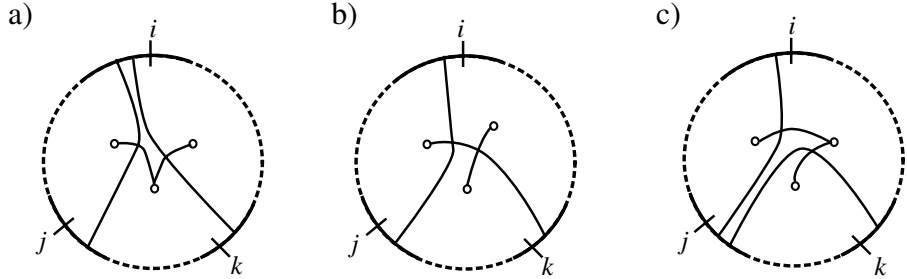

The singular leaves of the singular foliations of $D_{t_0-\delta}$, $D_{t_0}$, $D_{t_0+\delta}$, $\delta>0$, near an intersection point $(\varphi_0,t_0)$ in a Rampichini diagram are shown in Figure \ref{fig:rampimovie}. Pictures of the singular leaves of a singular foliation on $D_t$ as in Figure \ref{fig:rampimovie} are called \textit{photograms} by Rampichini \cite{rampi}. A finite ordered sequence of photograms depicting the changes of the singular leaves of a singular foliation on $D_t$ as $t$ varies from $0$ to $2\pi$ is called a \textit{film}. Depending on the signs of the two depicted hyperbolic points in Figure \ref{fig:rampimovie}, the figure corresponds to a BKL relation of the form $a_{j,k}a_{i,j}\to a_{i,j}a_{j,k}$, $a_{j,k}^{-1}a_{i,j}\to a_{i,j}a_{j,k}^{-1}$ or $a_{j,k}^{-1}a_{i,j}^{-1}\to a_{i,j}^{-1}a_{j,k}^{-1}$. The singular foliation depicted in Figure \ref{fig:rampimovie}b) corresponds to the presence of a singular leaf with more than one (i.e., two) hyperbolic points in the foliation on $D_{t_0}$. Hence the singular foliation on $D_{t_0}$ is not an open book foliation in the sense of Ito and Kawamuro.

When the cyclic variable $\varphi$ jumps from $2\pi$ to $0$ (that is, from the right edge of the Rampichini diagram to the left) all labels are increased by 1 modulo $n$. This is simply an artifact of the way that the labels of the arcs on $O$ and thus the labels in terms of band generators are defined.


Morton and Rampichini proved that not only does every totally braided open book give rise to a Rampichini diagram in this way, but also every Rampichini diagram arises from some totally braided open book \cite{mortramp, rampi}, i.e., from some open book that is mutually braided with an unbook. They use this to devise an algorithm that can decide whether a given braid is the binding of a totally braided open book. In the following we assume that the lines in Rampichini diagrams are labelled by band generators instead of transpositions. As we have already remarked, this is only a matter of convention and it is straightforward to move between transpositions and band generators.

For every $t_0\in S^1$ there are exactly $n-1$ intersection points of the lines in the diagram with the horizontal line $t=t_0$. The number of intersection points with vertical lines is also constant and is determined by the Euler characteristic of $F$. Note that for a fixed value of $\varphi=\varphi_0\in S^1$ the labels of the intersection points with the corresponding vertical line (read from the bottom to the top) spell a band word that represents $F$, whose boundary is a braid closing to $L$, while the labels at a fixed value of $t=t_0\in S^1$ give a band word for a braided disk.

By varying $\varphi$ from 0 to $2\pi$ we obtain a finite sequence of braid words in band generators, all of which close to $L$. Each braid word differs by the preceding one by a BKL relation or by a conjugation, so that all of them represent the same surface $F$. It follows from the properties of the diagram that exactly $n-1$ of these changes are conjugations. The sequence starts out with some braid word $B(0)$ and ends with the braid word $B(2\pi)$, which is the same as $B(0)$ except that all indices are shifted by $-1$ modulo $n$.

This offers a way of testing whether a given braid $B$ can be the binding of a totally braided open book, since this is the case if and only if it is the starting braid $B(0)=B$ of such a sequence corresponding to a diagram as above. Rampichini and Morton showed that for a given braid word there are only finitely many sequences to check, so that this test can be performed algorithmically. It is not stated explicitly in \cite{mortramp}, but it is not too hard to figure out what the conditions on a sequence of braid words are that guarantee that the lines in the corresponding diagram are indeed monotone. 

Since positive generators $a_{i,j}$ are required to be strictly monotone increasing, conjugation can move them only from the end of the braid word to the start of the braid word, not the other way round. Conversely, inverse generators $a_{i,j}^{-1}$ are only moved from the start to the end by conjugation, not from the end to the start. Furthermore, a BKL relation of the form $a_{i,j}^{-1}a_{k,l}\to a_{m,n}a_{o,p}^{-1}$ does not occur for any $i,j,k,l,m,n,o,p\in\{1,2,\ldots,n\}$, since this would require a strictly monotone decreasing line to intersect a strictly monotone increasing line from below. All other BKL relations can easily be realised so that there are no further restrictions on the sequences of braid words that correspond to diagrams with monotone lines.

Hence we have an algorithm that can decide if a given band word represents a braided fiber of a totally braided open book. However, it is still not known if every fibered link is the binding of a totally braided open book. The problem is that the algorithm only checks the band word. Hence if the algorithm decides that a given band word cannot be totally braided, this still leaves room for other braids with the same closure to be totally braided. Since a link is the closure of infinitely many different braids and a fiber can be represented by infinitely many band words, the algorithm cannot rule out that a fibered link can be totally braided, or, equivalently, generalised exchangeable.


\subsection{Simple branched covers}

Most of the definitions and results reviewed in this section can be found in \cite{morton}.

\begin{definition}
An $n$-sheeted branched covering map $\pi:F\to S$ between two surfaces $F$ and $S$ is called simple if for every point $p$ in the finite branch set $Q\subset S$ the preimage $\pi^{-1}(p)$ consists of $n-1$ points.

A map $\pi:M\to N$ between closed 3-manifolds $M$ and $N$ is a \textbf{simple branched cover} of degree $n$ with branch set $C\subset N$ if it is locally homeomorphic to the product of an interval with a simple $n$-sheeted branched cover of a disk, and the branch points in the products form the set $C$.  
\end{definition}

Consider a simple branched cover $\pi:S^3\to S^3$ whose branch set $C$ is some link $L_{branch}$. Let $\alpha$ be a braid axis for $L_{branch}$. Then $\pi^{-1}(\alpha)$ is a fibered link in $S^3$, since the fibration map $S^3\backslash \alpha\to S^1$ lifts through $\pi$, i.e. any fiber $F$ of the open book associated to $\pi^{-1}(\alpha)$ is the preimage $F=\pi^{-1}(S)$ of a fiber of the unbook, whose boundary is $\alpha$. Conversely, every fibered link arises in this way.

\begin{theorem}[Hilden-Montesinos \cite{hilden}]
For every fibered link $L$ there exists a simple branched cover $\pi:S^3\to S^3$, branched over some link $L_{branch}$ and a braid axis $\alpha$ of $L_{branch}$, such that $L=\pi^{-1}(\alpha)$. More precisely, the open book $(L,\Psi)$ is given by $(L,\Psi)=(\pi^{-1}(\alpha),\Phi(\pi))$, where $(\alpha,\Phi)$ is the unbook with binding $\alpha$.
\end{theorem}

An analogous result for every open book with connected binding of any 3-manifold $M$ was proved by Birman \cite{bircov}. Furthermore, in this case the degree of the branched cover can be fixed to be 3.

In 1991, when Harer's conjecture (cf. Conjecture \ref{con:harer}) was still open, Montesinos and Morton proposed an approach to prove this conjecture using simple branched covers of $S^3$. They observed that for a fixed simple branched cover $\pi:S^3\to S^3$ with branch set $L_{branch}$ the preimages $\pi^{-1}(\alpha)$ and $\pi^{-1}(\beta)$ of two braid axes $\alpha$ and $\beta$ of $L_{branch}$ are related by a sequence of Hopf plumbings and deplumbings \cite{morton}.

Actually, the Hopf (de)plumbings that correspond to changes of the braid axis take a very specific form.
\begin{definition}
Let $\pi:S^3\to S^3$ be a simple branched cover, branched over a link $L_{branch}$. Let $\alpha$ be a braid axis of $L_{branch}$ and $F$ be a page of the open book with binding $\pi^{-1}(\alpha)$. Let $\gamma$ be a simple path in $F$ with $\gamma\cap\partial F=\partial \gamma$. We say that $\gamma$ is \textbf{$\pi$-symmetric} if $\pi(\gamma)$ is a path in $\pi(F)$ with one endpoint on the boundary $\partial \pi(F)$ and the other on $\pi(F)\cap L_{branch}$.
We call a Hopf (de)plumbing along a $\pi$-symmetric path a \textbf{$\pi$-symmetric Hopf (de)plumbing}.
\end{definition}

\begin{figure}[h]
\centering
\labellist
\pinlabel $\gamma$ at 560 1680
\pinlabel $\pi(\gamma)$ at 530 330
\pinlabel $F$ at 1500 1000
\pinlabel $D$ at 1500 300
\Large
\pinlabel $\pi$ at 740 660
\endlabellist
\includegraphics[height=7cm]{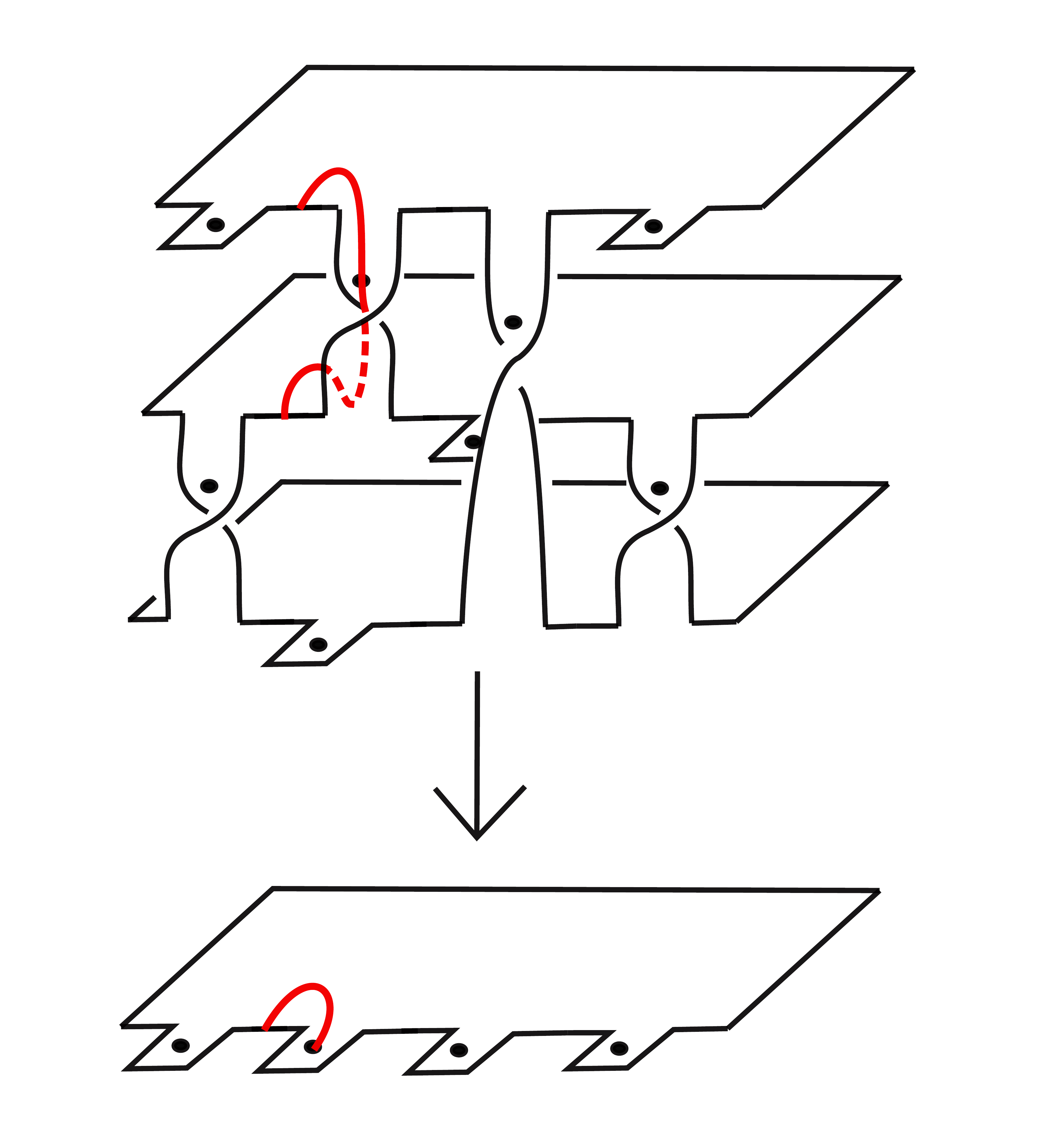}
\caption{A $\pi$-symmetric path on a fiber $F$ and its image in $D$ under $\pi$. \label{fig:cover}}
\end{figure}
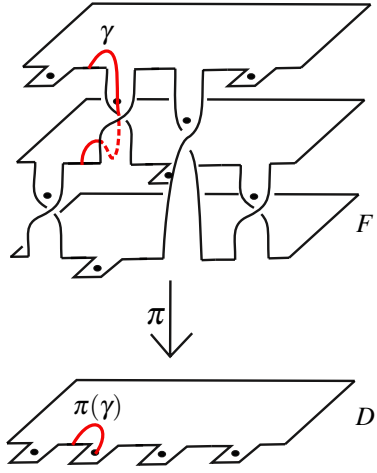

This leads to the following conjecture, which would imply a new proof of Harer's conjecture.
\begin{conjecture}
\label{con:morton}
For every fibered link $L$ there is a simple branched cover $\pi:S^3\to S^3$, branched over a link $L_{branch}$ with braid axes $\alpha$ and $\beta$, such that $\pi^{-1}(\alpha)=L$ and $\pi^{-1}(\beta)$ is the unknot.
\end{conjecture} 

Equivalently, Conjecture \ref{con:morton} is asking if every fibered link can be obtained from the unknot and its fiber disk by a sequence of $\pi$-symmetric Hopf (de)plumbings, which is stronger than Harer's conjecture and in particular not answered by Giroux's and Goodman's proof of the same.

Montesinos and Morton showed that Conjecture \ref{con:morton} is true for fibered links $L$ whose fiber surface can be obtained from the disk via a sequence of Hopf plumbings without any deplumbings \cite{morton}.

Since the open books in the domain 3-sphere are simply the lifts of the open books in the target 3-sphere, i.e. compositions of $\pi$ and the fibration map of the unbook, we can study some of their properties, most importantly, transverse and tangential intersection points, by considering the corresponding open books in the target 3-sphere.

\begin{definition}\label{def:hopf}
Let $\pi:S^3\to S^3$ a simple branched cover of degree $n$, branched over a link $L_{branch}$. Let $\alpha$ and $\beta$ be braid axes of $L_{branch}$ such that
\begin{itemize}
\item $\alpha\cup L_{branch}$ is braided relative to $\beta$,
\item $L_{branch}$ is an $(n-1)$-braid relative to $\beta$,
\item $\beta\cup L_{branch}$ is braided relative to $\alpha$.
\end{itemize}
Then we say that $\pi^{-1}(\alpha)\cup \pi^{-1}(\beta)$ is \textbf{lifted generalised exchangeable}.
If additionally $H\defeq\alpha\cup\beta$ is a Hopf link, we say that $H$ is a Hopf $n$-braid axis of $L_{branch}$. 
\end{definition}

Note that the number of strands of the braid $L_{branch}$ relative to $\beta$ is specified, while the number of strands of $\beta\cup L_{branch}$ relative to $\alpha$ and of $\alpha\cup L_{branch}$ relative to $\beta$ is not. This will guarantee that $\pi^{-1}(\beta)$ is an unknot (cf. Lemma \ref{lem:RH}), while $\pi^{-1}(\alpha)$ is some general fibered link. In particular, $\alpha$ is braided relative to $\beta$ and vice versa. Hence they are exchangeable.

We will see that the property 
\begin{enumerate}[label=B\arabic*)]
\setcounter{enumi}{4}
\item The binding $L$ has a braid axis $O$ such that $L\cup O$ is lifted generalised exchangeable for some simple branched cover $\pi:S^3\to S^3$.
\end{enumerate} 
is actually also equivalent to the four characterisations of braided open books given in the introduction. However, we will work with property B3) instead, which in addition to B5) requires the the link type of $\alpha\cup \beta$ to be a Hopf link. It is therefore clear, that B3) implies B5). 


We will go on to prove that B5)$\implies$ B1)$\implies$ B2) $\implies$ B4)$\implies$ B3), so that B3) and B5) are equivalent, i.e., a fibered link and one of its braid axes are lifted generalised exchangeable for some simple branched cover $\pi:S^3\to S^3$ if and only if they are preimages of a Hopf $n$-braid axis for some (in general different) simple branched cover $\pi':S^3\to S^3$  . 

\subsection{P-fibered braids}
\label{sec:poly}

This subsection follows the introduction of P-fibered braids in \cite{bodesat}.

There are several equivalent interpretations of the braid group $\mathbb{B}_n$ on $n$ strands. One of them is to consider the configuration space
\begin{equation}
C_n\defeq\{(z_1,z_2,\ldots,z_n)\in\mathbb{C}^n:z_i\neq z_j \text{ if }i\neq j\}/S_n
\end{equation} 
of unordered $n$-tuples of distinct complex numbers. Here $S_n$ denotes the symmetric group on $n$ elements, acting on $n$-tuples by permutation as usual. The braid group $\mathbb{B}_n$ is the fundamental group of $C_n$.

Since the fundamental theorem of algebra establishes a homeomorphism between $C_n$ and the space of monic, degree $n$, complex polynomials with distinct roots by sending $(z_1,z_2,\ldots,z_n)$ to the polynomial $\prod_{i=1}^n(z-z_j)$, the braid group can also be thought of as the fundamental group of this space of polynomials, so that every geometric braid (parametrised as $\cup_{i=1}^n (z_j(t),t)$ in $\mathbb{C}\times[0,2\pi]$) gives rise to a unique loop in the space of polynomials $g_t(z)=\prod_{i=1}^n(z-z_j(t))$.

Note that the zeros of $g_t$ trace out the braid $B$ as $t$ varies from 0 to $2\pi$. Hence the argument of the polynomials defines a map $\chi(z,t)\defeq\arg g_t(z):(\mathbb{C}\times S^1)\backslash cl(B)\to S^1$, where $cl(B)$ denotes the closed braid $B$ in $\mathbb{C}\times S^1$. We use the same expression to denote its closure in $S^3$.

\begin{definition}\label{def:pfib}
A geometric braid $B$ is called \textbf{P-fibered} if the argument $\chi:(\mathbb{C}\times S^1)\backslash cl(B)\to S^1$ of the loop of polynomials corresponding to $B$ is a fibration.\\
We call a braid $B\in\mathbb{B}_n$ P-fibered if it can be represented by a P-fibered geometric braid.
\end{definition}

Since $g_t$ is a monic polynomial for all values of $t$, it satisfies
\begin{equation}
\label{eq:limit}
\lim_{r\to\infty}\arg(g_t(r\rme^{\rmi \varphi}))=\varphi\deg g_t=\varphi n
\end{equation}
for all $\varphi\in[0,2\pi]$. Hence $\chi$ can be extended to all of $S^3\backslash cl(B)$ by filling the complementary solid torus with meridional disks or, equivalently, by considering the complex plane as an open disk and identifying the boundary of $D\times S^1$ along longitudes: $(\varphi,t)\sim(\varphi,t')$ for all $\varphi,t,t'\in [0,2\pi]$. Since the resulting fibration map $S^3\backslash cl(B)\to S^1$ has the required behaviour in a tubular neighbourhood of the braid closure $cl(B)$ \cite{bodesat}, closures of P-fibered braids are fibered links in $S^3$. 

Several families of fibered links have been shown to be closures of P-fibered braids. In particular, homogeneous braids (see Definition \ref{def:homogeneous}) are P-fibered \cite{bode:real, survey}. However, it is not known if every fibered link is the closure of a P-fibered braid.

The argument map $\chi$ being a fibration is equivalent to the absence of any critical points. Since $g_t$ is holomorphic for all $t$, we find that $(z_*,t_*)\in(\mathbb{C}\times S^1)\backslash cl(B)$ is a critical point of $\chi$ if and only if
\begin{align}
\frac{\partial \chi}{\partial \text{Re}(z)}(z_*,t_*)=\frac{\partial \chi}{\partial \text{Im}(z)}(z_*,t_*)=\frac{\partial \chi}{\partial t}(z_*,t_*)&=0\nonumber\\
\iff \frac{\partial g_t}{\partial z}(z_*)=\frac{\partial\chi}{\partial t}(z_*,t_*)&=0.\label{eq:crit}
\end{align}

Since $g_t$ is a polynomial of degree $n$, $\frac{\partial g_t}{\partial z}$ has $n-1$ zeros for every value of $t$. We denote by $c_j(t)$, $j=1,2,\ldots,n-1$, the zeros of $\frac{\partial g_t}{\partial z}$, i.e., the critical points of $g_t$.

We write $v_j(t)\defeq g_t(c_j(t))$ for the critical values of $g_t$. Since the roots of $g_t$ are simple, its critical values $v_j(t)$, $j=1,2,\ldots,n-1$, are non-zero and after a small isotopy to $B$ we can assume that $v_i(t)\neq v_j(t)$ if $i\neq j$. If this condition is satisfied, it follows that the curves $v_j(t)$ together with a vertical strand $(0,t)\in\mathbb{C}\times[0,2\pi]$, which we call the \textit{0-strand}, form a braid on $n$ strands. Eq. (\ref{eq:crit}) is then equivalent to
\begin{equation}
\label{eq:satellite}
\frac{\partial \arg v_j}{\partial t}(t_*)=0 \qquad\text{for some}j\in\{1,2,\ldots,n-1\}.
\end{equation}

In other words, a geometric braid is P-fibered if and only if Eq. (\ref{eq:satellite}) is never satisfied.

This has the following geometric interpretation. The distinct, non-zero critical values $v_j(t)$, $j=1,2,\ldots,n-1$, of $g_t$ move in the complex plane as $t$ goes from $0$ to $2\pi$. If $\chi(z,t)=\arg g_t(z)$ is a fibration, then Eq. (\ref{eq:satellite}) says that no critical value ever changes its orientation with which it twists around 0. In other words, for every $j$ there is a direction, clockwise or counterclockwise, corresponding to the sign of $\frac{\partial \arg v_j}{\partial t}$, such that $v_j(t)$ moves in that direction around 0 for all values of $t\in[0,2\pi]$.


In \cite{survey} we discuss a construction of P-fibered braids, that uses a deep relation between the direction, clockwise or counterclockwise, of a critical value $v_j(t)$ and the sign of the exponent, positive or negative, of an Artin generator appearing in the braid word of $B$, the braid that is formed by the roots of $g_t$.

\begin{definition}\label{def:homogeneous}
Let $B$ be a braid on $n$ strands. We say that $B$ is \textbf{homogeneous} if it is represented by a word in Artin generators that for all $i=1,2,\ldots,n-1$ contains the generator $\sigma_i$ if and only if it does not contain $\sigma_i^{-1}$.
\end{definition}

The family of homogeneous braids, sometimes in the literature also referred to as strictly homogeneous braids in order to distinguish them from split links with homogeneous braid diagrams in the sense of Cromwell \cite{cromwell}, can be constructed (up to conjugation) as P-fibered braids with the methods in \cite{survey}. This is precisely because of the relation between the direction of a critical value and the sign of a corresponding generator. Since no critical value changes its direction as $t$ varies, no generator changes its sign as we traverse the braid word.

\begin{theorem}[cf. \cite{bode:real, survey}]
Let $B$ be a homogeneous braid. Then $B$ is conjugate to a P-fibered braid.
\end{theorem}

The fact that homogeneous braids close to fibered links is due to Stallings \cite{stallings2}. The construction in \cite{survey} suggests that given a loop in the space of polynomials (monic, with fixed degree, distinct roots and distinct critical values) there is a relation between the braid that is formed by its roots and the braid that is formed by the union of the 0-strand and its critical values. The following results imply that this relation is of a topological nature, as deformations of the strands that correspond to critical values lift to braid isotopies of the braid that is formed by the roots.

\begin{theorem}[Beardon-Carne-Ng, Theorem 1 in \cite{critical}]
\label{thm:crit}
Let $X_n$ be the space of monic complex polynomials, with degree $n$, distinct roots, distinct critical values and constant term equal to 0. Let 
\begin{equation}
V_n\defeq\{(v_1,v_2,\ldots,v_{n-1})\in\mathbb{C}^{n-1}:v_j\neq 0, v_i\neq v_j\text{ for all }i\neq j\}/S_{n-1},
\end{equation}
where $S_{n-1}$ is the symmetric group on $n-1$ elements, acting by permutations on $(n-1)$-tuples, be the space of unordered critical values of polynomials in $X_n$.

Then the map $\theta_n:X_n\to V_n$ that maps a polynomial to its set of critical values is a covering map of degree $n^{n-1}$.
\end{theorem}

\begin{corollary}
\label{cor:lifting}
Let $\widehat{X_n}$ be the space of monic complex polynomials, with degree $n$, distinct roots, distinct critical values and constant term distinct from any of its critical values and 
\begin{equation}
\label{eq:whvn}
\widehat{V_n}\defeq\{((v_1,v_2\ldots,v_{n-1}),a_0)\in V_n\times \mathbb{C}:a_0\neq v_i \text{ for all } i\}.
\end{equation} 
Then $\widehat{\theta_n}:\widehat{X_n}\to\widehat{V_n}$, defined to be the map that sends a polynomial to its set of critical values $(v_1,v_2,\ldots,v_{n-1})$ and its constant term $a_0$, is a covering map of degree $n^{n-1}$.
\end{corollary}
\begin{proof}
The map $\widehat{\theta_n}$ is the composition of $h_1:\widehat{X_n}\to X_n\times \mathbb{C}$, $h_1(\sum_{j=0}^n a_j z^j)=(\sum_{j=1}^n a_j z^j,a_0)$, the covering map $(\theta_n,\text{id}):X_n\times\mathbb{C}\to V_n\times \mathbb{C}$ of degree $n^{n-1}$ and the homeomorphism $h_2:V_n\times \mathbb{C}\to V_n\times \mathbb{C}$, $((v_1,v_2,\ldots,v_{n-1}),a_0)\mapsto ((v_1+a_0,v_2+a_0,\ldots,v_{n-1}+a_0),a_0)$. 

Note that the image of $h$ is $(h_2\circ(\theta_n,\text{id}))^{-1}(\widehat{V_n})$ and that $h$ is a homeomorphism, if the target space is restricted to $(h_2\circ(\theta_n,\text{id}))^{-1}(\widehat{V_n})$. The statement of the corollary follows.
\end{proof}



In particular, we have the homotopy lifting property, so that deformations of the braid formed by the critical values and the 0-strand (as long as the critical values remain distinct from $a_0(t)$) lift to braid isotopies of the braid formed by the roots.

Theorem \ref{thm:crit} has been used in \cite{bode:adicact} to construct some non-homogeneous P-fibered braids and to explore some related algebraic structures.

The original motivation for the definition of P-fibered braids comes from the study of links of singularities of real polynomial maps, much in the spirit of \cite{milnor}.

For a given polynomial map $f:\mathbb{R}^4\to\mathbb{R}^2$, we consider the vanishing set of $f$: 
\begin{equation}
V_f=\{(x_1,x_2,x_3,x_4)\in\mathbb{R}^4:f(x_1,x_2,x_3,x_4)=(0,0)\}.
\end{equation}

Suppose that $f$ and its first derivates vanish at the origin, i.e., $f(0,0,0,0)=(0,0)$ and $\tfrac{\text{d}f}{\text{d}x_i}(0,0,0,0)=(0,0)$ for all $i=1,2,3,4$. If there is a neighbourhood $U$ of the origin in $\mathbb{R}^4$ such that the matrix $\nabla f(x)$ has full rank for all $x\in U\backslash\{(0,0,0,0)\}$, then we say $f$ has an \textit{isolated singularity} at the origin.

If this is the case, the intersection of $V_f$ and the three-sphere 
\begin{equation}
S^3_{\rho}=\{(x_1,x_2,x_3,x_4)\in\mathbb{R}^4:\sum_{i=1}^4 x_i^2=\rho^2\}
\end{equation} 
of sufficiently small radius $\rho>0$ is a link $L$. Furthermore, it is always the same link type $L$, no matter which (sufficiently small) radius $\rho$ is chosen. As in the complex setting the link $L$ is called the \textit{link of the singularity}. 

\begin{definition}
\label{def:ralg}
A link $L$ is called \textbf{real algebraic} if there is a polynomial $f:\mathbb{R}^4\to\mathbb{R}^2$ with an isolated singularity at the origin such that $L$ is the link of the singularity.
\end{definition}

Definition \ref{def:ralg} is the precise real analogue of Milnor's algebraic links of isolated singularities of complex plane curves. In contrast to the algebraic links it is not known which links are real algebraic. In the definition of real algebraic links we have essentially only replaced each $\mathbb{C}$ by $\mathbb{R}^2$. However, since real polynomial maps are not necessarily holomorphic, we have a lot more functions at our disposal. Hence we should expect many more real algebraic links than algebraic links. On the other hand, the gradient of a real polynomial as above is a 2-by-4 matrix and can hence be non-zero, but still not have full rank. This means that generic real polynomials do not have isolated singularities, which makes explicit constructions much more difficult. This was already noted by Milnor in the last chapter of his seminal work \cite{milnor}, where he also establishes the following result.
\begin{theorem}[Milnor, Theorem 11.2 in \cite{milnor}]
Let $L$ be a real algebraic link. Then $L$ is fibered.
\end{theorem}

Benedetti and Shiota conjecture that this implication should be an equivalence \cite{benedetti}.
\begin{conjecture}[Benedetti-Shiota, Conjecture 1.6 in \cite{benedetti}]
\label{conj:bene}
A link $L$ is real algebraic if and only if it is fibered.
\end{conjecture}

Despite several constructions of different families of real algebraic links (cf. \cite{bode:real,looijenga, perron, pichon, rudolph}), the conjecture is still wide open.

We would like to point out that if we consider real analytic maps with so-called \textit{tame isolated singularities} instead of real polynomials, then Conjecture \ref{conj:bene} is known to be true \cite{kn, looijenga}.

Real algebraic links are connected to P-fibered braids as follows.

\begin{theorem}[cf. \cite{bode:2016lemniscate, bode:real}]
\label{thm:ralg}
Let $B$ be a P-fibered braid. Then the closure of $B^2$ is real algebraic.
\end{theorem}

Here $B^2$ denotes the braid product of a braid $B$ with itself, which is represented by the concatenation of two copies of the same braid word.

Therefore, proving that all fibered links are closures of P-fibered braids would constitute a helpful step with regards to Conjecture \ref{conj:bene}. So far, we know that homogeneous braids and the non-homogeneous family from \cite{bode:adicact} are P-fibered. Moreover, there are certain satellite and twisting operations that were defined in \cite{bodesat}. These operations can be used to construct new P-fibered braids from known ones.

\subsection{Some relations between the different notions}

In this subsection we highlight (without being precise) some structural similarities and differences between the different types of braidings defined in the previous sections.

The properties B1) through B4) in the introduction put (a priori) different requirements on the open books in question. Property B1) only mentions the position of the bindings relative to the pages of the other OBD. Property B2) then demands that not only the bindings, but also the pages themselves are in a `good position' relative to the pages of the other OBD. Property B3) takes this one step further by saying that the OBDs should not only be in a good position to each other but also respect the additional structure of some simple branched cover and property B4) then further restricts the type of simple branched cover that we are allowed to consider.

Properties B3) and B4) are arguably somewhat further removed from an intuitive definition of a braided open book. However, they lend themselves more easily to applications such as the mentioned real algebraic links and Harer's conjecture.

There is a structural similarity between the properties B2) through B4), which will be very useful to us. Let $(L,\Psi)$ be an open book with fibers $F_{\varphi}$, $\varphi\in S^1$, that is mutually braided with the unbook $(O,\Phi)$ with pages $D_t$, $t\in S^1$. The derived bibraid is formed by the tangential intersection points of the pages $F_\varphi$ with the pages $D_t$. In other words, the derived bibraid of a pair of mutually braided open book is formed by the hyperbolic points of the singular foliation on $F_{\varphi}$, when we vary $\varphi$ from 0 to $2\pi$, or, equivalently, by the hyperbolic points of the singular foliation on $D_t$, varying $t$ from 0 to $2\pi$. The sign of a hyperbolic point is positive if it belongs to a component of $pos(L)$ and negative if it belongs to a component of $neg(L)$. Every such hyperbolic point can be associated to a band of the braided surface and diagrams as devised by Rampichini reveal the combinatorial information about which point belongs to which band and how this changes as we vary $\varphi$ or $t$.

For a simple branched cover $\pi:S^3\to S^3$ of degree $n$ branched over a link $L_{branch}$ with a Hopf $n$-braid axis $\alpha\cup\beta$, let $(L,\Psi)$ and $(O,\Phi)$ be the corresponding lifted open books in the domain $S^3$. Again we obtain open book foliations on (almost all) pages of each open book induced by the other open book. The hyperbolic points correspond to the critical points of $\pi$, where $\pi$ fails to be a covering map. The fact that its image $L_{branch}$ is braided relative to both $\alpha$ and $\beta$ implies that this link of hyperbolic points is braided relative to both $L$ and $O$, just like the bibraid of a pair of mutually braided open books.

For a P-fibered braid $B$ of degree $n$, given by the zeros of the loop of polynomials $g_t:\mathbb{C}\to\mathbb{C}$ of degree $n$, the singular foliation on the pages of the unbook $D_t$ induced by $\arg g_t$ has exactly $n-1$ hyperbolic points, which are the critical points $c_j(t)$, $j=1,2,\ldots,n-1$, of $g_t$. These obviously form a braid relative to $O$, which here is the boundary circle of $\mathbb{C}$, and the fact that $B$ is P-fibered implies that it is also braided relative to the closure of $B$. This is because the critical values $v_j(t)$, $j=1,2,\ldots,n-1$, never change their direction with which they twist around the 0-strand. In other words, when we consider the 0-strand, which is the image of $B$ under $g_t$, as $t$ varies from 0 to $2\pi$, as the unknot, the critical values are always transverse to the pages $\varphi=\arg z$, $z\in\mathbb{C}\backslash\{0\}$. The direction of the movement of the critical values (clockwise or counter-clockwise) reflects the sign of the corresponding hyperbolic point (negative or positive).

Thus for B2), B3) and B4) the singular foliations on fibers give rise to a bibraid. We haven't really proved any of the claims in this subsection (although many of them are not difficult to prove). The similarities above should be seen as a motivation of why we should expect these concepts to be related in some way. Note that Morton-Rampichini's algorithm to detect mutual braiding relies only on the combinatorial data of the bibraid stored in the Rampichini diagram. Hence the structural similarity outlined above captures a lot of what it means for an open book to be braided.

\section{Hopf braid axes and polynomials}\label{sec:43}

In this section, we describe a construction of simple branched covers of $S^3$ starting from a P-fibered braid. The P-fibered braid and its braid axis become the preimage of a Hopf $n$-braid axis of the branch link of the constructed simple branched cover.

\begin{lemma}
\label{lemma43}
Let $B$ be a P-fibered braid. Then its closure $L$ and its braid axis $O$ arise as the preimages of the two components of a Hopf $n$-braid axis of the branch link $L_{branch}$ of a simple branched cover $\pi:S^3\to S^3$.
\end{lemma}
\begin{proof}
To a large degree this has already been shown in \cite{survey}. We are going to explicitly construct the simple branched cover $\pi$ from the polynomials $g_t$ corresponding to the P-fibered braid $B$.

Recall that by Definition \ref{def:pfib} there is a loop in the space of monic complex polynomials $g_t:\mathbb{C}\to\mathbb{C}$ of degree $n$, the number of strands of $B$, with the property that the roots of $g_t$ trace out the braid $B$ as $t$ varies from 0 to $2\pi$ and $\arg g:(\mathbb{C}\times S^1)\to S^1$ is a fibration map defining an open book for the closure of $B$.

Let $D$ be the open unit disk in $\mathbb{C}$ with closure $\overline{D}$. Note that $(\overline{D}\times S^1)/((\rme^{\rmi \chi},\rme^{\rmi t_1})\sim(\rme^{\rmi \chi},\rme^{\rmi t_2}))\cong S^3$. Let $\phi:D\to \mathbb{C}$ be an orientation-preserving diffeomorphism such as $\phi(u)=\frac{u}{1+|u|}$. Given a P-fibered braid $B$ on $n$ strands and the corresponding function $g_t$, we can define the map $\pi:S^3\to S^3$,
\begin{align}
\pi(\phi^{-1}(u),\rme^{\rmi t})&=(\phi(g_t(u)),\rme^{\rmi t}),\nonumber\\
\pi(\rme^{\rmi \chi},\rme^{\rmi t})&=(\rme^{\rmi \chi n},\rme^{\rmi t}).
\end{align}

Using Eq. (\ref{eq:limit}) and basic properties of complex polynomials it is not difficult to check that $\pi$ is a simple branched cover, branched over the link $L_{branch}$ that is the closure of $(v_1(t),v_2(t),\ldots,v_{n-1}(t))$, i.e., the critical values of the complex polynomial $g_t$. In the coordinates chosen above the branch link is formed by the curves $(\phi(v_j(t)),\rme^{\rmi t})$, $j=1,2,\ldots,n-1$.

We claim that with the choice of $\alpha=(0,\rme^{\rmi t})$, with $t$ going from 0 to $2\pi$, and $\beta=(\rme^{\rmi s},\rme^{\rmi t})$, with $s$ going from 0 to $2\pi$, the union $\alpha\cup\beta$ forms a Hopf $n$-braid axis of $L_{branch}$.

The link $\alpha\cup L_{branch}$ is a braid on $n$ strands relative to $\beta$, $n-1$ of which are formed by the critical values $v_j(t)$ and the $n$th strand is the 0-strand, i.e. $\alpha$. The preimage $\pi^{-1}(\beta)$ is the unknot $(\rme^{\rmi \chi},\rme^{\rmi t})$, where $\chi$ is going from 0 to $2\pi$. 

The fact that $B$ is P-fibered implies that Eq. (\ref{eq:satellite}) is never satisfied, so that all of $(\phi(v_j(t)),\rme^{\rmi t})$, $j=1,2,\ldots,n-1$, and $\beta$ are transverse to the fibers of $(z,\rme^{\rmi t})\mapsto \arg(z)$, the fibration map of the complement of $\alpha$. We can thus choose the orientation of the components of $L_{branch}$ such that $\beta\cup L_{branch}$ is braided relative to $\alpha$. Clearly, $\alpha\cup \beta$ is a Hopf link. Therefore $\alpha$ and $\beta$ form a Hopf $n$-braid axis for $L_{branch}$

The preimage $\pi^{-1}(\alpha)$ is the closure of $B$, which proves the lemma.  
\end{proof}

Lemma \ref{lemma43} shows that the braiding property B4) implies B3).

\section{Simple branched covers and exchangeable braids}\label{sec:42}

In this section we show that B3) implies B1), that is, the preimages of a Hopf $n$-braid axis are generalised exchangeable.

\begin{lemma}\label{lem:RH}
Let $\pi:S^3\to S^3$ be a simple branched cover, branched over a link $L_{branch}$, of degree $n$. Then the braid index $b(L_{branch})$ of $L_{branch}$ is at least $n-1$ and a braid axis $\alpha$ of $L_{branch}$ lifts to the unknot $\pi^{-1}(\alpha)=O$ if and only if it realises the braid index.
\end{lemma}
\begin{proof}
Let $\alpha$ be a braid axis of $L_{branch}$, realising the braid index of $L_{branch}$, i.e., each fiber disk intersects $L_{branch}$ in exactly $b(L_{branch})$ points. The Riemann-Hurwitz formula expresses the Euler characteristic of $\pi^{-1}(D)$ in terms of $b(L_{branch})$ as
\begin{equation}
\chi(\pi^{-1}(D))=n\chi(D)-b(L_{branch}),
\end{equation} 
and since the Euler characteristic of $\chi(\pi^{-1}(D))$ is at most 1 (if and only if $\pi^{-1}(D)$ is also a disk), it follows that $b(L_{branch})\geq n-1$ with equality if and only if $\pi^{-1}(\alpha)$ is an unknot.
\end{proof}

\begin{corollary}
Let $\pi:S^3\to S^3$ be a simple branched cover, branched over a link $L_{branch}$, of degree $n$. Then the preimage $\pi^{-1}(\alpha)$ of any braid axis $\alpha$ of $L_{branch}$ can be obtained from the unknot via a sequence of $\pi$-symmetric Hopf plumbings and deplumbings if and only if $b(L_{branch})=n-1$.
\end{corollary}
\begin{proof}
By Montesinos-Morton \cite{morton} two fibered links are related by a sequence of $\pi$-symmetric Hopf plumbings and deplumbings if and only if they arise as the preimages $\pi^{-1}(\alpha)$ and $\pi^{-1}(\beta)$ of two braid axes $\alpha$ and $\beta$ of $L_{branch}$.

The proof of the previous lemma shows that there is a braid axis $\beta$ of $L_{branch}$ whose preimage is an unknot if and only if $b(L_{branch})=n-1$.
\end{proof}

\begin{remark}
Note that the Corollary above does not exclude the possibility of some other simple branched cover $\pi':S^3\to S^3$ such that $\pi^{-1}(\alpha)=\pi'^{-1}(\alpha')$ for braid axes $\alpha$ and $\alpha'$ of the branch links $L_{branch}$ and $L_{branch}'$, respectively, and such that $\pi^{-1}(\alpha')$ can be obtained from the unknot via a sequence of $\pi'$-symmetric Hopf plumbings and deplumbings. In particular, this corollary cannot be used (in a straightforward fashion) to disprove Montesinos-Morton's version of Harer's conjecture (Conjecture \ref{con:morton}).
\end{remark}

\begin{lemma}
\label{lem:liftex}
Let $L=\pi^{-1}(\alpha)$ and $L'=\pi^{-1}(\beta)$ be lifted generalised exchangeable for some simple branched cover $\pi:S^3\to S^3$ with branch link $L_{branch}$ and braid axes $\alpha$ and $\beta$. Then $L'$ is the unknot $O$, and $L$ and $O$ are generalised exchangeable.
\end{lemma}
\begin{proof}
By Definition \ref{def:hopf} there are braid axes $\alpha$ and $\beta$ of the branch link $L_{branch}$ of $\pi$ such that $L=\pi^{-1}(\alpha)$ and $L'=\pi^{-1}(\beta)$, and $\beta\cup L_{branch}$ is braided relative to $\alpha$ and $\alpha\cup L_{branch}$ is braided relative to $\beta$, with $L_{branch}$ being an $(n-1)$-braid relative to $\beta$. Therefore, by Lemma \ref{lem:RH} $\pi^{-1}(\beta)$ is an unknot $O$.

Exchangeability of $\alpha$ and $\beta$ implies that $\alpha$ is positively transverse to the pages of the unbook with binding $\beta$ and $\beta$ is positively transverse to the pages of the unbook with binding $\alpha$.

Since the open book decompositions of $\pi^{-1}(\alpha)$ and $\pi^{-1}(\beta)$ are the lifts of the open books of $\alpha$ and $\beta$, $L=\pi^{-1}(\alpha)$ is transverse to the pages of the open book of $O=\pi^{-1}(\beta)$ and vice versa. Hence $L$ and $O$ are generalised exchangeable.
\end{proof}

Lemma \ref{lem:liftex} shows that B5) implies B1). It is clear from Definition \ref{def:hopf} that B3) implies B5), so that Lemma \ref{lem:liftex} proves that B3) implies B1). Note that in this case, where $\alpha$ and $\beta$ form a Hopf $n$-braid axis, $\alpha$ is a 1-braid relative to $\beta$. It follows that $L=\pi^{-1}(\alpha)$ is an $n$-braid relative to $O=\pi^{-1}(\beta)$.





Rudolph also found a way to obtain totally braided open books from certain simple branched covers \cite{rudolph2}. The implication B3)$\implies$ B2) (and hence also B3)$\implies$ B1)) can alternatively be proved by showing that simple branched covers with a Hopf $n$-braid axis satisfy the conditions in Proposition 2.6 in \cite{rudolph2}.

\section{Mutually braided open books and polynomials}\label{sec:31}

In this section we prove the implication B2)$\implies$B4), that is, the bindings of totally braided open books are closures of P-fibered braids.

\begin{proposition}\label{prop:42}
Let $(L,\Psi)$ be an open book in $S^3$ that is mutually braided with the unbook $(O,\Phi)$. Then $L$ is the closure of a P-fibered braid.
\end{proposition}


\subsection{The cactus of a branched cover}

We start by explaining a combinatorial structure that is widely used to describe branched covers of the disk $D$ or the sphere $S^2$. We will see that this combinatorial structure is not just analogous, but in fact identical to the combinatorial structures that Rampichini defined in her study of mutually braided open books \cite{rampi}. Our main reference for this subsection is \cite{cactus} both for its clear exposition and its explicit connection to braid groups. For a broader treatment and aspects of the historical development of the study of branched covers of the disk or $S^2$ the reader should consult the references given in \cite{cactus}.

Let $p:\mathbb{C}\to\mathbb{C}$ be a branched cover of degree $n$ with the property
\begin{equation}
\label{eq:limit2}
\lim_{r\to\infty}\arg p(r\rme^{\rmi \chi})=n \chi.
\end{equation}

Note that this is the same as Eq. (\ref{eq:limit}) and in particular, it is satisfied by any monic polynomial of degree $n$. In the literature it is often advantageous to compactify $\mathbb{C}$, so that $p$ becomes a branched cover of the closed unit disk $p:\overline{D}\to \overline{D}$ or a branched cover of the sphere $p:S^2\to S^2$, where the north-pole $(0,0,1)\in S^2\subset\mathbb{R}^3$ (or whichever point is chosen to correspond to the point at infinity of $\mathbb{C}$) is a branch point with $p^{-1}(0,0,1)=(0,0,1)$.

We assume that $p$ has $n-1$ distinct critical values and denote them by $v_j$, $j=1,2,\ldots,n-1$. Furthermore, we assume that $\arg v_j\neq\arg v_i$ if $j\neq i$ and that (possibly after relabelling) $0\leq \arg v_1<\arg v_2<\ldots<\arg v_{n-1}<2\pi$. As in previous sections we denote by $c_j$ the critical point of $p$ with $p(c_j)=v_j$. Now pick $\chi_0=0$. Interpreting $p$ as a branched cover of the closed unit disk $p:\overline{D}\to \overline{D}$, we pick one preimage point $x_1\in p^{-1}(\rme^{\rmi \chi_0})$. Note that $x_1\in \partial \overline{D}$. As we go around $\partial \overline{D}$ we label the other preimage points of $\rme^{\rmi \chi_0}$ by $x_2,x_3,\ldots,x_{n-1}$ in the order in which we encounter them. We write $A_{i}$ for the arc of $\partial \overline{D}$ that only contains $x_i$ and $x_{i+1}$ (modulo $n$) with $\partial A_{i}=\{x_i,x_{i+1 \text{ mod }n}\}$.


Consider a polygon in the complex plane, whose $n-1$ edges are straight lines connecting the critical value $v_j$ to $v_{j+1 \text{ mod } n-1}$. The preimage of this polygon consists of $n$ polygons with $n-1$ edges that are glued together in a specific way. Among the $n-1$ corners of each of these $n$ polygons is exactly one preimage point for each critical value. We can thus label the corners of each polygon by $1,2\ldots,n-1$ according to the index of the corresponding critical value. Some of these preimage points are critical points of $p$, while others merely happen to have the same image as a critical point. Two polygons are glued together at each critical point. 

For each $j=1,2,\ldots,n-1$ the preimage of $\arg(v_j)$ consists of $n-2$ regular leaves (a-arcs) connecting a root to $\partial \overline{D}$ and one singular leaf, i.e., a line between two roots crossing a line between two points on $\partial \overline{D}$ in a saddle point $c_j$. For each $j=1,2,\ldots,n-1$ we define the permutation $\tau_j\defeq (k\ \ell)$ in cycle notation, where the endpoints on $\partial \overline{D}$ of the singular leaf containing $c_j$ lie on the arcs $A_k$ and $A_{\ell}$.
 

This collection of glued polygons with labelled corners and the list of permutations associated to critical points is called the \textit{cactus} of $p$.

This description of $p$ in terms of a list of permutations associated with its critical points is of course fairly standard in topology. Our choice of labelling the points $x_i$ gives a numbering of the sheets of $p$ when considered as a branched cover. The permutation $\tau_j$ associated to the critical point $c_j$ is then simply the permutation of the sheets induced by a small loops around $c_j$. Eq. (\ref{eq:limit2}) implies that the counterclockwise loop along the boundary $\partial \overline{D}$ induces the $n$-cycle $(1\ 2\ \ldots \ n)$, which implies that
\begin{equation}
\label{eq:permprod1}
\tau_1\tau_2\ldots \tau_{n-1}=(1\ n\ \ n-1\ \ldots \ 3\ 2).
\end{equation}

Eq. (\ref{eq:permprod1}) differs from the usual convention (as in \cite{cactus}), which is $\prod_{i=1}^{n-1}\tau_i=(1\ 2\ \ldots \ n)$. If we wanted to agree with this convention, we could have labelled the arcs $A_i$ and points $x_i$ clockwise along the boundary instead of counterclockwise. In order to facilitate a comparison to Rampichini's work, where the boundary arcs are labelled counterclockwise, we prefer to work with this ordering, which results in Eq. (\ref{eq:permprod1}).

\begin{definition}
A \textbf{cactus} $C$ of degree $n$ is a $k$-tuple of permutations $C=(\tau_1,\tau_2,\ldots,\tau_k)$ on the set $\{1,2,\ldots,n\}$ which satisfies the condition
\begin{equation}
\label{eq:permprod}
\prod_{i=1}^k \tau_i=(1\ n\ \ n-1\ \ldots \ 3\ 2).
\end{equation}
\end{definition}

Above we have considered a branched cover with distinct critical values, so that in the definition above we take $k=n-1$. In particular, each critical point of $p$ has ramification index 2, so that $p$ behaves like $z\mapsto z^2$ in a neighbourhood of each critical point. The above definition of the cactus of a polynomial in terms of polygons easily generalises to any branched cover $p:\mathbb{C}\to\mathbb{C}$ satisfying Eq. (\ref{eq:limit2}). In this case, at each critical point $c_j$ of $p$ the number of polygons that meet at $c_j$ is equal to the ramification index of $c_j$, so that the permutation associated to $c_j$ is a cycle, whose length equals the ramification index of $c_j$.

Note that the cactus associated to a branched cover depends on the choice of labelling of arcs $A_i$ of the boundary circle. A different choice for the first arc, say $A_1'=A_j$, implies a cyclic permutation of the labels, i.e., $A_i'=A_{j+i-1\text{ mod }n}$. Hence the cactus $C'=(\tau_1',\tau_2',\ldots,\tau_k')$ that results from the labelling $A_i'$ differs from the cactus $C'=(\tau_1,\tau_2,\ldots,\tau_k)$ that results from the labelling $A_i$ by $\tau_j'=\gamma^{-m}\tau_j\gamma^m$, for some $m\in\{1,2,\ldots,n\}$ where
\begin{equation}
\gamma=(1\ 2\ 3\ \ldots\ n-1\ n).
\end{equation}

Looking at a cactus as in Figure \ref{fig:cactus}a), it is clear that the picture contains some unnecessary information. We don't really need to draw the preimage points of critical values that are not critical points, since we know that they are points labelled by 1 through $n-1$ going counterclockwise around each polygon. Hence, we can represent a cactus by a graph (cf. Figure \ref{fig:cactus}b)), with a node for each root (or, equivalently, each polygon) and an edge for each pair of polygons that are glued together. Each edge therefore represents a critical point $c_j$ and therefore has a corresponding transposition $\tau_j$ associated to it. (For non-simple branched covers the definition of this graph has to be slightly modified, to include nodes for critical points, compare \cite{cactus}.) Given the information of the transpositions $\tau_j$ we could draw lines in the disk that complete the graph to critical level sets of $\arg p$, which are the singular leaves of the singular foliation on $D$ induced by $\arg p$ (cf. Figure \ref{fig:cactus}c)). Note the similarity between Figure \ref{fig:cactus}c) and Figure \ref{fig:movie}b), which depicts the singular leaves of the foliation on a disk induced by a totally braided open book. Recall that in Figure \ref{fig:movie}b) each singular leaf contains exactly one hyperbolic point, which is equipped with a transposition in $S_n$.

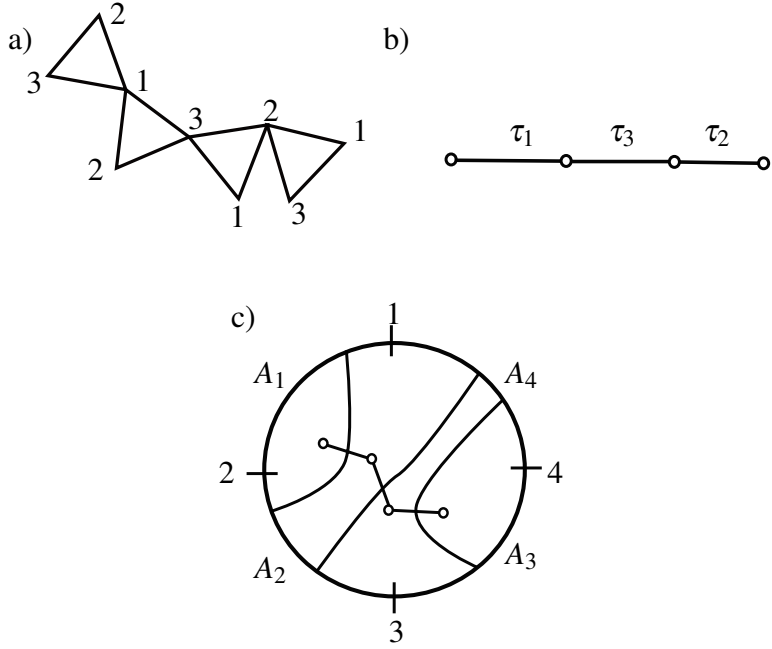
\begin{figure}[h]
\labellist
\Large
\pinlabel a) at 100 900
\pinlabel b) at 1450 900
\pinlabel c) at 900 -100
\pinlabel 2 at 450 1000
\pinlabel 3 at 150 750
\pinlabel 1 at 540 760
\pinlabel 2 at 370 440
\pinlabel 3 at 730 620
\pinlabel 1 at 880 270
\pinlabel 2 at 1000 650
\pinlabel 3 at 1100 280
\pinlabel 1 at 1320 580
\pinlabel $\tau_1$ at 1900 560
\pinlabel $\tau_3$ at 2260 560
\pinlabel $\tau_2$ at 2600 560
\pinlabel $A_1$ at 1000 -300
\pinlabel $A_2$ at 1000 -1000
\pinlabel $A_3$ at 1900 -950
\pinlabel $A_4$ at 1900 -300
\pinlabel 1 at 1440 -80
\pinlabel 2 at 840 -650
\pinlabel 3 at 1450 -1220
\pinlabel 4 at 2020 -650
\endlabellist
\centering
\includegraphics[height=4cm]{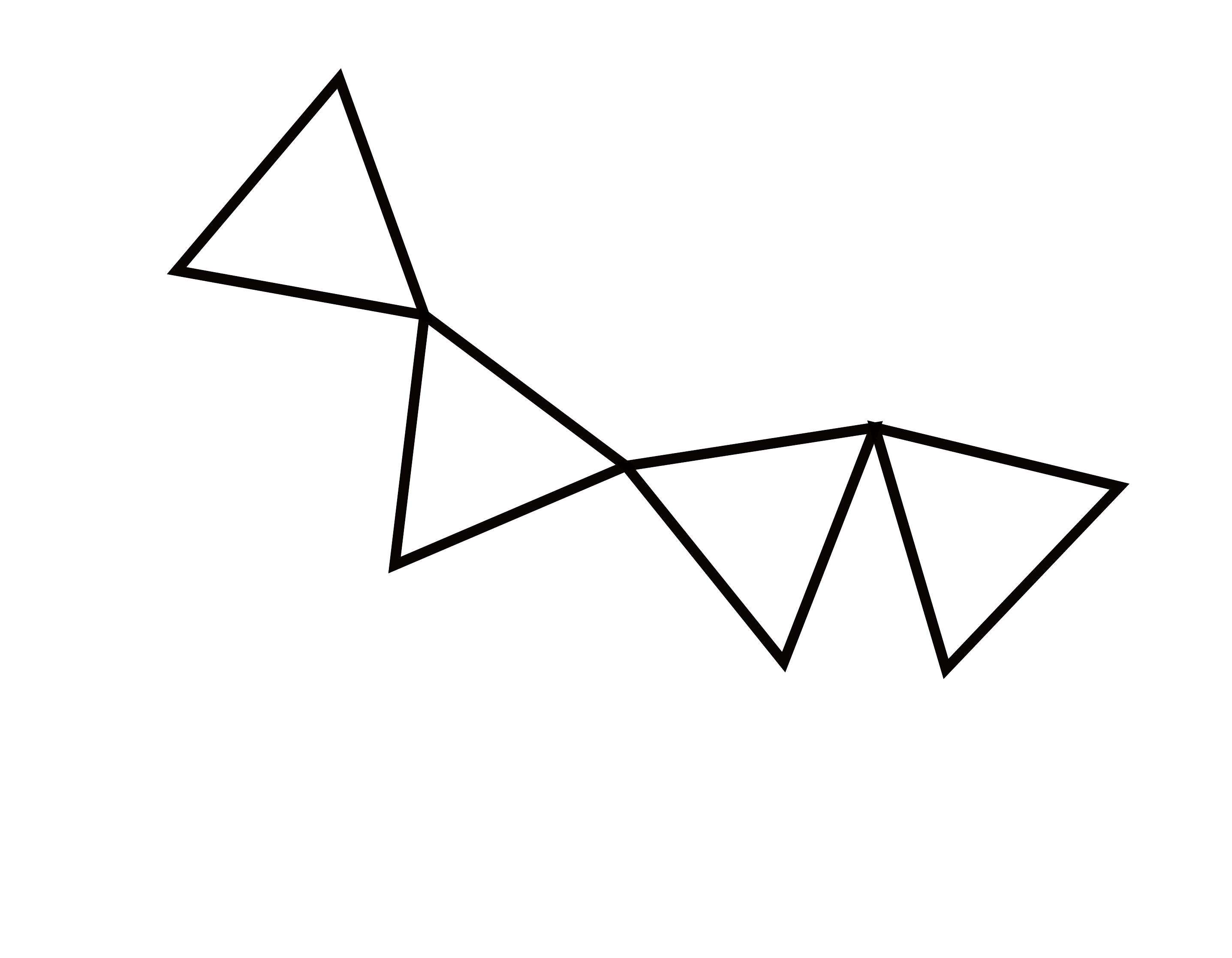}
\includegraphics[height=4cm]{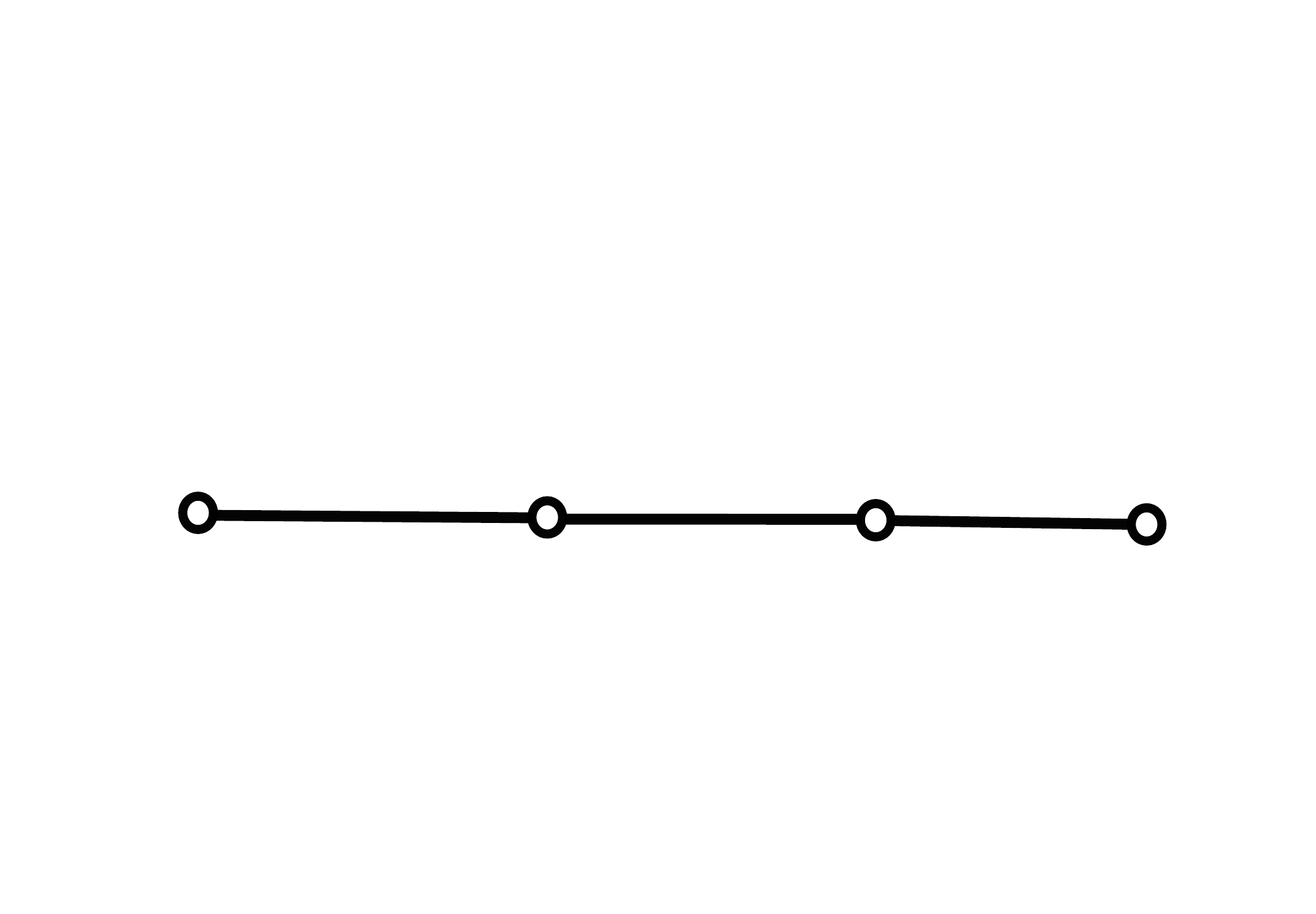}
\includegraphics[height=5cm]{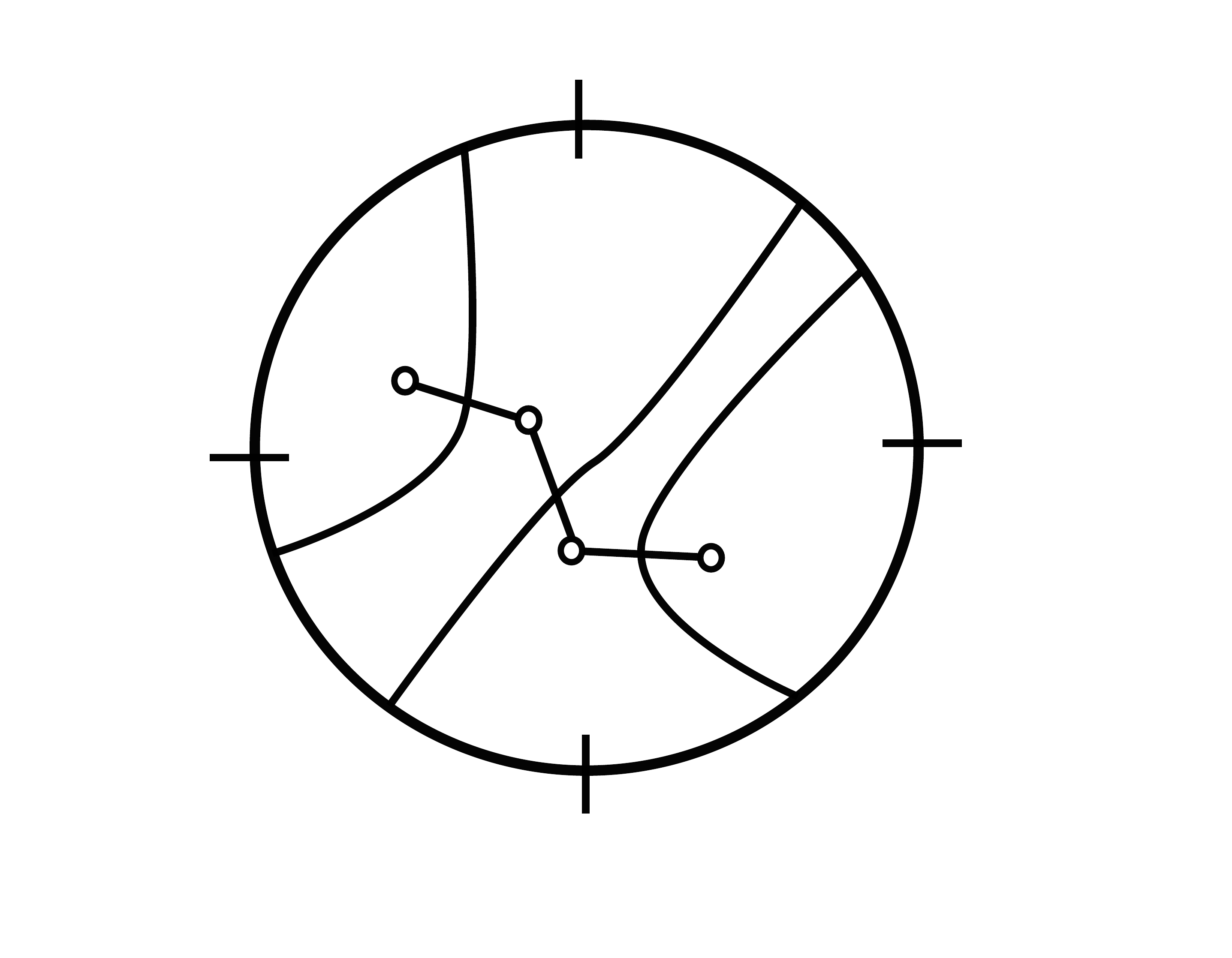}
\caption{The cactus of a simple branched cover $p:\overline{D}\to\overline{D}$ describes combinatorially the singular foliation on $D$ induced by $\arg p$. a) The cactus of a simple branched cover of degree 4. The corners of each triangle are labelled from 1 to 3 in a counterclockwise fashion. For each label $j$ there is a unique corner $c_j$ at which two triangles are glued together. To every $c_j$ we associate a transposition $\tau_j\in S_4$. b) The graph corresponding to the cactus. Each node represents a triangle and each edge a critical point $c_j$. Thus we can label each edge by the corresponding transposition $\tau_j$. c) The singular foliation on $D$ induced by $\arg p$ with $\tau_1=(1\ 2)$, $\tau_2=(3\ 4)$ and $\tau_3=(2\ 4)$. The graph is embedded in $D$ and arcs of $\partial \overline{D}$ are labelled from 1 to 4 in a counterclockwise fashion. The singular leaf through the hyperbolic point $c_j$ connects the two arcs of $\partial \overline{D}$ whose indices are permuted by $\tau_j$. \label{fig:cactus}}
\end{figure}

So far, we have associated to each branched cover, such as a monic polynomial, whose critical values have distinct arguments, a graph and a list of permutations, both of which look suspiciously similar to the combinatorial data describing the open book foliation on a disk, where every singular leaf only contains one hyperbolic point, induced by a totally braid open book. Later we want to find polynomials whose graphs and list of permutations describe the same combinatorial structure as the open book foliations coming from a totally braided open book. For this, we need the converse, i.e., starting with a cactus we want to find a corresponding complex polynomial.

\begin{theorem}[Riemann's existence theorem]
Let $C$ be a cactus of degree $n$ and $v_1,v_2,\ldots,v_k$ be arbitrary complex numbers. Then there exists a complex polynomial $p(z)$ of degree $n$, with $k$ critical values equal to $v_1,v_2,\ldots, v_k$ and with cactus $C$. This polynomial is unique up to an affine change of variable $z\mapsto az+b$, $a,b\in\mathbb{C}$, $a\neq 0$. 
\end{theorem}

We have not defined the cactus of a polynomial with one of the critical values equal to 0 or $\arg(v_i)\geq\arg(v_j)$ for $i<j$, so we will restrict Riemann's existence theorem to non-zero distinct critical values, labelled with increasing argument.

Now, we need to see how the cactus and the list of permutations change, when we vary the critical values $v_1,v_2\ldots,v_{n-1}$. In the following, we consider paths of critical values $v_1(t),v_2(t)\ldots,v_{n-1}(t)$, $t$ going from 0 to 1, with $v_i(0)=v_i$. We denote the end configuration of the path by $\{v_1',v_2',\ldots,v_{n-1}'\}\defeq \{v_1(1),v_2(1),\ldots,v_{n-1}(1)\}$, where the indices are again chosen such that $0\leq \arg v_1'<\arg v_2'<\ldots<\arg v_{n-1}'$. In particular, $v_i'$ is not necessarily equal to $v_i(1)$, but to $v_j(1)$ for some other index $j$.

As long as $0\leq \arg(v_1(t))<\arg(v_2(t))<\ldots<\arg(v_{n-1}(t))<2\pi$, the graph and the permutations $\tau_j$ do not change at all, since their definition is purely topological. We thus need to check what happens when $\arg(v_i(t))=\arg(v_{i+1}(t))$ and when $\arg(v_{n-1}(t))=0\text{ mod }2\pi$ or $\arg(v_1(t))=0\text{ mod }2\pi$.

\begin{figure}[h]
\labellist
\pinlabel 0 at 350 200
\pinlabel 0 at 1650 200
\pinlabel $v_i(0)$ at 1000 400
\pinlabel $v_i(0)$ at 2200 380
\pinlabel $v_{i+1}(t)$ at 880 650
\pinlabel $v_{i+1}(t)$ at 2040 580
\pinlabel $v_i(1)$ at 600 720
\pinlabel $v_i(1)$ at 1800 680
\Large
\pinlabel a) at 100 800
\pinlabel b) at 1350 800
\endlabellist
\centering
\includegraphics[height=4cm]{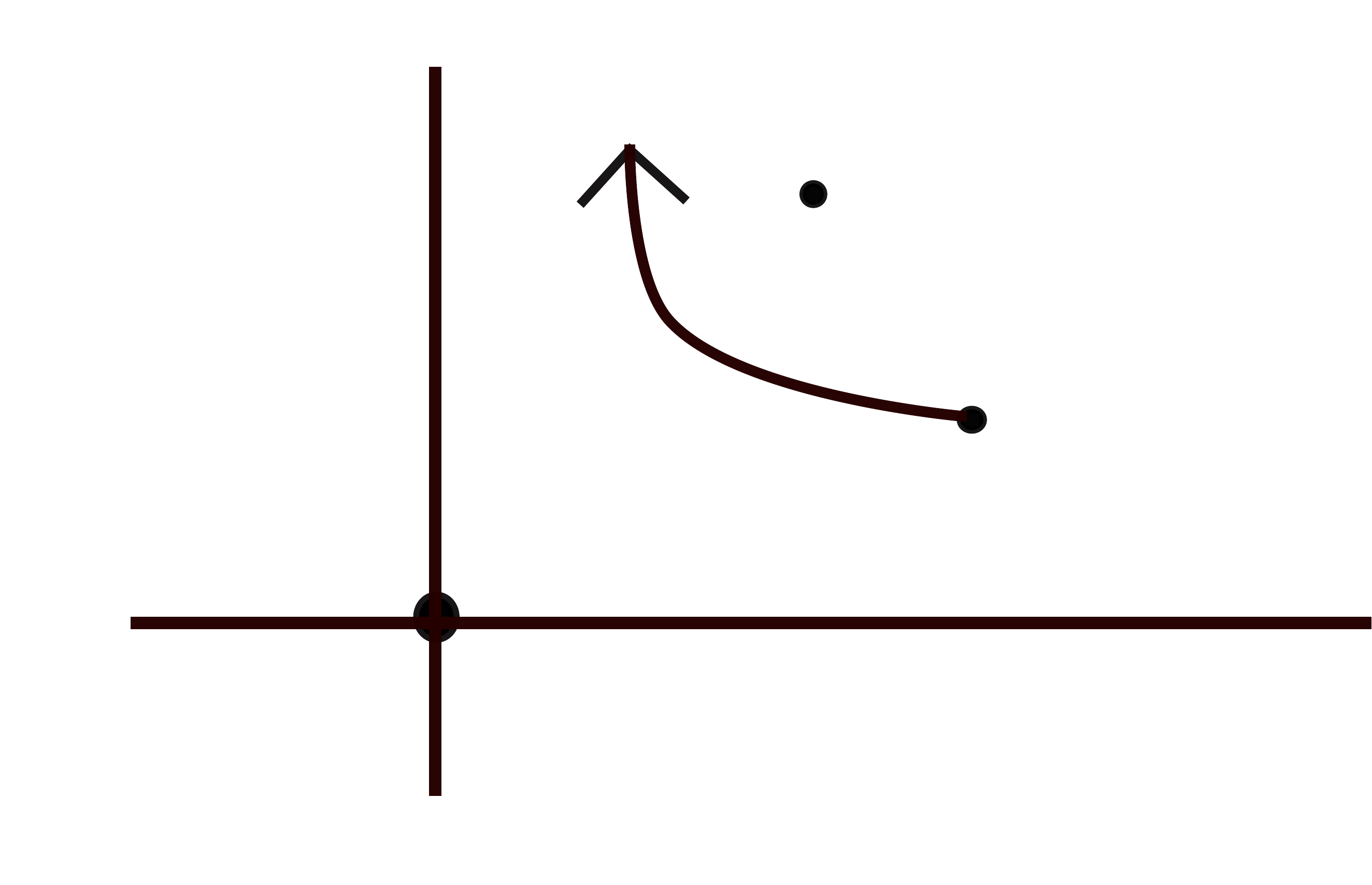}
\includegraphics[height=4cm]{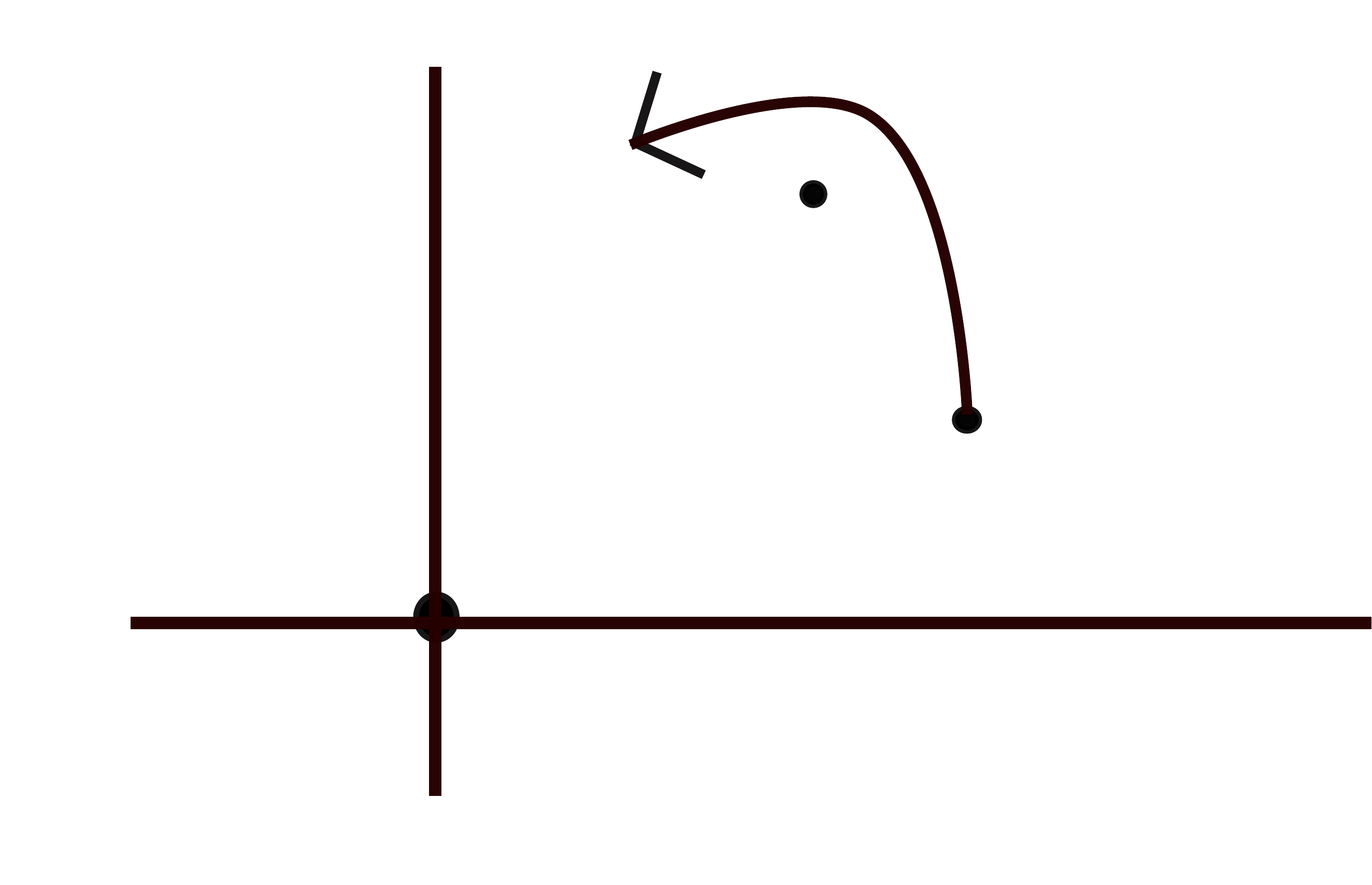}
\caption{Paths of critical values $v_i(t)$ in the complex plane. a) The critical value $v_i(t)$ moves past $v_{i+1}(t)$ with $|v_i(t)|<|v_{i+1}(t)|$. b) The critical value $v_i(t)$ moves past $v_{i+1}(t)$ with $|v_i(t)|>|v_{i+1}(t)|$.  \label{fig:paths}}
\end{figure}

These situations can be created by choosing paths $(v_1(t),v_2(t),\ldots,v_{n-1}(t))$ where only one of the critical values $v_i(t)$ is non-stationary. Suppose we move $v_i$ past $v_{i+1}$ as in Figure \ref{fig:paths}a). In other words, we consider two paths in the plane, $v_i(t)$ and $v_{i+1}(t)$ (the latter being constant), $t$ going from $0$ to $1$, with $v_i(0)=v_i$ and $v_{i+1}(0)=v_{i+1}$, and when $\arg(v_i(t))=\arg(v_{i+1}(t))$ we have $|v_i(t)|<|v_{i+1}(t)|$. In order to maintain our rule by which the critical values are numbered, we then have to change labels, so that $v_i(1)=v_{i+1}'$ and $v_{i+1}(1)=v_i'$. Then the corresponding permutations $\tau_i$ and $\tau_{i+1}$ change as follows (cf. \cite{cactus}):
\begin{align}
\label{eq:tau1}
\tau_i&\mapsto\tau_i'=\tau_{i+1}\nonumber\\
\tau_{i+1}&\mapsto\tau_{i+1}'=\tau_{i+1}\tau_i\tau_{i+1}^{-1}.
\end{align}

Similarly, if we move $v_i$ past $v_{i+1}$ as in Figure \ref{fig:paths}b), i.e., $|v_i(t)|>|v_{i+1}(t)|$ when $\arg(v_i(t))=\arg(v_{i+1}(t))$, then $\tau_i$ and $\tau_{i+1}$ change as follows (cf. \cite{cactus}):
\begin{align}
\label{eq:tau2}
\tau_i&\mapsto\tau_i'=\tau_i^{-1}\tau_{i+1}\tau_i\nonumber\\
\tau{i+1}&\mapsto\tau_{i+1}'=\tau_i.
\end{align}

Note that these are precisely the BKL relations, which describe how the transpositions (interpreted as band generators with appropriate signs) change at crossing points in Rampichini diagrams  (cf. Eq. (\ref{eq:tau3}) and Eq. (\ref{eq:tau4})). Recall that crossing points in Rampichini diagrams occur when a fiber disk intersects a page $F_\varphi$ tangentially in more than one point, or, equivalently, the corresponding open book foliation has a singular leaf with more than one hyperbolic point. This is analogous to the point where $\arg v_i(t)=\arg v_{i+1}(t)=\varphi$, so that $(\arg p)^{-1}(\varphi)$ contains two critical points.

The way that the indices of critical points, critical values and associated permutations are defined implies that we also have to change labels when a critical value, $v_1$ or $v_{n-1}$, passes the line of $\arg=0$. Obviously, the indices of all critical values, critical points and associated permutations have to be cyclically permuted if this happens. For example, if we move $v_{n-1}$ over the line of $\arg=0$, then $v_{n-1}(1)=v_1'$ is now the new critical value with minimal argument between 0 and $2\pi$, $v_1(1)=v_2'$ has the next smallest argument and so on.

Furthermore, the transposition $\tau_{n-1}=(i\ j)$ that was associated to $v_{n-1}$ becomes $\tau_{1}'=(i+1\ j+1)$ modulo $n$. Again, this is exactly the rule that determines the difference between the band generators on the left edge of a Rampichini diagram ($\varphi=0$) and the right edge of a Rampichini diagram $(\varphi=2\pi)$. 

Conversely, moving $v_1$ over the $(\arg=0)$-line from above, induces a cyclic shift of the indices of critical points, critical values and transpositions and $\tau_1=(i\ j)$ becomes $\tau_{n-1}'=(i-1\ j-1)$.

It is easily checked that with these changes the new permutations $\tau_i'$ still satisfy Eq. (\ref{eq:permprod}) as they should.

\subsection{Permutations and singular foliations of the disk}

\begin{lemma}
\label{lem:transprod}
Let $(L,\Psi)$ be an open book in $S^3$. Let $\mathcal{F}$ be an open book foliation of the disk induced by $(L,\Psi)$ with exactly $n$ elliptic singularities, all of which are positive, exactly $n-1$ hyperbolic singularities, and no leaves that are closed loops. We label the hyperbolic points $c_j$, $j=1,2,\ldots,n-1$, such that $0\leq \Psi(c_1)<\Psi(c_2)<\ldots<\Psi(c_{n-1})<2\pi$. We write $\tau_j$ for the transposition associated to the hyperbolic point $c_j$ as defined in Section \ref{sec:mutual}. Then we have $\prod_{j=1}^{n-1}\tau_j=(1\ n\ n-1\ \ldots\ 3\ 2)$. 
\end{lemma}
\begin{proof}
We want to define a function $r:\overline{D}\to[0,1]$ such that $r\rme^{\rmi \Psi}:\overline{D}\to \overline{D}$ is simple branched cover with no branch points on the boundary. The result then follows from the discussion in the previous subsection.
First of all, we define $r(z)=0$ for all elliptic singularities $z\in \overline{D}$ and $r(z)=1$ for all $z\in\partial \overline{D}$. Every singular leaf is cross-shaped, with two opposite endpoints being elliptic singularities and the other two endpoints being on the boundary $\partial \overline{D}$. For every singular leaf we can define $r$ such that it is strictly monotone increasing along each of the arcs, that is, we choose some value $r(c_j)\in(0,1)$ for $r$ at the hyperbolic point $c_j$ and let $r$ be strictly monotone increasing from $0$ to $r(c_j)$ along the arcs whose endpoints are elliptic points and strictly monotone increasing from $r(c_j)$ to 1 on the arcs whose endpoints are on the boundary.

Consider the complement of the singular leaves in the open disk $D$. It consists of a number of open disks, each of which contains exactly one elliptic point. For each of these open disks, we think of its closure in $\overline{D}$ as the image of a square, whose left edge is collapsed to the elliptic point, whose right edge is the unique arc along the boundary $\partial \overline{D}$ and whose vertical edges are formed by singular leaves, possibly collapsed along an arc. Each horizontal line in the square maps to a regular leaf of the open book foliation.

\begin{figure}[h]
\labellist
\small
\endlabellist
\centering
\includegraphics[height=6cm]{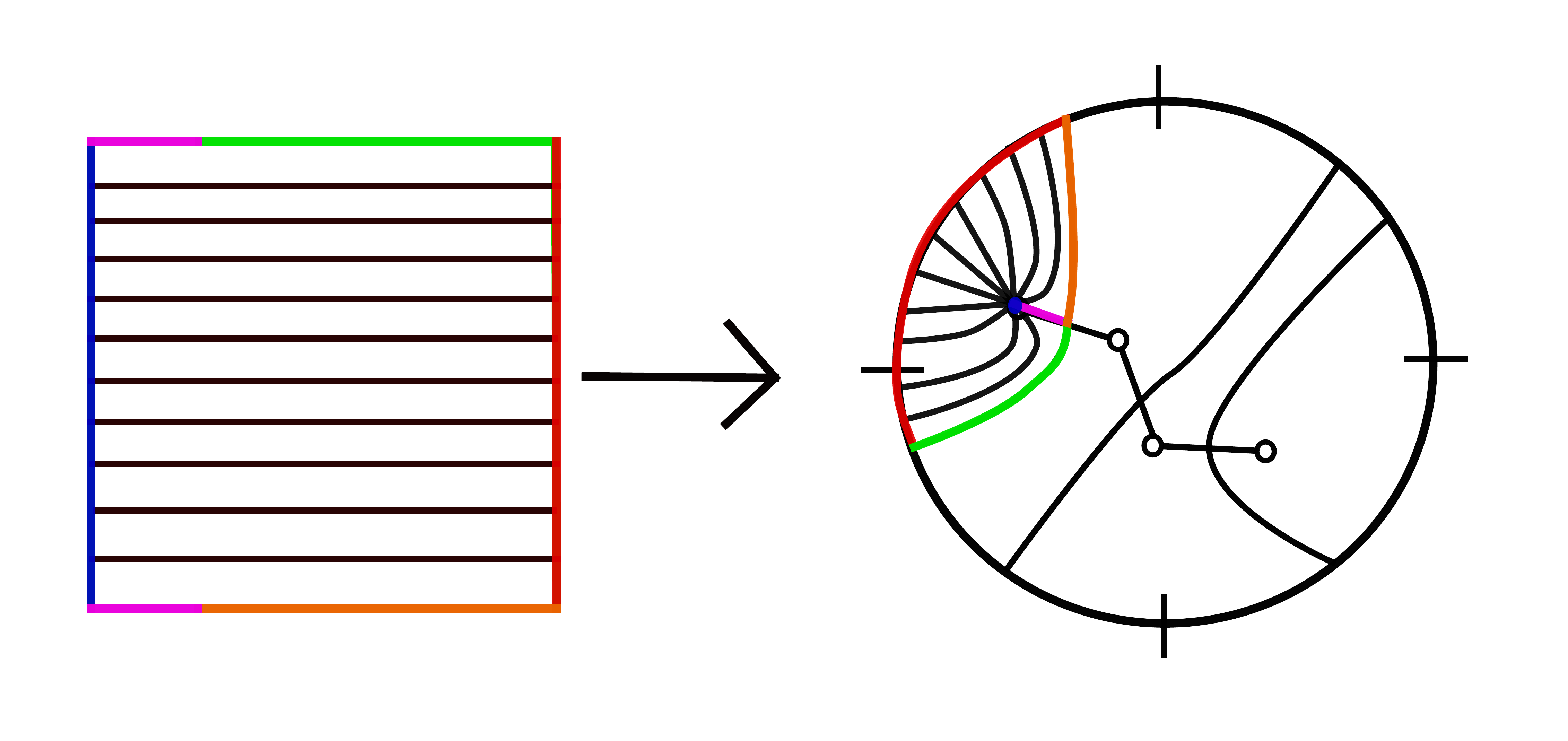}
\caption{Mapping a square into $\overline{D}$. The left edge of the square is mapped to an elliptic point. The horizontal edges are mapped to a singular leaf. The right edge is mapped to $\partial \overline{D}$. Horizontal lines are mapped to regular leaves of the singular foliation. \label{fig:square}}
\end{figure}

The function $r$ defined on the singular leaves, the elliptic points and on $\partial \overline{D}$ thus defines a function on the boundary of the square, which is strictly monotone increasing along the horizontal edges. It is equal to 0 on the left edge and equal to 1 on the right edge. Every pair of strictly monotone increasing functions $f_0$, $f_1:[0,1]\to[0,1]$ with $f_i(0)=0$, $f_i(1)=1$, $i=0,1$, is homotopic via strictly monotone increasing functions $f_s:[0,1]\to[0,1]$ with $f_s(0)=0$, $f_s(1)=1$, $s\in[0,1]$. Therefore, there is a function on the square, which agrees with $r$ on the boundary of the square and which is strictly monotone increasing along each horizontal line.

Mapping the square back into the disk and doing this for all of the components of the complement of the singular leaves, we obtain a function $r:\overline{D}\to [0,1]$ that is strictly monotone increasing along the leaves of the singular foliation. It follows that $r\rme^{\rmi \Psi}:\overline{D}\to \overline{D}$ is a simple branched cover whose branch points are exactly the hyperbolic points of the open book foliation, which proves the lemma by the previous subsection.   
\end{proof}

Let $g_t$ be a loop in the space of polynomials $\widehat{X_n}$. Then we can associate to it a \textit{square diagram}, whose horizontal and vertical edges are coordinate axes representing values of $\varphi\defeq\arg g_t$ and $t$, both going from 0 to $2\pi$. In the square we draw the curves of points $(\varphi,t)$ for which $g_t$ has a critical point $c_j(t)$ with $\varphi=\arg g_t(c_j)$. It follows from the previous subsection that each point on the curve comes with a transposition $\tau_j$, which only change at intersection points of curves and at the right edge of the square. This square diagram satisfies all conditions of Definition \ref{def:rampichini} except 2), that is, the curves in the diagram are not necessarily strictly monotone increasing or strictly monotone decreasing.

The following observation is not necessary for our proof of Theorem \ref{thm:main}, but it is a direct consequence of the discussion above about Rampichini diagram and loops of polynomials.

\begin{proposition}
\label{prop:pfibmut}
Closures of P-fibered braids are bindings of totally braided open books.
\end{proposition}
\begin{proof}
Consider the loop of polynomials associated to a given P-fibered braid. As discussed above the corresponding square diagram satisfies all but one of the defining conditions of a Rampichini diagram. The fact that the loop of polynomials corresponds to a P-fibered braid can be expressed as $\tfrac{\partial \arg v_j(t)}{\partial t}\neq 0$ for all critical values $v_j(t)=g_t(c_j(t))$ and all values of $t\in[0,2\pi]$, which is equivalent to the one missing condition, that the curves in the square diagram are strictly monotone increasing or strictly monotone decreasing.

Rampichini proved that every Rampichini diagram gives rise to a totally braided open book. Note that we can read off band words of the fiber surface in band generators from the diagrams, so that the obtained totally braided open book is indeed identical to the one given by the argument of the loop of polynomials. 
\end{proof}

There are more direct ways of proving Proposition \ref{prop:pfibmut}. It will for example follow directly from Theorem \ref{thm:main}. The important aspect of the proposition is not so much in the statement itself, but in the observation that Rampichini diagrams are the same thing as the square diagrams that one can associate to loops of polynomials, provided that the corresponding braid is P-fibered.

Another consequence of this observation is that condition 9) in the definition of a Rampichini diagram (Definition \ref{def:rampichini}) could be changed to the following:
\begin{enumerate}[label=\arabic*)$'$]
\setcounter{enumi}{8}
\item Along the $(t=0)$-edge the transpositions $\tau_i$ (indexed with increasing $\varphi$) satisfy $\prod_{i=1}^{n-1}\tau_i=(1\ n\ n-1\ n-2\ \ldots \ 3 \ 2)$.
\end{enumerate}






\subsection{The proof of Proposition \ref{prop:42}}

We have seen in the previous subsection that P-fibered braids give rise to totally braided open books via a correspondence between square diagrams associated to loops of polynomials on the one hand and Rampichini diagrams associated to mutually braided open books on the other. This same correspondence along with Riemann's existence theorem will also help us prove the converse of that statement.

\begin{theorem}
\label{thm:42}
Let $(L,\Psi)$ be a totally braided open book of degree $n$. Let $G$ denote the corresponding Rampichini diagram. Then there is a loop of polynomials $g_t\subset \widehat{X_n}$ whose argument induces a fibration of $S^3\backslash cl(B)$ over the circle, where $B$ is the braid formed by the roots of $g_t$, and whose square diagram is equal to $G$. In particular, $B$ is a P-fibered braid and closes to $L$.
\end{theorem} 

\begin{proof}
Consider the Rampichini diagram $G$ and the induced singular foliations on the fiber disks of the unbook. For each $t\in[0,2\pi]$ we denote the hyperbolic points of these foliations by $c_j(t)$, $j=1,2,\ldots,n-1$, such that $0\leq\Psi(c_1(t))<\Psi(c_2(t))<\ldots<\Psi(c_{n-1}(t))<2\pi$ for all values of $t$ for which there are no intersections between curves in $G$. We denote the transposition associated to a hyperbolic point $c_j(t)$ by $\tau_j(t)$. Note that this convention of labelling the critical points (which differs from that in \cite{bode:adicact, bodesat}) requires us to relabel the critical points at each value of $t$ for which $\Psi(c_i(t))=\Psi(c_{i+1}(t))$, i.e., for each intersection point in the Rampichini diagram.

We want to define complex numbers $v_j(t)$, $t\in[0,2\pi]$, which will play the role of the critical values of the loop of complex polynomials. We set $\arg(v_j(t))=\Psi(c_j(t))$ and define the modulus of $v_j(t)$ such that the BKL relations on intersection points of curves in $G$ are consistent with the movements of $v_j(t)$ in the complex plane. 

This means that if the transpositions $\tau_i(t)$ and $\tau_{i+1}(t)$ change at an intersection point $(\varphi_0,\tau_0)$ according to Eq. (\ref{eq:tau1}), then $|v_i(t_0)|<|v_{i+1}(t_0)|$. If the transpositions change according to Eq. (\ref{eq:tau2}), then $|v_i(t_0)|>|v_{i+1}(t_0)|$. Note that the values of $|v_i(t)|$ only matter at the values of $t$, for which $\arg v_i(t)=\arg v_{i+1}(t)$.

We now have for every value of $t$ a set of `critical values' $v_j(t)$ and a corresponding set of transpositions $\tau_j(t)$, which satisfy $\prod_{j=1}^{n-1}\tau_j(t)=(1\ n\ n-1\ \ldots\ 3\ 2)$ for all $t\in[0,2\pi]$ by Lemma \ref{lem:transprod}.

By Riemann's existence theorem there is polynomial $p_0$ with critical values $v_j(0)$, $j=1,2,\ldots,n-1$, with cactus $\tau_j(0)$, $j=1,2,\ldots,n-1$. The critical values $(v_1(t),v_2(t),\ldots,v_{n-1}(t))$ form a loop in the space of critical values $V_n$. We can now choose a loop $a_0(t)$ in $\mathbb{C}$ with $a_0(t)\neq v_j(t)$ for all $j=1,2,\ldots,n-1$ and $t\in[0,2\pi]$ based at $a_0(0)=a_0(2\pi)$, the constant term of $p_0$. For simplicity we choose the constant loop $a_0(t)=0$ for all $t\in[0,2\pi]$. Thus $((v_1(t),v_2(t),\ldots,v_{n-1}(t)),a_0(t)$) is a loop in $\widehat{V_n}$, see Eq. (\ref{eq:whvn}). By Theorem \ref{thm:crit} this loop lifts to a path $p_t$ in $\widehat{X_n}$, the space of monic polynomials of degree $n$ with distinct roots and disctinct critical values $(v_1(t),v_2(t),\ldots,v_{n-1}(t))$ and constant term different from all $v_j(t)$ The base point of this lifted path is $p_0$.

Note that the lifted path is not necessarily a loop. Its endpoint $p_{2\pi}$ is some other polynomial with the same critical values and the same constant term as $p_0$. We have seen in the previous two subsections how the cactus of a polynomial changes, when the critical values are varied. Recall that the critical values $v_j(t)$ were constructed precisely such that the cactus of polynomial $p_t$ is described by the permutations $\tau_j(t)$ in the given Rampichini diagram. In particular, $p_0$ and $p_{2\pi}$ have the same cactus.

By Riemann's existence theorem $p_0$ and $p_{2\pi}$ differ by an affine transformation $z\mapsto az+b$ for some $a,b\in\mathbb{C}$, $a\neq 0$, say $p_{2\pi}(z)=p_0(az+b)$. However, we know that both $p_0$ and $p_{2\pi}$ are monic, so that $a$ has to be an $n$th root of unity. Note that the $n$ different choices for $a$ correspond to the $n$ different choices for an arc $A_1$ on $\partial \overline{D}$. We have made an arbitrary choice for an arc $A_1$ for $p_0$ and by continuity this choice remains the same for every $p_t$, since $p_t$ is monic for all $t\in[0,2\pi]$. Hence $a=1$.

Now let $\gamma:[0,2\pi]\to\mathbb{C}\backslash\{v_1(0),v_2(0),\ldots,v_{n-1}(0)\}$ be a path starting at $0$ and ending at $b$. Then the composition 
\begin{equation}
g_{t}\defeq 
\begin{cases}
p_{2t} &\text{ if }t\in[0,\pi],\\
p_{2\pi}(z-\gamma(2(t-\pi)) &\text{ if }t\in[\pi,2\pi]
\end{cases}
\end{equation}
of $p_t$ and $p_{2\pi}(z-\gamma)$ is a loop in $\widehat{X_n}$, the space of monic polynomials of degree $n$ with distinct roots, distinct critical values and constant term distinct from any the critical values. Furthermore, its square diagram is identical to the given Rampichini diagram $G$, except that in the top part, which corresponds to $p_{2\pi}(z-\gamma(s))$, the curves are vertical. From this diagram we can already deduce that the roots of $g_t$ form the closed braid that is identical to $L$, the binding of the given mutually braided open book.

In terms of the critical values, this can be visualised as follows. We write $w_j(t)$ for critical values of $g_t$, labelled such that $w_j(t)=v_j(t)$ for $t\in[0,\pi]$ and $w_j(t)=v_j(2\pi)$ for $t\in[\pi,2\pi]$. For the first half of $g_t$ the critical values $w_j(t)$, $j=1,2,\ldots,n-1$, form a braid with $\tfrac{\partial \arg w_j(t)}{\partial t}\neq0$ and for the second half the critical values are simply constant. The composition of these two braids is a loop in $V_n$, which can easily be deformed to a loop that satisfies $\tfrac{\partial \arg w_j(t)}{\partial t}\neq0$ everywhere, not only on the first part.

This deformation can be performed inside a tubular neighbourhood of the strands $w_j(t)$. In particular, the critical values remain distinct from the constant term of $g_t$ and therefore the deformation is a homotopy in $\widehat{V_n}$, which lifts to a homotopy of the loop $g_t$ in $\widehat{X_n}$ by Corollary \ref{cor:lifting}. The resulting loop in the space of polynomials has by construction a square diagram that is equal to the given Rampichini diagram $G$ and it corresponds to a P-fibered braid, since $\tfrac{\partial \arg w_j(t)}{\partial t}\neq0$.
\end{proof}




We know that B1)$\iff$B2) by Rampichini. Furthermore, we have proved that B4)$\implies$B3) (by Lemma \ref{lemma43}), that B3)$\implies$B1) (by Lemma \ref{lem:liftex}) and that B2)$\iff$B4) (by Proposition \ref{prop:pfibmut} and Theorem \ref{thm:42}). We conclude that B1), B2), B3) and B4) are all equivalent to each other as stated in Theorem \ref{thm:main}.

To a totally braided open book we can associate its \textit{boundary braid}, which is simply the braid that is the boundary of the fibers, which is clearly a braid relative to the same braid axis $O$ as the open book. Note that the isotopy of the open book in Theorem \ref{thm:main} restricts to a braid isotopy of the binding $L$ in the complement $S^3\backslash O$ of the braid axis $O$, so that we have the following corollary.

\begin{corollary}
A braid $B$ is P-fibered (up to conjugation) if and only if it is the boundary braid of a totally braided open book in $S^3$.
\end{corollary}

Therefore, the algorithm by Morton and Rampichini \cite{mortramp} that decides whether a given band word is the boundary braid of a totally braided open book is also an algorithm that checks P-fiberedness of braids given in terms of band generators.

We can use Riemann's existence result, Theorem \ref{thm:crit} and the proof of Corollary \ref{cor:lifting} to give a proof of the following fact.

\begin{proposition}
\label{prop:graphcount}
There are exactly $n^{n-2}$ cacti $\mathcal{C}=(\tau_1,\tau_2,\ldots,\tau_{n-1})$ of degree $n$ such that $\tau_i$ is a transposition for all $i=1,2,\ldots,n-1$.
\end{proposition}
\begin{proof}
For the case of $n=2$, the statement is obvious, since $\tau_1=(1\ 2)$ is the only cactus for $n=2$. In the following, we assume that $n>2$.

Recall that $\theta_n:X_n\to V_n$ maps a polynomial to its set of critical values. By Theorem \ref{thm:crit} it is a covering map of degree $n^{n-1}$. Fix a point in $V_n$ and consider the $n^{n-1}$ polynomials in its fiber. 


Riemann's existence result says that for $(v_1,v_2,\ldots,v_{n-1})$ and every fixed cactus $\mathcal{C}$ there is a polynomial $p$ with the corresponding set of critical values and cactus $\mathcal{C}$, which is unique up to affine transformation $z\mapsto az+b$. Since we only consider monic polynomials with constant term equal to zero, there are at most $n^2$ polynomials that can be obtained from such a transformation with values of $(a,b)=(\rme^{\rmi 2\pi k/n},\rme^{\rmi 2\pi k/n}z_j)$, $j,k=1,2,\ldots,n$, where $z_j$, $j=1,2,\ldots,n$ are the roots of $p$. 

In fact, these $n^2$ transformations all result in different polynomials, so that we obtain exactly $n^2$ polynomials. This can be seen as follows. Suppose that there is an affine transformation $z\mapsto az+b$ that maps a polynomial $p$ in the fiber to itself. Then the transformation must preserve the critical points $c_j$ of $p$ as a set, so for every $j=1,2,\ldots,n-1$ there is a $k\in\{1,2,\ldots,n-1\}$ with $ac_j+b=c_k$. Since the transformation maps $p$ to itself, we obtain $p(c_k)=c_j$, which implies $k=j$, since all critical values of $p$ are distinct. Therefore $ac_j+b=c_j$ for all $j$, which is a contradiction, since $n>2$ and the $n-1$ critical points are distinct.

Now recall that we can associate to every polynomial not just one cactus, but $n$ different cacti corresponding to the $n$ different counterclockwise labellings of the boundary arcs $A_i$, so that two polynomials have a cactus in common if and only if they have the same $n$ cacti. Therefore, the family of $n^2$ polynomials above shares its set of $n$ associated cacti. The total number of different cacti is therefore the number of polynomials in the fiber, $n^{n-1}$, divided by the number of polynomials that share the same set of $n$ cacti, $n^2$, multiplied by $n$, which results in $n^{n-2}$.
\end{proof}


Proposition \ref{prop:graphcount} is a special case of the more general formula proved in Theorem 3.2 of \cite{goulden}. We present our proof, not because it is easier, more direct or more intuitive (it is perhaps none of those), but precisely because of how roundabout it is. The proposition is stated in purely combinatorial and group theoretic terms and yet our proof is based on topological properties of spaces of complex polynomials and the formula itself stems from the degree of the covering map in Theorem \ref{thm:crit}, which itself is a consequence of Bezout's Theorem.

\section{Some braided open books}\label{sec:23}

We now know that the four different braiding characterisations B1), B2), B3) and B4) are equivalent. This immediately provides us with some examples of braided open books and various consequences. Rudolph showed that the family of T-homogeneous braids, a generalisation of the family of homogeneous braid, gives rise to totally braided open books, that is, open books that are mutually braided with the unbook \cite{rudolph2}. As a corollary we obtain
\begin{corollary}\label{cor:thomo}
T-homogeneous braids are P-fibered braids (up to conjugation).
\end{corollary}
\begin{corollary}\label{cor:thomo2}
The closures of T-homogeneous braids and their fiber surfaces can be obtained from the unbook via a sequence of $\pi$-symmetric Hopf plumbings and deplumbings for some simple branched cover $\pi:S^3\to S^3$.
\end{corollary}  

\begin{corollary}
Let $B$ be a T-homogeneous braid. Then the closure of $B^2$ is real algebraic.
\end{corollary}

Likewise, the satellite and twisting operation defined \cite{bodesat} that can be used to construct new P-fibered braids from known ones, carry over to the setting of generalised exchangeable braids, mutually braided open books and $\pi$-symmetric Hopf plumbings/deplumbings.

In this final section we produce further examples of families of braided open books.

\subsection{Over- and underpasses}

Every banded surface can be represented by a \textit{ladder diagram} as in Figure \ref{fig:ladders}a), where each of the $n$ disks is represented by a vertical line and each half-twisted band corresponds to a horizontal line connecting the vertical lines that are connected by the band. Every horizontal line $h$ comes with a sign $sign(h)\in\{\pm\}$ that reflects the sign of the half-twisted band. For our purposes we take the lines in the diagram to be oriented such that the vertical lines point upwards and the horizontal lines go from left to right. In particular, the `starting point' of a horizontal line refers to its leftmost point. Furthermore, we define for each horizontal line $h$ going from the $i$th disk to the $j$th disk, i.e., corresponding to the band generator $a_{i,j}^{sign(h)1}$, the tuple $\tau(h)=(i,j)$.


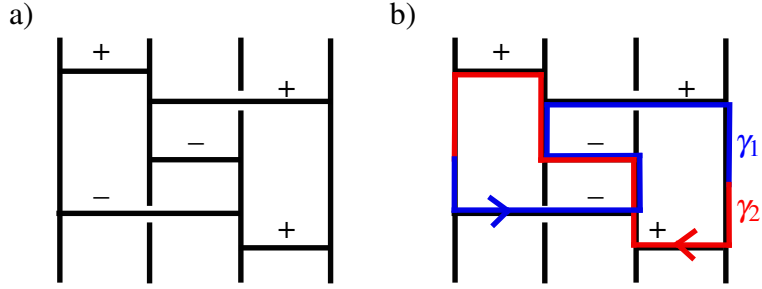
\begin{figure}[h]
\labellist
\Large
\pinlabel a) at 100 900
\pinlabel b) at 1150 900
\pinlabel + at 320 800
\pinlabel -- at 580 540
\pinlabel + at 830 700
\pinlabel + at 830 310
\pinlabel -- at 320 390
\pinlabel + at 1420 800
\pinlabel -- at 1680 540
\pinlabel + at 1930 700
\pinlabel + at 1850 310
\pinlabel -- at 1680 390
\pinlabel {\color{blue}$\gamma_1$} at 2100 540
\pinlabel {\color{red}$\gamma_2$} at 2100 360
\endlabellist
\centering
\includegraphics[height=4.5cm]{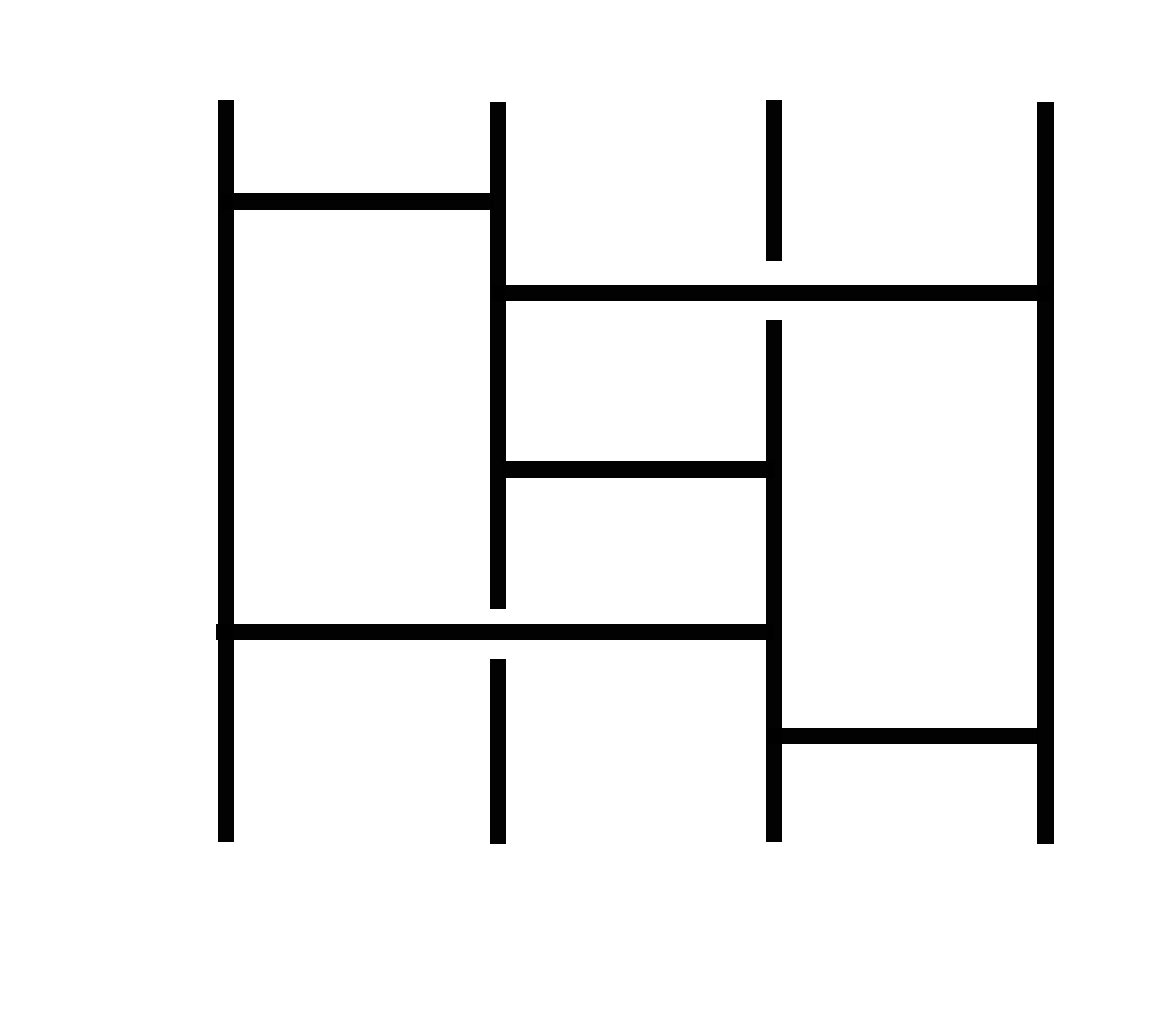}
\includegraphics[height=4.5cm]{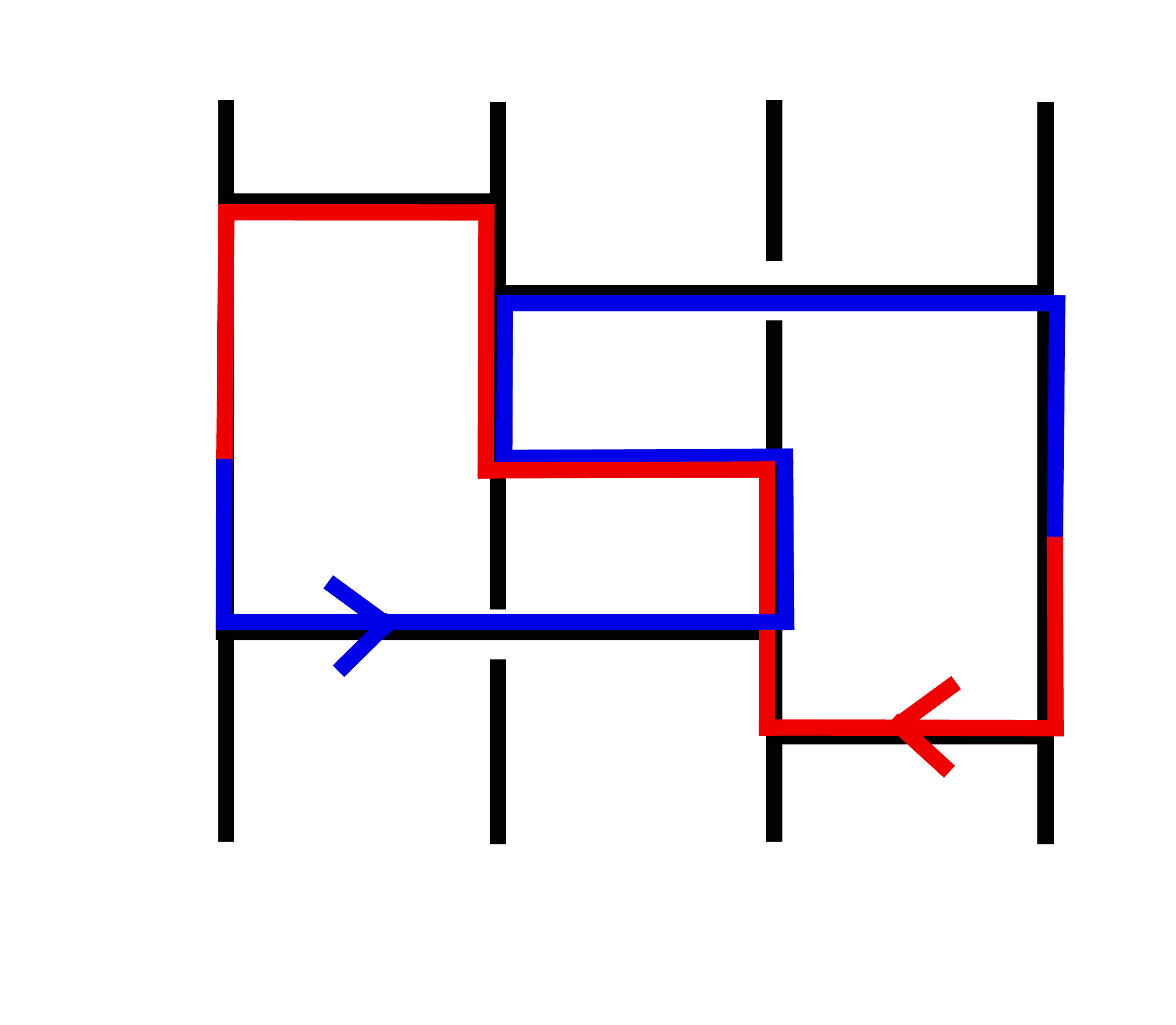}
\caption{a) A ladder diagram representing the banded surface $a_{3,4}a_{1,3}^{-1}a_{2,3}^{-1}a_{2,4}a_{1,2}$. b) The ladder diagram has an overpass $\gamma_1$ and an underpass $\gamma_2$. \label{fig:ladders}}
\end{figure}

A simple path $\gamma:[0,1]\to \mathcal{D}$ in a ladder diagram $\mathcal{D}$, with the orientation induced by that on the unit interval, consists of a sequence of horizontal $h_i$ and vertical segments $v_i$, $i=1,2,\ldots,m$, so that $\gamma$ is the concatenation $\gamma=h_1v_1h_2v_2\ldots h_mv_m$, where $h_1$ and $v_m$ are potentially constant, i.e., only a point. We define $o(h_i)=+1$ if the orientation of a horizontal line in $\mathcal{D}$ matches the orientation of $\gamma$ on that segment $h_i$, and $o(h_i)=-1$ if the two orientations are opposite. Analogously, we define $o(v_i)$. If $h_1$ or $v_m$ is constant, we leave $o(h_1)$ or $o(v_m)$ undefined.

For a diagram $\mathcal{D}$ as in Figure \ref{fig:ladders} with $n$ vertical lines we can ask if paths $\gamma_1,\gamma_2:[0,1]\to \mathcal{D}$, $\gamma_1=h_1^1v_1^1h_2^1v_2^1\ldots h_{m}^1v_{m}^1$, $\gamma_2=h_1^2v_1^2h_2^2v_2^2\ldots h_{n-1}^2v_{n-1}^2$ with all of the following properties exist:
\begin{itemize}
\item $\gamma_1(0)=\gamma_2(1)$ lies on the leftmost vertical line and $\gamma_1(1)=\gamma_2(0)$ on the rightmost vertical line.
\item Neither $\gamma_1$ nor $\gamma_2$ cross themselves (but they may cross each other).
\item On segments where $\gamma_1$ is vertical, it does not cross any horizontal lines, nor does it contain starting points of any horizontal lines.
\item If $o(h_i^1)=1$ for some $i\in\{1,2\ldots,m_1\}$, then $o(v_{i-1})=sign(h_i)1$.
\item If $o(h_i^1)=-1$ for some $i\in\{1,2\ldots,m_1\}$, then $o(v_{i})=-sign(h_i)1$.
\item For all $i=1,2,\ldots,n-1$ we have $\tau(h_i^2)=(n-i,n-i+1)$ .
\end{itemize}

\begin{definition}
Paths $\gamma_1$ and $\gamma_2$ as above are called an \textbf{overpass} and \textbf{underpass} of $\mathcal{D}$, respectively.
\end{definition}

The defining properties of over- and underpasses might seem a bit arbitrary. When visualised on a braided surface however, it becomes apparent that the union of an overpass and an underpass gives us the kind of braid axis that we are looking for in a generalised exchangeable braid.

\begin{lemma}
\label{lem:over}
Let $F$ be a braided surface of degree $n$ with boundary braid $B$ with closure $L$. If the corresponding ladder diagram $\mathcal{D}$ has an overpass and an underpass, then there is a braid axis $O$ that lies in a neighbourhood of $F\cup B$ and that can be taken to be positively transverse to all level sets of any circle-valued map $\Psi:S^3\backslash L\to S^1$ that behaves like an open book in a neighbourhood of $L$ and that satisfies $\Psi^{-1}(\chi)=F$ for some $\chi\in S^1$. In particular, if $F$ is a Bennequin surface and $L$ is fibered, the existence of an overpass and an underpass implies that $B$ is a generalised exchangeable braid and hence that the corresponding open book can be braided.
\end{lemma}

\begin{proof}
We prove this lemma by drawing the pictures for the horizontal segments of the over- and underpass. For this we first interpret the overpass as a curve on the surface $F$ and the underpass as a curve on $F^+\cup N(B)$ (cf. Figure~\ref{fig:bigexample}a)), where $F^+$ is a positive push-off of $F$, e.g., if $F$ is a regular level set $\Psi^{-1}(\chi)$, $\chi\in S^1$ of some circle-valued Morse function $\Psi:S^3\backslash L$ that behaves like an open book in a tubular neighbourhood of $L$, then $F^+=\Psi^{-1}(\chi+\varepsilon)$ for some small $\varepsilon$, and $N(B)$ is a tubular neighbourhood of $B$.

We claim that the overpass (interpreted as a curve in $F$) can be isotoped to lie above all strands of $B$, meaning above the diagram plane, and the underpass can be isotoped to lie below the diagram plane, which justifies their names.

\begin{figure}[h]
\labellist
\Large
\pinlabel a) at 200 1700
\pinlabel b) at 200 -200
\endlabellist
\centering
\includegraphics[height=8cm]{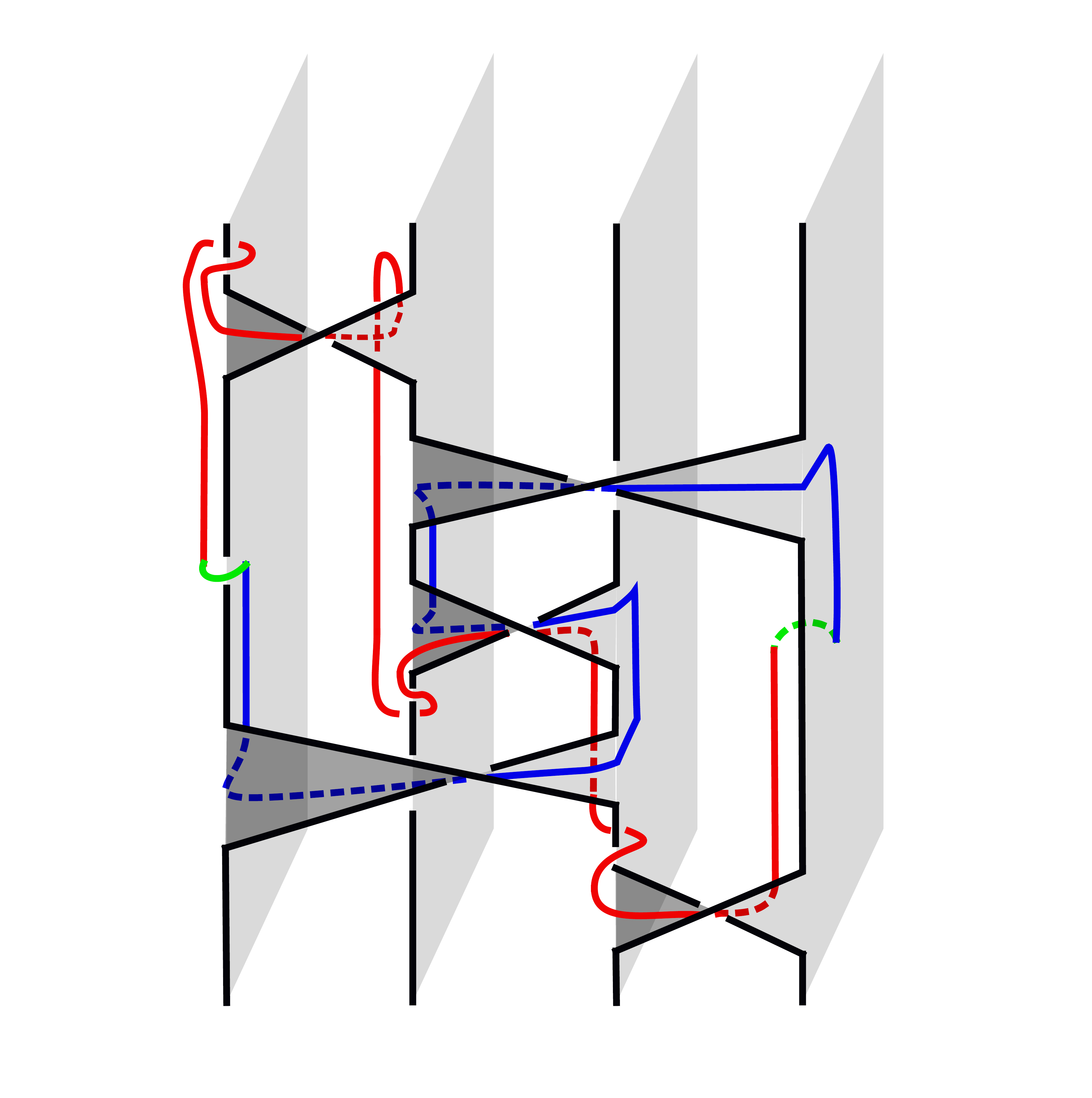}
\includegraphics[height=8cm]{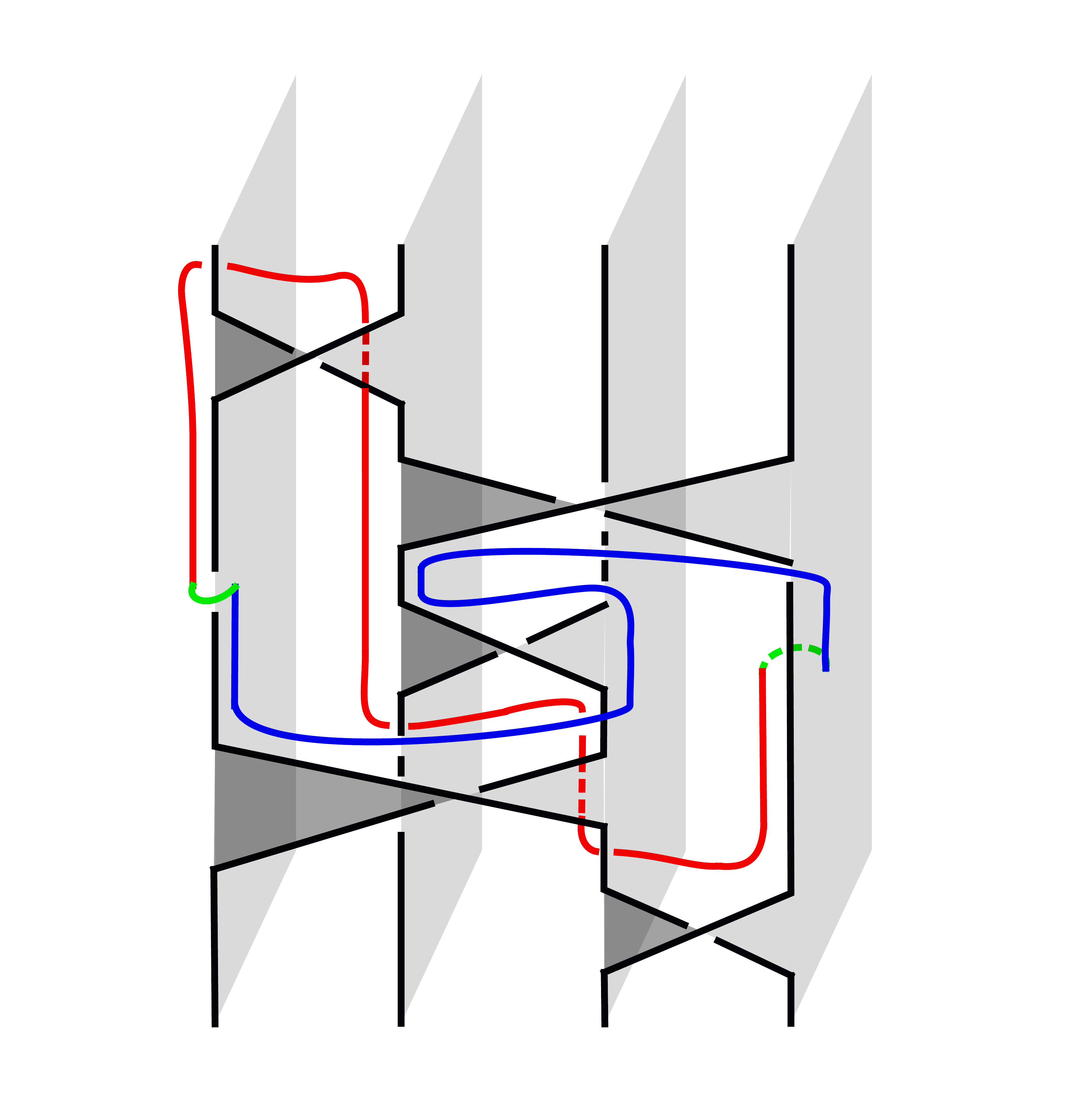}
\caption{a) The banded surface $F$ corresponding to the ladder diagram in Figure \ref{fig:ladders} with curves $\gamma_1$ and $\gamma_2$ representing the overpass and underpass. The endpoints of $\gamma_1$ and $\gamma_2$ are joined by arcs in a tubular neighbourhood of the braid $B=\partial F$. b) After an isotopy it becomes apparent that the union of $\gamma_1$, $\gamma_2$ and the two arcs in $N(B)$ is a braid axis for $B$. \label{fig:bigexample}}
\end{figure}

We interpret an underpass $\gamma_2$ in the ladder diagram $\mathcal{D}$ as a curve in $F^+\cup N(B)$ as follows. The behaviour of the underpass at a positive and negative crossing, i.e., a positive and a negative horizontal line in $\mathcal{D}$, is shown in Figure \ref{fig:underpass}a) and c), respectively. In both figures the red underpass is oriented from right to left. This means in particular, that in both cases the twisting around the leftmost strand is such that the underpass $\gamma_2$ is positively transverse to the fibers of any $S^1$-valued map that satisfies the boundary condition of an open book on $N(B)$.

In Figure \ref{fig:underpass}a) the underpass starts at the top right corner and ends at the top left corner of the picture. If $\gamma_2$ in the ladder diagram $\mathcal{D}$ enters the crossing from the bottom right and/or travels down after traversing the horizontal line corresponding to the band, the curve in Figure \ref{fig:underpass}a) is continued on $F^+$ as seen for example in the band corresponding to $a_{1,2}$ in Figure \ref{fig:bigexample}a). Thus the simple curve $\gamma_2$ in $F^+\cup N(B)$ can cross itself (in the diagram), even if the underpass $\gamma_2$ in the ladder diagram is not allowed to. However, all of these crossings can be easily removed by a Reidemeister move of type 1.

The analogous extension of the curve $\gamma_2$ in Figure \ref{fig:underpass}c) if the underpass enters the crossing from the top right and/or travels upwards after transversing the horizontal line corresponding to the band is straightforward.

\begin{figure}[h]
\labellist
\Large
\pinlabel a) at 100 2100
\pinlabel b) at 1150 2100
\pinlabel c) at 100 900
\pinlabel d) at 1150 900
\endlabellist
\centering
\includegraphics[height=15cm]{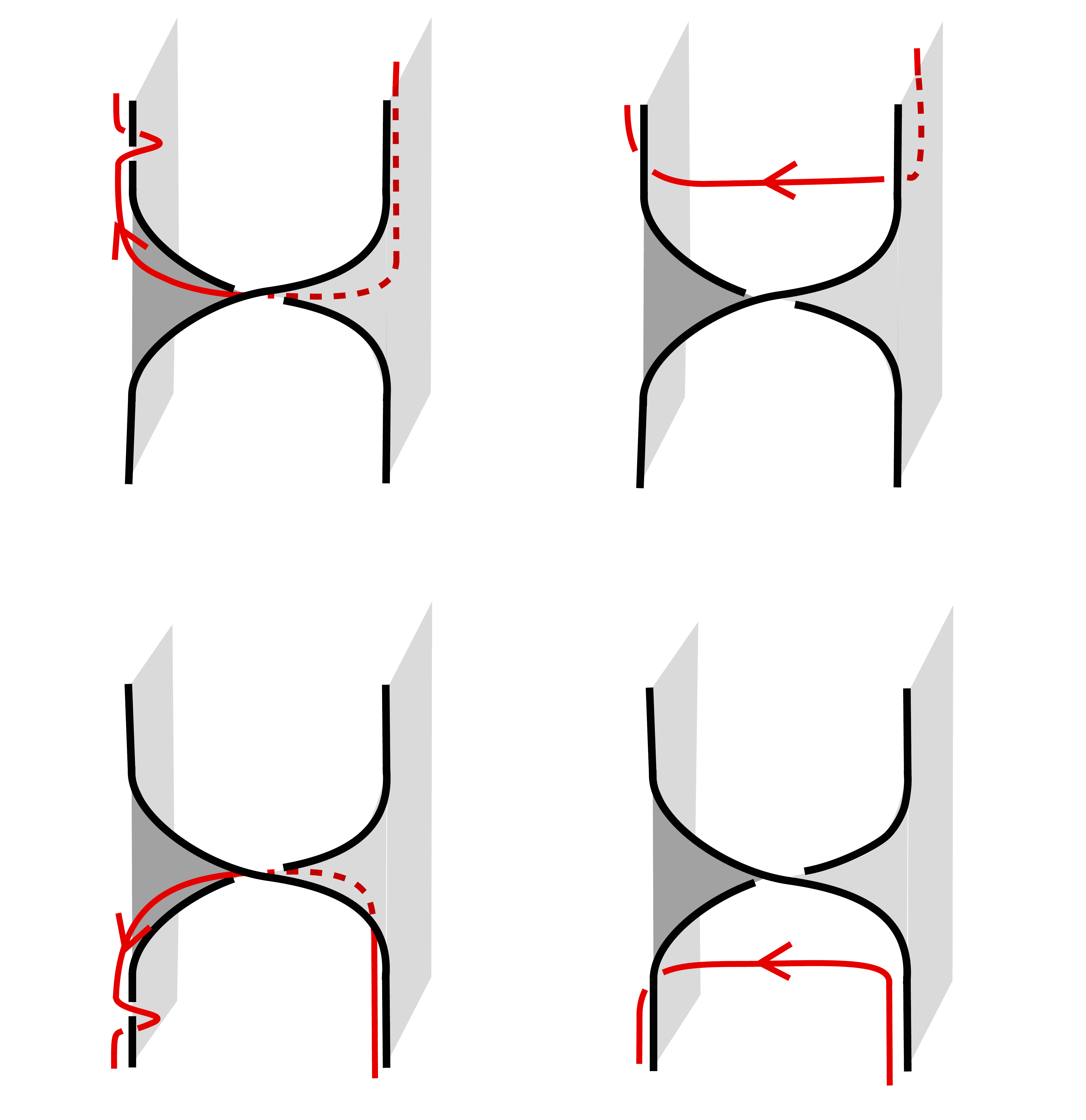}
\caption{a) An underpass $\gamma_2$ interpreted as a curve on $F^+\cup N(B)$ in a neighbourhood of a positive crossing. b) $\gamma_2$ can be deformed to lie below the strands of the braid $\partial F$. c) An underpass $\gamma_2$ interpreted as a curve on $F^+\cup N(B)$ in a neighbourhood of a negative crossing. d) $\gamma_2$ can be deformed to lie below the strands of the braid $\partial F$. \label{fig:underpass}}
\end{figure}

In Figure \ref{fig:overpass}a) the blue curve is part of the overpass $\gamma_1$ on a positive band, while in Figure \ref{fig:overpass}c) the blue curve is part of the overpass $\gamma_1$ on a negative band. If the blue curve is oriented from left to right, then $o(h_i^1)=1$, where $h_i^1$ is the horizontal segment corresponding to the band shown in the picture. If this segment of $\gamma_1$ is oriented from right to left, then $o(h_i^1)=-1$.

The conditions on the orientations $o(v_i^1)$ of the vertical segments $v_i^1$ of the overpass mean that in these figures the curve continues on the leftmost sheet as indicated, e.g. in Figure \ref{fig:overpass}a) the vertical segment of the overpass continues downward if $o(h_i^1)=-1$. In contrast to the underpass $\gamma_2$, the overpass $\gamma_1$ is not allowed to turn around on a sheet of $F$. Hence there are no self-crossings. On the right sheet, on which the band terminates, there are no restrictions on the orientation of the overpass. It could continue upwards or downwards.

The vertical segments of underpasses and overpasses in a ladder diagram that are not near horizontal lines correspond to parts of the corresponding curves that are vertical on the corresponding sheet of $F^+$ and $F$, respectively. In particular, every vertical segment of an underpass that crosses a horizontal line in the ladder diagram lies behind the corresponding band in the surface diagram, which is precisely why this is not allowed for vertical segments of overpasses.

\begin{figure}[h]
\labellist
\Large
\pinlabel a) at 100 900
\pinlabel b) at 1050 900
\pinlabel c) at 100 -150
\pinlabel d) at 1050 -150
\endlabellist
\centering
\includegraphics[height=7cm]{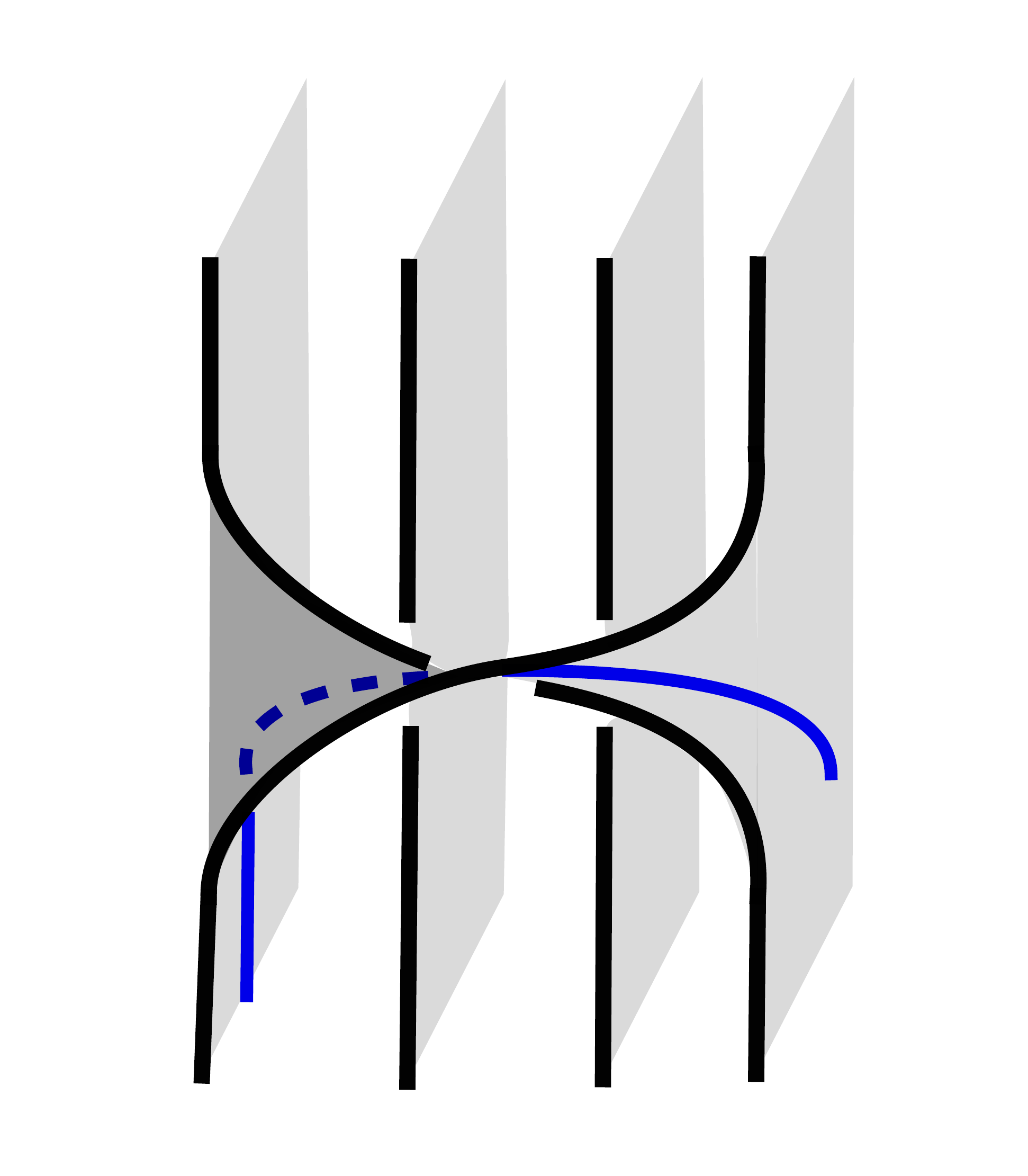}
\includegraphics[height=7cm]{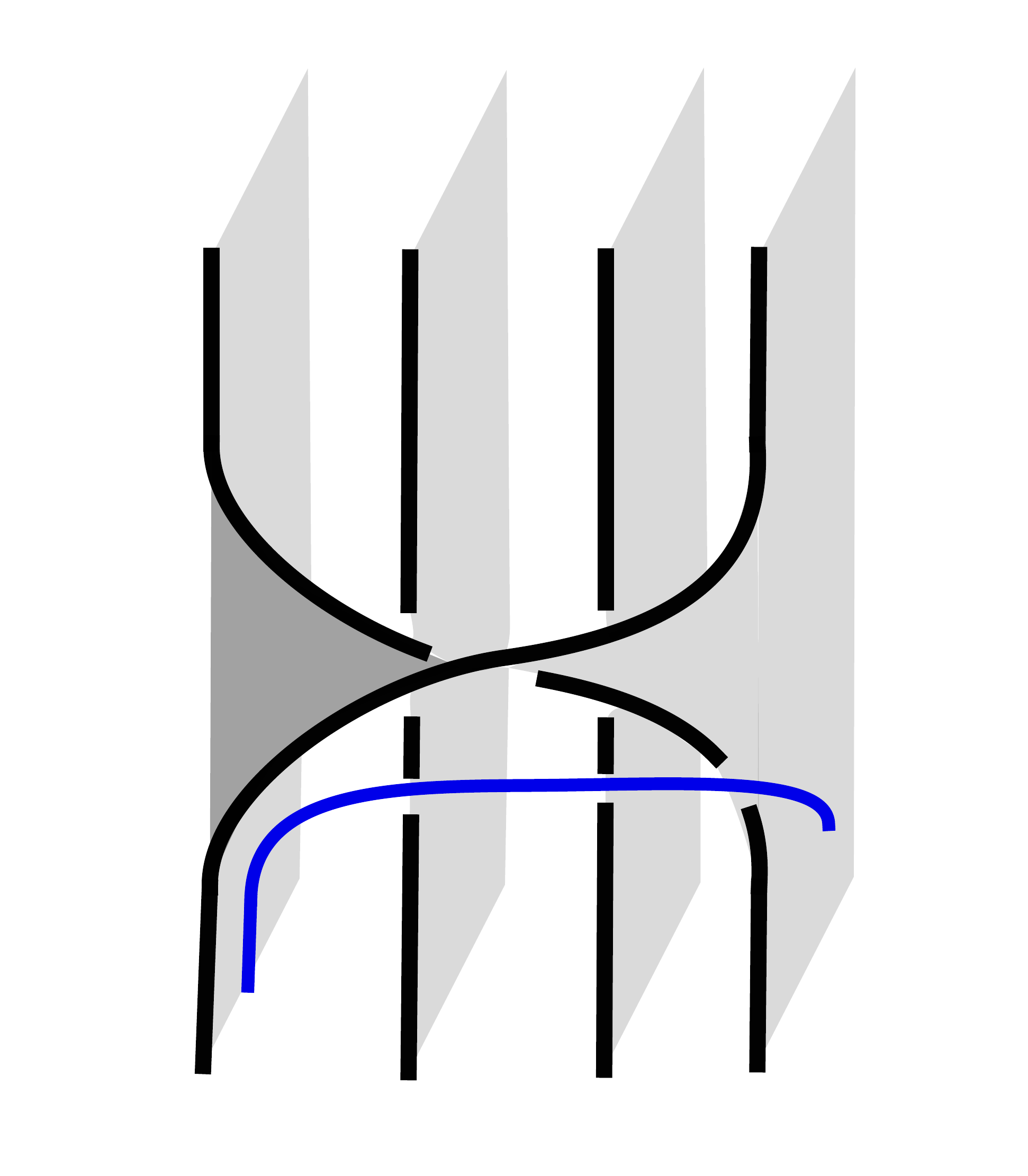}
\includegraphics[height=7cm]{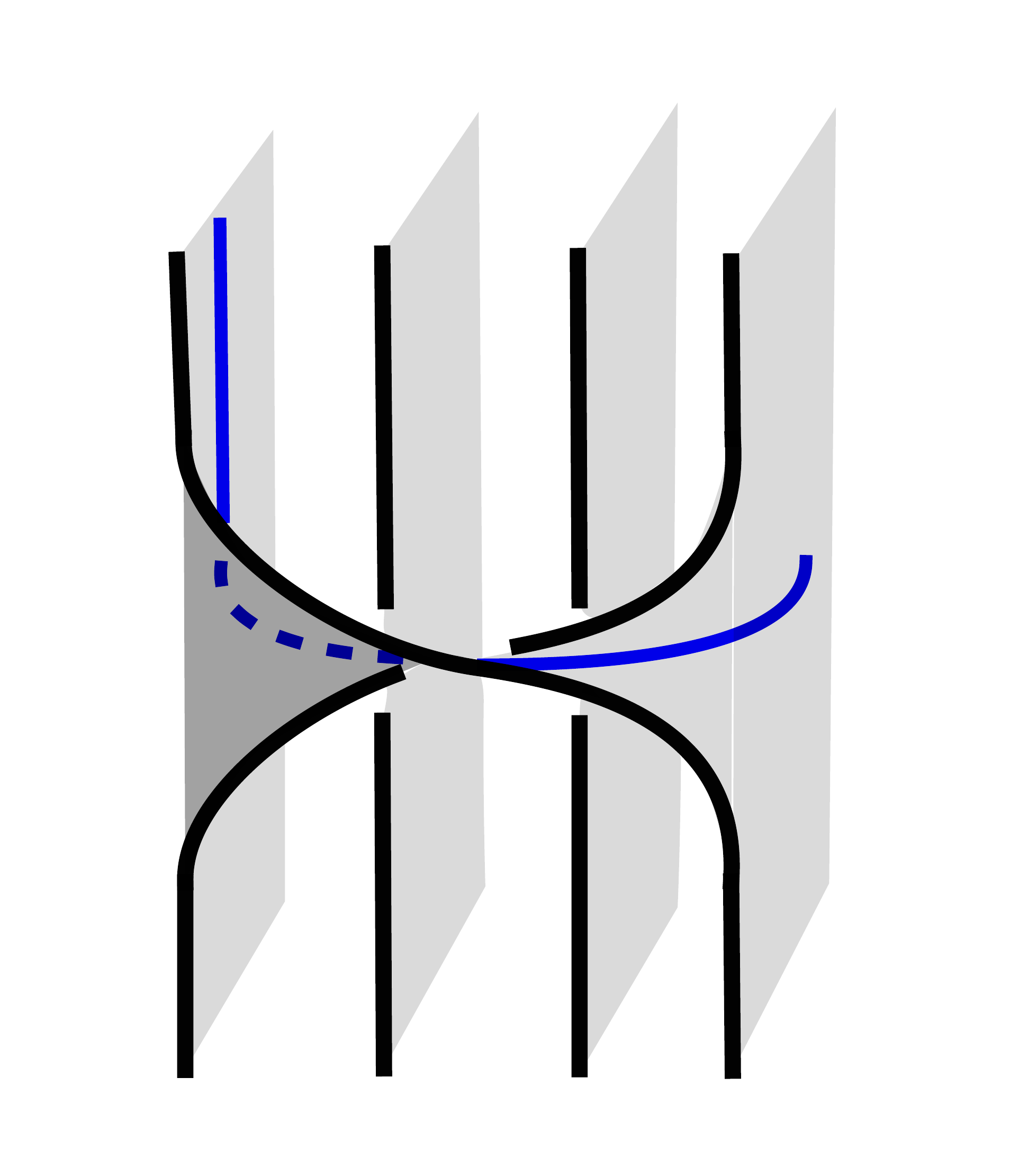}
\includegraphics[height=7cm]{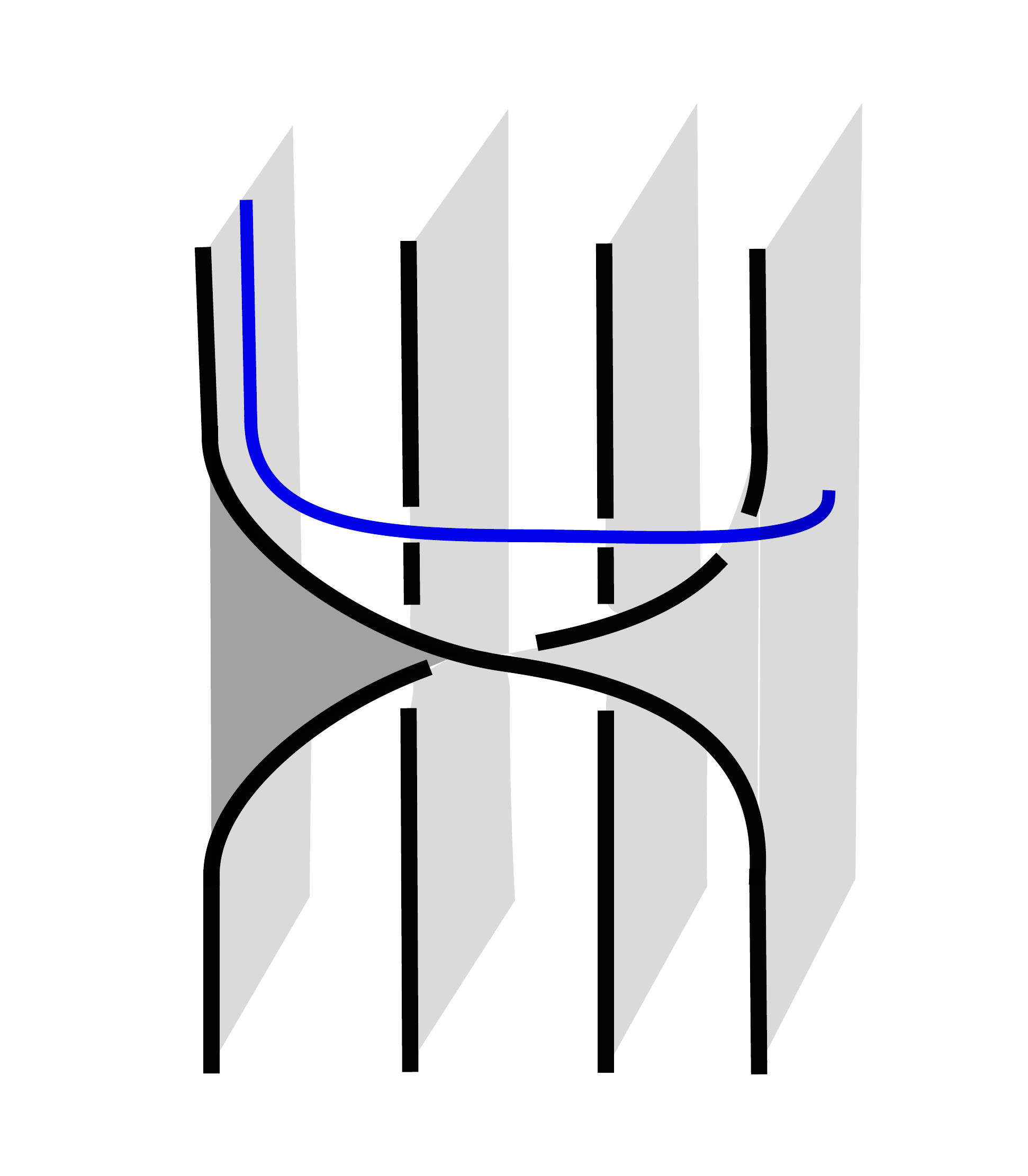}
\caption{a) An overpass $\gamma_1$ interpreted as a curve on $F$ in a neighbourhood of a crossing corresponding to a positive horizontal edge $h$. b) $\gamma_1$ can be deformed to lie above the strands of the braid $\partial F$. c) An overpass $\gamma_1$ interpreted as a curve on $F$ in a neighbourhood of a crossing corresponding to a negative horizontal edge $h$. d) $\gamma_1$ can be deformed to lie above the strands of the braid $\partial F$.\label{fig:overpass}}
\end{figure}

We can assume that the ends of $\gamma_1$ and $\gamma_2$ lie in $F\cup N(B)$ and $F^+\cup N(B)$, respectively, and at the same height in the surface diagram, since their union in the ladder diagram is a loop. Now connect the end point of $\gamma_2$ to the start point of $\gamma_1$ and the end point of $\gamma_1$ to the start point of $\gamma_2$ via arcs in $N(B)$ that are positively transverse to the fibers of any $S^1$-valued map that satisfies the boundary condition of an open book on $N(B)$. 

In this way we obtain one closed loop formed by $\gamma_1\subset F$, $\gamma_2\subset F^+\cup N(B)$ and the two connecting arcs in $N(B)$. Figure \ref{fig:bigexample}a) shows this curve for the example ladder diagram with over- and underpass from Figure \ref{fig:ladders}.

Figures \ref{fig:underpass}b) and d) show how for both types of crossings the underpass can be deformed to lie below the strands of $B$. Note that these isotopies can be performed without changing the unique intersection point of that segment of $\gamma_2$ with $F$ and without introducing new intersection points between $\gamma_2$ and $F$. This means that the isotopy can be performed without intersecting $\gamma_1\subset F$.

By performing such an isotopy in the neighbourhood of every band that corresponds to a horizontal segment of $\gamma_2$ and leaving the vertical segments fixed, we obtain an isotopy (with fixed endpoints) from $\gamma_2$ to a simple path $\gamma_2'$, which starts at the rightmost sheet of $F^+$, ends at the leftmost sheet of $F^+$, and in between lies behind all strands of $B$ without crossing itself.

In Figure \ref{fig:overpass}b) the overpass is isotoped to lie above the strands of $B$. Note that this move is only possible because the red curve comes from the bottom left (if it is oriented from left to right) or ends at the bottom left (if it is oriented from right to left), i.e., $o(v_{i-1})=o(h_i)=1$ or $o(v_i)=o(h_i)=-1$. Whether the overpass goes upwards or downwards on the far-right disk in the figure is not relevant for the isotopy.

Figure \ref{fig:overpass}c) and d) show the analogous isotopy for the case where $sign(h)=-1$.

Performing such an isotopy for every horizontal segment of the overpass $\gamma_1$ results in a path $\gamma_1'$ that starts on the leftmost disk of the surface $F$, ends on the rightmost disk and in between lies above all strands of the braid or is vertical on $F$. If $\gamma_1$ had a vertical segment that crosses a horizontal line in the diagram $\mathcal{D}$ or that contains the starting point of a horizontal line, it would lie below the corresponding band. The definition of an overpass guarantees that this does not happen.

The result of these isotopies for the example ladder diagram from Figure \ref{fig:ladders} is shown in Figure \ref{fig:bigexample}b).


The union of $\gamma_1'$ and $\gamma_2'$ and the two arcs in $N(B)$ that connect their endpoints forms an unknot $O$ that is a braid axis for $B$, since neither $\gamma_1'$ nor $\gamma_2'$ cross themselves, $\gamma_1'$ lies above all strands, and $\gamma_2'$ lies below all strands. 

The unknot $O$ is isotopic in $S^3\backslash L$ to the union of $\gamma_1$, $\gamma_2$ and the two arcs in $N(B)$ that connect their endpoints, which lies in $F\cup F^+\cup N(B)$ and has the property that all parts that lie in $N(B)$ are positively transverse to the open book boundary condition in $N(B)$. Thus for any map $f:S^3\backslash L\to S^1$ that has the required behaviour on $N(B)$ and that has $F$ as a level set, we can isotope $O$ to be positively transverse to all fibers. In particular, if the closure of $B$ is a fibered link, $B$ is a generalised exchangeable braid with braid axis $O$.
\end{proof}

Lemma \ref{lem:over} raises an interesting generalisation of the concept of generalised exchangeable braids. For any link $L$ (not necessarily fibered) we could consider a map $f:S^3\backslash L\to S^1$ that behaves like an open book in a neighbourhood of $L$, but has a finite number of Morse-critical points. The minimal number of critical points of any such map for a given link $L$ is called the Morse-Novikov number $\mathcal{MN}(L)$ of a link $L$ and clearly $\mathcal{MN}(L)=0$ if and only if $L$ is fibered. The obvious generalisation of the question whether all fibered links are closures of generalised exchangeable braids is whether for any link $L$ and any such map $f$ that realises $\mathcal{MN}(L)$, there is a braid axis of $L$ that is positively transverse to all level sets of $f$.

\begin{theorem}
\label{thm:1234}
Let $F$ be a Bennequin surface of degree $n$, whose boundary is a fibered link. If $F$ is represented by a band word that contains the band generators $a_{1,2}$, $a_{2,3}$, $a_{3,4},\ldots, a_{n-1,n}$ and $a_{1,n}$ (anywhere in the braid word, not necessarily consecutively, in any order, with any sign), then the open book with binding $L$ and fiber $F$ can be braided. 
\end{theorem}

\begin{proof}
The fact that the band word contains $a_{1,n}$ (or $a_{1,n}^{-1}$) immediately implies the existence of an overpass, while the occurrence of $a_{i,i+1}$ or $a_{i,i+1}^{-1}$ for all $i\in\{1,2,\ldots,n-1\}$ implies the existence of an underpass. The theorem then follows from Lemma \ref{lem:over}.
\end{proof}

\subsection{Braided open books of degree 3}\label{subsec:3}

The infinite family of inhomogeneous P-fibered braids that was constructed in \cite{bode:adicact} consisted of braids on 3 strands. Using over- and underpasses, we find that in fact all fibered links of braid index at most 3 are closures of P-fibered braids on at most 3 strands.


\begin{lemma}
\label{lem:subsub}
Let $A\in\mathbb{B}_n$ be a braid word in band generators. If $A$ is represented by a ladder diagram with and overpass and an underpass, then $AB$ and $BA$ can be represented by ladder diagrams with an overpass and an underpass for any braid $B\in\mathbb{B}_n$.
\end{lemma}
\begin{proof}
The ladder diagram for $AB$ consists of the ladder diagram for $B$ on top of the ladder diagram for $A$. Hence, if the ladder diagram for $A$ allows an overpass $\gamma_1$ and an underpass $\gamma_2$, the same curves $\gamma_1$ and $\gamma_2$ are an overpass and an underpass for the ladder diagram for $AB$. The same argument holds for $BA$, whose ladder diagram is the ladder diagram for $A$ on top of the ladder diagram for $B$.
\end{proof}

\begin{lemma}
\label{lem:3}
Let $L$ be a link with braid index at most 3. Then for $L$ or its mirror image any of its minimal genus Seifert surfaces can be made into a braided surface of degree 3 with a ladder diagram that has an overpass and an underpass.
\end{lemma}

\begin{proof}
If the braid index of $L$ is strictly less than 3, the statement is obvious. In the following we assume that the braid index of $L$ is equal to 3.

It was shown by Bennequin that for links of braid index 3 any minimal genus Seifert surface $F$ can be made into a braided surface of degree 3 \cite{benn}. If the corresponding band word contains all of $a_{1,2}$, $a_{2,3}$ and $a_{1,3}$ with some signs, then by Theorem \ref{thm:1234} $F$ can be braided. We can thus assume that the band word does not contain all three generators (modulo signs).

We claim that $F$ or its mirror image can be represented by a band word $w$ in band generators that contains $a_{1,2}^{\varepsilon}a_{2,3}$ for some $\varepsilon\in{\pm1}$. This proves the lemma for $F$ or its mirror image by Lemma \ref{lem:subsub}, since an overpass and an underpass can be drawn in the ladder diagram for $a_{1,2}^{\varepsilon}a_{2,3}$.

Let $w=\prod_{k=1}^{\ell}a_{i_k,j_k}^{\varepsilon_k}$ be the band word in the three band generators $a_{1,2}$, $a_{1,3}$, $a_{2,3}$ and their inverses that represents the braided surface $F$. Note that $w'=\prod_{k=1}^{\ell}a_{i_k-1,j_k-1}^{\varepsilon_k}$ and $w''=\prod_{k=1}^{\ell}a_{i_k-2,j_k-2}^{\varepsilon_k}$ with indices taken modulo 3 also represent $F$ as a braided surface. (Recall that $a_{j,i}=a_{i,j}$.)

Since $F$ is connected, $w$, $w'$ and $w''$ all must contain at least two of the three generators (with some possibly negative exponent). Since we assume that none of these band words contains all three generators, one of $w$, $w'$ and $w''$ consists only of $a_{1,2}$, $a_{2,3}$ and their inverses. This implies that (possibly after a conjugation) it contains $a_{1,2}^{\varepsilon_1}a_{2,3}^{\varepsilon_2}$ for some $\varepsilon_1,\varepsilon_2\in\{\pm1\}$.


Note that $a_{1,2}$, $a_{2,3}$ and their inverses are precisely the Artin generators and their inverses, so that the mirror image of $F$ is represented by the same band word except that all exponents are multiplied by $-1$. Therefore, $F$ or its mirror image can be represented by a word in band generators that contains $a_{1,2}^{\varepsilon}a_{2,3}$ for some $\varepsilon\in\{\pm 1\}$.
\end{proof}

The proof of Theorem \ref{thm:3} is now simply a combination of Lemma \ref{lem:over} and Lemma \ref{lem:3}.

\begin{theorem}
\label{thm:3f}
Let $(L,\Psi)$ be an open book in $S^3$ whose binding $L$ has a braid index of at most 3. Then $L$ is the closure of a generalised exchangeable braid.
\end{theorem}

\begin{proof}
Lemma \ref{lem:over} and Lemma \ref{lem:3} show that $L$ or its mirror image is the closure of a generalised exchangeable braid. By Theorem \ref{thm:main} this implies that $L$ or its mirror image is the closure of a P-fibered braid. 

If $B$ is a P-fibered braid, parametrised by $\cup_i(z_{i}(t),t)\subset\mathbb{C}\times[0,2\pi]$, then its mirror image is parametrised by the complex conjugates $\cup_i(\overline{z}_{i}(t),t)$. If $g_t$ denotes the loop of polynomials corresponding to $B$, then $\overline{g_t}$, the loop of polynomials where all coefficients of $g_t$ have been complex conjugated, corresponds to the mirror image of $B$. Since $\arg \overline{g_t}(\overline{z})=-\arg g_t(z)$, it follows that the mirror image of $B$ is also P-fibered.

Therefore, both $L$ and its mirror image are closures of P-fibered braids and by Theorem \ref{thm:main} closures of generalised exchangeable braids.
\end{proof}

In the proof of the lemmata leading up to Theorem \ref{thm:3f} we have used that for links of braid index 3 every minimal genus Seifert surface can be turned into a braided surface of degree 3, which was proved by Bennequin. It is known that the analogous statement is not true for links of braid index 4 \cite{mikami}. Of course, the corresponding Seifert surfaces can still be braided, but they could require a higher number of strands. 

Combining Theorem \ref{thm:3f} with our main Theorem \ref{thm:main}, Theorem \ref{thm:ralg} and the results from \cite{morton} we immediately obtain the following corollaries.

\begin{corollary}
Let $L$ be a fibered link with braid index at most 3. Then $L$ is the closure of some braid $B$, such that the closure of $B^2$ is real algebraic.
\end{corollary}

\begin{corollary}
Let $L$ be a fibered link with braid index at most 3 and $F$ be its fiber surface. Then there exists a simple branched cover $\pi:S^3\to S^3$, branched over a link $L_{branch}$, such that $F$ can be obtained from a disk via a sequence of successive $\pi$-symmetric Hopf plumbings and deplumbings.
\end{corollary}

\subsection{Concluding remarks}
We have shown that the four different definitions of a braided open book in $S^3$ are all equivalent. It is still an open question if every fibered link is the binding of a braided open book. However, the equivalence between the different definitions allows us to translate results about one type of braiding to another. This way we can for example prove that a braid is P-fibered if it is generalised exchangeable and vice versa. Examples of this occur for example in Corollary \ref{cor:thomo} and Corollary \ref{cor:thomo2}. It also results in new techniques that allow us to prove that certain open books can be braided, compare Section \ref{subsec:3}.

The connections to the Benedetti-Shiota conjecture on real algebraic links (Conjecture \ref{conj:bene}) and Montesino's and Morton's stronger version of Harer's conjecture (Conjecture \ref{con:morton}) make the concept of a braided open book a promising tool for making progress on these open problems. In fact, more recently Rampichini diagrams have already been successfully used to prove that a large family of P-fibered braids (including all T-homogeneous braids) closes to real algebraic links \cite{bode:thomo}.

Since some of the definitions of a braided open book have natural generalizations to 3-manifolds that are not necessarily $S^3$, future research should also investigate these structures in this more general setting.







\end{document}